\documentclass[a4paper,oneside,fleqn,10pt]{article}
\title{Homological invariants relating the super Jordan plane to the Virasoro algebra}
\author{Sebasti\'an Reca and Andrea Solotar
\thanks{{\footnotesize This work has been supported by the projects  
UBACYT 20020130100533BA, PIP-CONICET 11220150100483CO, MATHAMSUD-REPHOMOL and PICT 2015-0366.
The first named author is a CONICET fellow. The second named author is a research member of CONICET (Argentina) and a Senior Associate of ICTP Associate
Scheme.}}
}

\date{}

\usepackage{amsmath}
\usepackage{amsthm}
\usepackage{amsfonts}
\usepackage{amssymb}
\usepackage{mathrsfs}
\usepackage{mathtools}
\usepackage[shortlabels]{enumitem}
\usepackage{stmaryrd}
\usepackage{braket}
\usepackage{theoremref}
\usepackage{subfiles}
\usepackage{xcolor}
\usepackage{xr}
\usepackage{fullpage}
\usepackage[colorlinks,citecolor=red!70!black,linkcolor=blue!70,anchorcolor=pink]{hyperref}
\usepackage{chngcntr}
\usepackage{multicol}
\numberwithin{equation}{section}
\usepackage[utf8]{inputenc}
\usepackage[T1]{fontenc}
\usepackage[osf,noBBpl]{mathpazo}
\usepackage{euler}

\usepackage[all]{xy}
\usepackage{graphics}

\newcommand\NN{\mathbb{N}}
\newcommand\ZZ{\mathbb{Z}}

\newcommand\CC{\mathbb{C}}

\newcommand\ox{\otimes}
\newcommand\field{\Bbbk}

\newcommand{\ov}[1]{\overline{#1}}

\DeclareMathOperator\Hom{Hom}
\DeclareMathOperator\yoneda{E}
\DeclareMathOperator\Ima{Im}
\DeclareMathOperator\Ker{Ker}
\DeclareMathOperator\gr{gr}

\DeclareMathOperator\Ext{Ext}

\DeclareMathOperator\Hy{H}
\DeclareMathOperator\Vir{Vir}
\DeclareMathOperator\GK{GK}
\DeclareMathOperator\ad{ad}
\DeclareMathOperator\rad{rad}
\DeclareMathOperator\degree{deg}
\DeclareMathOperator\car{char}

\newtheorem{theorem}{Theorem}[section]
\newtheorem{proposition}[theorem]{Proposition}
\newtheorem{lemma}[theorem]{Lemma}
\newtheorem{corollary}[theorem]{Corollary}
\newtheorem{remark}[theorem]{Remark}
\newtheorem{thmx}{Theorem}

\begin{document}
\maketitle

\begin{abstract}
	Nichols algebras are an important tool for the classification  of Hopf algebras. Within those
	with finite $\GK$ dimension, we study homological invariants of the super Jordan plane, that is,
	the Nichols algebra $A = \mathfrak{B}(V(-1,2))$. These invariants are Hochschild homology, the Hochschild
	cohomology algebra, the Lie structure of the first cohomology space - which is a Lie subalgebra
	of the Virasoro algebra - and its representations $\Hy^{n}(A, A)$ and also the Yoneda algebra.
	We prove that  the algebra $A$ is $\mathcal{K}_2$. Moreover, we prove that the Yoneda algebra of the bosonization  $A\#\field\ZZ$
	of $A$ is also finitely generated, but not $\mathcal{K}_2$.
\end{abstract}



\section{Introduction}
\label{introduccion}
\addcontentsline{toc}{chapter}{\nameref{introduccion}}
Homological methods provide important information about the structure of associative algebras, revealing
sometimes hidden connections amongst them. In this work, our main task is to compute
the Hochschild homology and cohomology of the super Jordan plane $A$, as well as the cup
product in cohomology, and the Lie structure of $\Hy^1(A,A)$ with the Gerstenhaber bracket,
and describe higher cohomology spaces as $\Hy^1(A,A)$-Lie modules. We also describe the Yoneda
algebra $\yoneda(A)$ and the Yoneda algebra of the bosonization $A\#\field\ZZ$, proving that both of them are finitely generated as algebras. 
Roughly speaking, besides the explicit information about cohomology,
we obtain three main results:
\begin{thmx}
	The Gerstenhaber structure endows $\Hy^{1}(A,A)$ with a Lie algebra structure such that
	it is isomorphic to the Cartan subalgebra plus the positive part of the Virasoro algebra.
	We identify intermediate series modules within the representation $\Hy^{n}(A,A)$.
\end{thmx}
\begin{thmx}
	The Yoneda algebra $\yoneda(A)$ is generated as an algebra over $\field$ by three elements:
	two in degree one and one in degree two, subject to relations that we make explicit. As a consequence
	$A$ is $\mathcal{K}_2$.
\end{thmx}
\begin{thmx}
	The Yoneda algebra $\yoneda(A\#\field\ZZ)$ is generated as an algebra over $\field$ by three elements:
	one in degree one, one in degree two and one in degree three, subject to relations that we make explicit. As a consequence
	$A\#\field\ZZ$ is not $\mathcal{K}_2$.
\end{thmx}

As a corollary of this last result, the cohomology algebra $\Hy^{\bullet}(H,\field)$, where $H$ is a Hopf algebra whose associated 
graded algebra with respect to the coradical fitration is $A\#\field\ZZ$ -- see \cite{AAH2}-- is also finitely generated. Even if this
Hopf algebra is not finite dimensional, its Gelfand-Kirillov dimension is finite, and our result is in line with 
\cite{GK, MPSW, StVa}, all of them related to a conjecture by Etingof and Ostrik \cite{EO}. Notice that in our case
the braiding group is infinite and the braiding is not diagonal.

We now describe in more detail the objects involved, and the contents of this article. The super
Jordan plane is the $\field$-algebra generated by two elements $x$ and $y$, subject to
two relations, one monomial quadratic and one cubic:
\[
	\mathfrak{B}(V(-1,2)) \cong \field\langle x, y \rangle / (x^2, y^2x - xy^2 - xyx).
\]
As pointed in \cite{AAH}, the second
relation is equivalent to the quantum Serre relation $(\ad_c y)^2x - x(ad_c y)x$.
Our main motivation to study this algebra comes from the fact that it is a Nichols
algebra of finite Gelfand-Kirillov dimension appearing in the classification
of pointed Hopf algebras, see \cite{AAH}. One of the main tools used is the notion
of braided vector space. There are still many open problems when the braiding is not of diagonal type.
This is the case of the super Jordan plane, where the vector space $V(-1, 2) = \field x \oplus \field y$
is a Yetter-Drinfeld module over $\field \ZZ$ with braiding $c$ as follows
\begin{align*}
	c\left(x \ox x\right) &= - x \ox x,\quad c\left(x \ox y\right) = x \ox x - y \ox x,\\
	c\left(y \ox x\right) &= - x \ox y,\quad c\left(y \ox y\right) = x \ox y - y \ox y. 
\end{align*}
The Gelfand-Kirillov dimension of the super Jordan plane is $2$,  the set
$\{x^a(yx)^by^c \mid a \in \{0, 1\}, b, c \in \NN_0 \}$ is a PBW basis, and clearly it is not a domain,
see \cite{AAH}. This algebra is a test case for Nichols algebras corresponding to non-diagonal
braidings: here the described braiding corresponds to a $2\times 2$ Jordan block with eigenvalue $-1$.

On the other hand, the Virasoro algebra, that we will denote $\Vir$, is a very important Lie
algebra both in mathematics and in physics, for example in string theory and conformal field theory
and its representations have been extensively - but not completely - studied (\cite{FF}, \cite{Z}, \cite{MZ}).
It is the unique - up to isomorphisms - one dimensional central extension of the Witt algebra, and it is generated
by a family $\left\lbrace L_n \right\rbrace_{n \in \ZZ}$ together with a central charge $C$, subject to the
relations
\begin{align*}
	\left[L_m, L_n\right] = (n - m)L_{m + n} + \delta_{m, -n}\frac{m^3 - m}{12}C, \text{ for all } n, m \in \ZZ.
\end{align*}
Let us write the triangular decomposition as $\Vir = \Vir^{+} \oplus \mathfrak{h} \oplus \Vir^{-}$.

We compute Hochschild homology and cohomology of $A$ using
the minimal resolution obtained with methods from \cite{CS}. Our results
are explicit enough to afford computations necessary to describe $\Hy^{\bullet}(A, A)$
as a graded commutative algebra. The most interesting result is, in our opinion,
the characterization of the Lie algebra $\Hy^{1}(A, A)$ - with the Gerstenhaber bracket -
as a subalgebra of the Virasoro algebra. Since we also compute the Gerstenhaber brackets
$\left[\Hy^{1}(A, A), \Hy^{n}(A, A) \right]$ for all $n \geq 2$, we obtain induced representations
of the Virasoro algebra. We are also very explicit concerning the computations of Hochschild homology
because it will be useful for a description of the cap product.

Next, we compute the Yoneda algebra of $A$ and as a consequence
we prove that, even if $A$ is not $N$-Koszul, it is $\mathcal{K}_2$ \cite{CaSh}. Finally, 
we compute the Yoneda algebra of $A\#\field\ZZ$, and we conclude the finite generation of the aforementioned Hopf algebra $H$.

We do not know yet if there is an intrinsic relation between the Virasoro algebra and the
super Jordan plane, but we hope to made this clear in subsequent work.

The contents of the article are as follows. In Section \ref{resolution} we construct the minimal
projective resolution of $A$ as bimodule over itself, using an adequate reduction system. This
resolution can be nicely organized as a total complex of a bicomplex, emphasizing the
role played by the monomial relation, which is responsible for the infinity of the
global dimension of $A$.

In Section \ref{cohomology} we compute the Hochschild cohomology of $A$. The very detailed and
exhaustive description is useful for studying the cup product and the Gerstenhaber bracket. Our
method includes computing the spectral sequence associated to filtering the bicomplex by columns.
In Theorem \ref{cohomology_theorem} we exhibit infinite bases for each space $\Hy^i(A,A)$, with $i > 0$,
and prove that the center of $A$ is $\field$.

In Section \ref{homology}, we use similar methods to describe the Hochschild homology of $A$, see Theorem
\ref{homology_theorem}.

Section \ref{product} is devoted to the description of the graded commutative algebra $\Hy^{\bullet}(A,A)$,
with the cup product. This description will also be important for the Gerstenhaber bracket, which
is a graded derivation with respect to the cup product in both variables. We proceed
by constructing - sometimes partially - comparison maps between the bar resolution and ours. The
results are summarized in Theorem \ref{product_theorem_table}. We prove in particular that this algebra
is not finitely generated, not even modulo nilpotents. Moreover, there is a an element $u$ of $\Hy^2(A,A)$
such that the cup product with $u$ is an isomorphism between $\Hy^{2p}(A,A)$ and $\Hy^{2(p + 1)}(A,A)$,
for all $p \geq 1$, and similarly for the odd cohomology spaces.

The main part of the article is contained in Sections \ref{derivations}, \ref{representations}, \ref{yoneda} and \ref{smash_product}. 
In Section \ref{derivations} we achieve the description of $\Hy^1(A,A)$ as Lie algebra - see Theorem \ref{derivations_theorem} -
while Section \ref{representations} is dedicated to the detailed description of the Lie module structure of $\Hy^n(A,A)$,
for all $n \geq 2$, and to identify some irreducible submodules. The main purpose of Sections \ref{yoneda} and \ref{smash_product}
is to describe the Yoneda algebras of $A$ and $A\#\field\ZZ$, respectively. In the former we prove that $\yoneda(A)$ is finitely
generated and $\mathcal{K}_2$, while in the latter we study $\yoneda(A\#\field\ZZ)$. In both cases
we find generators and relations to present this algebra.

In this article $\field$ will be an algebraically closed field of characteristic zero. This
last hypothesis is necessary for our computations, whereas the first one is not, but it comes
from the context of Hopf algebra classification.

We would like to thank Mariano Su\'arez-\'Alvarez for discussions about this work, in particular about the last section. We also thank Nicol\'as Andruskiewitsch for interesting comments 
about this subject, and Iv\'an Angiono and Leandro Vendram\'\i n for a careful reading of Sebasti\'an Reca's degree thesis.
\section{Projective resolution}
\label{resolution}
\addcontentsline{toc}{chapter}{\nameref{resolution}}
From now on $A$ will denote the super Jordan plane, that is,
the algebra $\field\langle x, y\rangle/ \left(x^2, y^2x - xy^2 -xyx\right)$.
It appears as the Nichols algebra $\mathfrak{B}(V(-1,2))$, where $\left(V(-1,2), c\right)$
is a braided $\field$-vector space of dimension $2$ not of diagonal type. Here the braiding group
is $\ZZ$ and $c: V \ox V \to V \ox V$ is defined by:
\begin{align*}
	c\left(x \ox x\right) &= - x \ox x,\quad c\left(x \ox y\right) = x \ox x - y \ox x,\\
	c\left(y \ox x\right) &= - x \ox y,\quad c\left(y \ox y\right) = x \ox y - y \ox y. 
\end{align*}
The super Jordan plane has been studied in \cite{AAH}, \cite{AG}. It is graded with 
$\degree(x) = \degree(y) = 1$ and it has a Poincaré-Birkhoff-Witt basis
$\mathcal{B} = \{x^a(yx)^by^c \mid a \in \{0, 1\}, b, c \in \NN_0 \}$. It is a noetherian algebra
of Gelfand-Kirillov dimension $2$, whose global dimension is infinite. There are lots of open questions concerning
Nichols algebras appearing in the classification of Hopf algebras. 

Since our aim is to describe
completely the Hochschild cohomology of $A$ as an associative algebra with the cup product and as a graded
Lie algebra with the Gerstenhaber bracket, we will start by constructing its minimal resolution as
bimodule over itself.

Our first lemma describes the commutation rules in $A$. The proof is inductive and,
since no difficulty arises, we will not include it. It is written in detail in \cite{R}.
\begin{lemma}
For all $n \in \NN_{0}$ and $b \in \NN$, the following equalities hold,
\begin{align}
	y^{2n}x &= \sum_{i = 0}^n\frac{n!}{i!}x(yx)^{n - i}y^{2i},\\
	y^{2n + 1}x &= \sum_{i = 0}^n \frac{n!}{i!}(yx)^{n - i + 1}y^{2i},\\
	y^{2n}(yx)^b &= \sum_{i = 0}^{n}\binom{n}{i} \frac{(b + n - i - 1)!}{(b - 1)!} (yx)^{b + n - i}y^{2i},\\
    y^{2n + 1}(yx)^{b} &= \sum_{i = 0}^{n + 1}\binom{n + 1}{i}\frac{(b + n - i)!}{(b - 1)!}x(yx)^{b + n - i}y^{2i}.
\end{align}
\end{lemma}

We will use the methods and terminology of \cite{CS} to construct the minimal resolution. For this, we will consider
an order in the set of words in $x$ and $y$ such that $y > x$, compatible with concatenation, and the
induced reduction system $\mathfrak{R} = \lbrace (x^2, 0), (y^2x, xy^2 + xyx) \rbrace$. It is clear
that for all $(s, f) \in \mathfrak{R}$, $f$ is irreducible. Although we already know that the ambiguities can be solved,
since it is proved in \cite{AAH} that $\mathcal{B}$ is a PBW basis of $A$, we will verify it anyway because the process
will provide some of the differentials appearing in the resolution as well as the comparison maps that will be needed afterwards.
The set of ambiguities is $\mathcal{A}_2 = \{x^3, y^2x^2\}$. Since $x^3$ comes from a monomial relation, there
is nothing to prove for it. For $y^2x^2$, the possible two paths are the following.
\begin{align*}
\xymatrix{
	& (y^2x)x \ar[r]& (xy^2 + xyx) x \ar@{=}[r]& xy^2x + xyx^2 \ar[r] & x(xy^2 + xyx) \ar[dr]&\\
	y^2x^2 \ar@{=}[ur] \ar@{=}[dr]& & & & & 0\\
	& y^2(x^2) \ar[r]& y^2 0 \ar@{=}[rr]& & 0 \ar@{=}[ur] &
}
\end{align*}
The set of $n$-ambiguities for $n \geq 3$ is $\mathcal{A}_n = \{x^{n+ 1}, y^2x^n\}$. 
Let $\mathcal{A}_0$, $\mathcal{A}_1$ denote the sets $\{x, y\}$, $\{x^2, y^2x\}$, respectively.
The associated monomial algebra, taking into account the order in the set of words, is
$A_{mon} = \field\langle x, y \rangle/\left(x^2, y^2x\right)$. For a monomial algebra, the minimal
resolution as bimodule over itself is Bardzell's resolution \cite{B}, that we will denote
$\left(P^{mon}_{\bullet}, \delta_{\bullet}\right)$. Note that the successive modules appearing in this
resolution are $A_{mon}\ox \field \mathcal{A}_n^{mon} \ox A_{mon}$,
but $\mathcal{A}^{mon}_n = \mathcal{A}_n$ for all $n$, because the $n$-ambiguities depend on the
leading terms of the relations. Moreover, each $\delta_n$ can be extended to $A\ox\field \mathcal{A}_n \ox A$
as a map of $A$-bimodules. Also, the order defined previously induces an order in the monomials of $A$, which in turn defines an order in the monomials of 
$A\ox\field \mathcal{A}_n \ox A$ for each $n$.
The bimodules $A\ox\field \mathcal{A}_n \ox A$ are free, hence projective.
We consider the following sequence of $A^{e}$-modules:
\begin{align}
\xymatrix@=1.5em{
	P_{\bullet}A: \cdots \ar[r]^(.4){d_{3}} & A \ox \field \mathcal{A}_2 \ox A \ar[r]^{d_2}
		& A \ox \field \mathcal{A}_1 \ox A \ar[r]^{d_1} \ar[r]
		& A \ox \field \mathcal{A}_0 \ox A \ar[r]^(.6){d_0} & A\ox A \ar[r] & 0
} \label{projective_resolution}
\end{align}
where the morphisms between the $A^e$-bimodules are given by
\begin{align*}
	d_0(1\ox x \ox 1) &= x \ox 1 - 1 \ox x,\\
	d_0(1 \ox y \ox 1) &= y \ox 1 - 1\ox y,
\end{align*}
\begin{align*}
	d_1(1\ox x^2 \ox 1) &= x \ox x \ox 1 + 1 \ox x \ox x,\\
	d_1(1\ox y^2x \ox 1) &= y^2 \ox x \ox 1 + y \ox y \ox x  + 1 \ox y \ox yx\\
	&\qquad - \left(xy \ox y \ox 1 + x \ox y \ox y + 1 \ox x \ox y^2\right)\\
	&\qquad - \left(xy \ox x \ox 1 + x \ox y \ox x + 1 \ox x \ox yx\right),
\end{align*}
\begin{align*}
	d_n(1\ox x^{n + 1} \ox 1) &= x \ox x^{n} \ox 1 + (-1)^{n + 1} \ox x^{n} \ox x,\\
	d_n(1\ox y^2x^{n} \ox 1) &= y^2 \ox x^{n} \ox 1 +  (-1)^{n + 1} \ox y^2x^{n - 1} \ox x\\
	&\qquad - \left(x\ox y^{2}x^{n - 1} \ox 1 + xy \ox x^{n} \ox 1 + 1 \ox x^{n} \ox y^2 + 1 \ox x^{n}\ox yx\right)
\end{align*}
for each $n \geq 2$.

According to \cite[Thm. 4.1]{CS}, to prove that $\left(P_{\bullet}A, d_{\bullet}\right)$ is a projective resolution
of $A$ as $A^e$-bimodule we only have to verify that
\begin{itemize}
	\item $\left(P_{\bullet}A, d_{\bullet}\right)$ is a complex such that $\mu d_0 = 0$, where $\mu: A\ox A \to A$ is the multiplication of $A$.
	\item For all $q \in \field \mathcal{A}_n$, $(d_n - \delta_n)(1 \ox q \ox 1)$ is a
	sum of terms in $A\ox \field \mathcal{A}_{n - 1} \ox A$ that are smaller than $1 \ox q \ox 1$.
\end{itemize}
Both items can be checked by straightforward computations. Note that the resolution is minimal since
the image of $d_n$ is contained in $\rad\left(A \ox\field \mathcal{A}_{n - 1} \ox A\right)$, for all $n \in \NN$.

A careful look at this resolution shows that it is the total complex of a double complex $X_{\bullet, \bullet}$ with $A\ox A$ at $(0,0)$,
two more columns and infinite rows. The rows correspond to the commutation relation and the columns correspond
to the relation $x^2 = 0$, which is responsible of the global dimension being infinite. The double complex is
\begin{align}
\xymatrix{
	& & \\
	& A \ox \field\{x^4\} \ox A \ar[d]^{\delta} \ar@{<--}[u]_{\partial} & A \ox \field\{y^2x^4\} \ox A \ar[l]_d \ar[d]^{\delta'} \ar@{<--}[u]_{\partial'} \\
	& A \ox \field\{x^3\} \ox A \ar[d]^{\partial} & A \ox \field\{y^2x^3\} \ox A \ar[l]_d \ar[d]^{\partial'} \\
	& A \ox \field\{x^2\} \ox A \ar[d]^{\delta} & A \ox \field\{y^2x^2\} \ox A \ar[l]_d \ar[d]^{\delta'} \\
	A \ox A & A \ox \field\{x, y\} \ox A \ar[l]_(.6){d_0} & A \ox \field \{y^2x\} \ox A \ar[l]_{d_1} \\
}
\end{align}
where, $d_0$ and $d_1$ have already been defined,
\begin{align}
	d(1 \ox y^2x^n \ox 1) &= y^2 \ox x^n \ox 1 - xy \ox x^n \ox 1 - 1 \ox x^n \ox y^2 - 1 \ox x^n \ox yx, \quad n\geq 2,\\
	\delta(1 \ox x^n \ox 1) &= x \ox x^{n -1} \ox 1 + 1 \ox x^{n - 1} \ox x,\\
	\delta'(1 \ox y^2x^n \ox 1) &= - (x \ox y^2 x^{n -1} \ox 1 + 1 \ox y^2x^{n - 1} \ox x), \quad n \geq 2,\\
	\partial(1 \ox x^n \ox 1) &= x \ox x^{n -1} \ox 1 - 1 \ox x^{n - 1} \ox x,\\
	\partial'(1 \ox y^2x^n \ox 1) &= -(x \ox y^2 x^{n -1} \ox 1 - 1 \ox y^2x^{n - 1} \ox x), \quad n \geq 3.
\end{align}
Note that the rows are finite since the cubic relation generates no self ambiguity.

The differentials in the first column correspond to the minimal projective resolution of $\field\left[x\right]/(x^2)$,
while those of the second column differ in a factor $y^2$ appearing on the left, that can be thought as "indexing" the column, and
a factor $-1$, so that we obtain a bicomplex. At this stage, we can already deduce just looking at the resolution
that both $\Hy^{\bullet}(A,A)$ and $\Hy_{\bullet}(A,A)$ will be periodic of period 2, starting at $\Hy^{3}(A,A)$ and $\Hy_{3}(A,A)$, respectively.
\section{Hochschild cohomology}
\label{cohomology}
\addcontentsline{toc}{chapter}{\nameref{cohomology}}
Applying the functor $\Hom_{A^e}(-, A)$ to the double complex $(X_{\bullet, \bullet})$ we obtain the following bicomplex
such that the homology of its total complex is $\Hy^{\bullet}(A,A)$:
\begin{align*}
\xymatrix{
	& & \\
	& A \ar[u]^{\partial} \ar[r]^{d} & A \ar[u]^{-\partial} \\
    & A \ar[u]^{\delta} \ar[r]^{d} & A \ar[u]^{-\delta} \\
    & A \ar[u]^{\partial} \ar[r]^{d} & A \ar[u]^{-\partial} \\
    A \ar[r]^{d^{0}} &A \oplus A \ar[r]^{d^{1}} \ar[u]^{\hat{\delta}} & A \ar[u]^{-\delta}
} 
\end{align*}
with differentials
\begin{align*}
    d^{0}(a) &= \left([x, a], [y,a]\right), &\hat{\delta}(a, b) &= xa + ax,\\
    d^{1}(a, b) &= [y^2, a] + [yb + by, x] - (xya + ayx) - xbx,  &\delta(a) &= xa + ax,\\
    d(a) &= [y^2, a] - (xya + ayx),  &\partial(a) &= [x, a].
\end{align*}
We will use a spectral sequence argument coming from the filtration by columns to compute the homology of the total complex.
We need to compute the values of $\delta$ and $\partial$ on elements of $\mathcal{B}$. Given
$z = x^a(yx)^{b}y^c \in \mathcal{B}$,
\begin{align*}
    \delta(z) = x^a(yx)^by^cx + x^{a + 1}(yx)^{b}y^c.
\end{align*}
There are four cases to consider.
\begin{itemize}
	\item $a = 0$, $c = 2k$:
	\begin{align*}
		\delta(z) &= (yx)^by^{2k}x + x(yx)^by^{2k}
			= (yx)^b\left( \sum_{i = 0}^k\frac{k!}{i!}x(yx)^{k - i}y^{2i}\right) + x(yx)^by^{2k} \\
		&= \sum_{i = 0}^k\frac{k!}{i!}(yx)^bx(yx)^{k - i}y^{2i} + x(yx)^by^{2k}. 
	\end{align*}
	\begin{itemize}
		\item If $b = 0$, then $\delta(z) = \sum_{i = 0}^k\frac{k!}{i!}x(yx)^{k - i}y^{2i} + xy^{2k}.$
		\item If $b \geq 1$, then $\delta(z) = \sum_{i = 0}^k\frac{k!}{i!}(yx)^{b-1}(yx)x(yx)^{k - i}y^{2i} + x(yx)^by^{2k} = x(yx)^by^{2k}.$
	\end{itemize}
	\item $a = 0$, $c = 2k + 1$:
	\begin{align*}
		\delta(z) &= (yx)^by^{2k + 1}x + x(yx)^by^{2k + 1} \\
		&= (yx)^b\left( \sum_{i = 0}^k\frac{k!}{i!}(yx)^{k - i + 1}y^{2i}\right) + x(yx)^by^{2k + 1}\\
		&= \sum_{i = 0}^k\frac{k!}{i!}(yx)^{b + k - i + 1}y^{2i} + x(yx)^by^{2k + 1}.
	\end{align*}
	\item $a = 1$, $c = 2k$:
		\begin{align*}
			\delta(z) &= x(yx)^by^{2k}x + x^2(yx)^by^c = x(yx)^by^{2k}x \\
			&= x(yx)^b\left( \sum_{i = 0}^k\frac{k!}{i!}x(yx)^{k - i}y^{2i}\right) = \sum_{i = 0}^k\frac{k!}{i!}x(yx)^bx(yx)^{k - i}y^{2i} = 0.
		\end{align*}
	\item $a = 1$, $c = 2k + 1$:
	\begin{align*}
		\delta(x) &= x(yx)^by^{2k + 1}x = x(yx)^b\left( \sum_{i = 0}^k\frac{k!}{i!}(yx)^{k - i + 1}y^{2i}\right) = \sum_{i = 0}^k\frac{k!}{i!}x(yx)^{b + k - i + 1}y^{2i}.
	\end{align*}
\end{itemize}
Summarizing, the image of $\delta$ is generated by the set
\begin{align*}
    \Bigg\lbrace &\sum_{i = 0}^k\frac{k!}{i!}x(yx)^{k - i}y^{2i} + xy^{2k},\ x(yx)^{b + 1}y^{2k},
		 \sum_{i = 0}^k\frac{k!}{i!}(yx)^{b + k - i + 1}y^{2i} + x(yx)^by^{2k + 1}, \\
		&\qquad\sum_{i = 0}^k\frac{k!}{i!}x(yx)^{b + k - i + 1}y^{2i}\ :\ b,k \geq 0	\Bigg\rbrace.
\end{align*}
\begin{lemma}\label{lemma_imdelta}
    The subset $\left\{ x(yx)^by^{2k}, \sum_{i = 0}^k\frac{k!}{i!}(yx)^{b + k - i + 1}y^{2i} + x(yx)^by^{2k + 1} :b, k \geq 0 \right\}$
    is a basis of $\Ima\delta$.
\end{lemma}
\begin{proof}
    Let us fix some notation for the generators of $\Ima\delta$:
    \begin{align*}
	    \eta_k &=  \sum_{i = 0}^k\frac{k!}{i!}x(yx)^{k - i}y^{2i} + xy^{2k}, &\theta_{b,k} &= x(yx)^{b + 1}y^{2k},\\
	    \lambda_{b,k} &= \sum_{i = 0}^k\frac{k!}{i!}(yx)^{b + k - i + 1}y^{2i} + x(yx)^by^{2k + 1}, \quad &\mu_{b,k} &= \sum_{i = 0}^k\frac{k!}{i!}x(yx)^{b + k - i + 1}y^{2i}.
    \end{align*}
    Since
    \begin{align*}
	    \eta_k &= \sum_{i = 0}^{k-1}\frac{k!}{i!}x(yx)^{k - i}y^{2i} + 2xy^{2k} = \sum_{i = 0}^{k -1}\frac{k!}{i!}\theta_{k - i - 1, i} + 2xy^{2k}
    \end{align*}
    and $\car\field = 0$, the element $xy^{2k}$ belongs to $\Ima\delta$, for all $k \geq 0$. Also note that
    \begin{align*}
	    \mu_{b,k} &= \sum_{i = 0}^k\frac{k!}{i!}\theta_{b + k - i, i},
    \end{align*}
    so $\left\{\theta_{b,k},\ xy^{2k},\ \lambda_{b,k}\ :\ b,k \geq 0 \right\}$ generates $\Ima\delta$. Moreover, it is a basis,
    since the terms appearing in each generator are different elements of the basis $\mathcal{B}$.
\end{proof}

Next we compute the values of $\partial$ on elements of $\mathcal{B}$. Given
$z = x^a(yx)^{b}y^c \in \mathcal{B}$,
\begin{align*}
    \partial(z) = x^{a + 1}(yx)^by^cx - x^{a}(yx)^{b}y^cx.
\end{align*}
Again, there are four cases to consider.
\begin{itemize}
	\item $a = 0$, $c = 2k$: $\partial(z) = x(yx)^by^{2k} - (yx)^by^{2k}x 
				=  x(yx)^by^{2k} - \sum_{i = 0}^k\frac{k!}{i!}(yx)^bx(yx)^{k - i}y^{2i}.$
		\begin{itemize}
			\item If $b = 0$, $\partial(z) =  xy^{2k} - \sum_{i = 0}^k\frac{k!}{i!}x(yx)^{k - i}y^{2i}
					= -\sum_{i = 0}^{k - 1}\frac{k!}{i!}x(yx)^{k - i}y^{2i}.$
			\item If $b \geq 1$, then $\partial(z) = x(yx)^by^{2k}.$
		\end{itemize}
		\item $a = 0$, $c = 2k + 1$:
		\begin{align*}
			\partial(z) &= x(yx)^by^{2k + 1} - (yx)^by^{2k + 1}x 
					= x(yx)^by^{2k + 1} - \sum_{i = 0}^k\frac{k!}{i!}(yx)^{b + k - i + 1}y^{2i}.
		\end{align*}
		\item $a = 1$, $c = 2k$: $\partial(z) = x^2(yx)^by^{2k} - x(yx)^by^{2k}x = 0.$
		\item $a = 1$,  $c = 2k + 1$, $\partial(z) = x^2(yx)^by^{2k} - x(yx)^by^{2k}x = -\sum_{i = 0}^k\frac{k!}{i!}x(yx)^{b + k - i + 1}y^{2i}.$
\end{itemize}
The following lemma can be proved similarly to the previous one.
\begin{lemma}\label{lemma_impartial}
    The set $\left\{ x(yx)^{b + 1}y^{2k}, \sum_{i = 0}^k\frac{k!}{i!}(yx)^{b + k - i + 1}y^{2i} - x(yx)^by^{2k + 1} :b, k \geq 0 \right\}$
    is a basis of $\Ima\partial$.
\end{lemma}

We are now able to compute the first page $E_{1}^{\bullet, \bullet}$ of the spectral sequence. The next propositions will provide a 
description of each space $E_1^{i, j}$.
\begin{proposition}\label{proposition_kerdelta}
    The set $X = \Bigg\lbrace x(yx)^{b}y^{2k + 1} - \sum_{i = 0}^k \frac{k!}{i!}(yx)^{k + b + 1 - i}y^{2i}, x(yx)^{b}y^{2k} \mid b, k \geq 0\Bigg\rbrace$    
    is a basis of the $\field$-vector space $\Ker\delta$.
\end{proposition}
\begin{proof}
    The elements of $X$ are homogeneous. Given $z \in \Ker\delta$ we may suppose without loss of generality that $z$ is homogeneous and consider two cases.
    \begin{itemize}
        \item $\degree(z) = 2n$,
        \[
            z= \sum_{l = 0}^{n}\alpha_l(yx)^{n - l}y^{2l}  +\sum_{l = 0}^{n - 1}\beta_l x(yx)^{n - l - 1}y^{2l + 1}.
        \]
        Since $z \in \Ker\delta$,
        \begin{align*}
             0 &= \delta(z) = \sum_{l = 0}^{n}\alpha_l\delta\left((yx)^{n - l}y^{2l}\right) 
                 +\sum_{l = 0}^{n - 1}\beta_l \delta\left(x(yx)^{n - l - 1}y^{2l + 1}\right)\\
             &= \sum_{l = 0}^{n}\left( \alpha_l + \alpha_n \frac{n!}{l!} \right) x(yx)^{n - l}y^{2l}
                 + \sum_{l = 0}^{n - 1}\sum_{i = 0}^{l}\beta_l \frac{l!}{i!} x(yx)^{n - i}y^{2i} \\
             &= \sum_{l = 0}^{n}\left( \alpha_l + \alpha_n \frac{n!}{l!} \right) x(yx)^{n - l}y^{2l}
                 + \sum_{i = 0}^{n - 1}\sum_{l = i}^{n - 1}\beta_l\frac{l!}{i!} x(yx)^{n - i}y^{2i} \\
             &= \sum_{l = 0}^{n}\left(\alpha_l + \alpha_n \frac{n!}{l!} + \sum_{j = l}^{n - 1}\beta_j \frac{j!}{l!}\right)x(yx)^{n - l}y^{2l}.
        \end{align*}
        All the monomials appearing in the sum are $\field$-linearly independent, so
        \[
            \alpha_n = 0 \text{ and } \alpha_l = -\sum_{j = l}^{n - 1}\beta_j \frac{j!}{l!}, \text{ for all } l, 0 \leq l < n.
        \]
        We can now write $z$ as a linear combination of elements of $X$ as follows
        \begin{align*}
                z&= -\sum_{l = 0}^{n - 1}\left(\sum_{j = l}^{n - 1}\beta_j \frac{j!}{l!}\right)(yx)^{n - l}y^{2l} 
                    +\sum_{l = 0}^{n - 1}\beta_l x(yx)^{n - l - 1}y^{2l + 1}\\
                 &= -\sum_{j = 0}^{n - 1}\beta_j \left(\sum_{l = 0} ^{j}\frac{j!}{l!}(yx)^{n -l}y^{2l} + x(yx)^{n - j - 1}y^{2j + 1}\right).
        \end{align*}
        \item $\degree(z) = 2n + 1$,
        \[
             z= \sum_{l = 0}^{n}\alpha_l(yx)^{n - l}y^{2l + 1}  +\sum_{l = 0}^{n}\beta_l x(yx)^{n - l}y^{2l}.            
        \]
        Then
        \begin{align*}
            0 &= \delta(z) = \sum_{l = 0}^{n}\alpha_l\delta \left((yx)^{n - l}y^{2l + 1}\right)
                    + \sum_{l = 0}^{n}\beta_l \delta \left( x(yx)^{n - l}y^{2l}\right)\\
            &= \sum_{l = 0}^{n}\sum_{i = 0}^{l}\alpha_l \frac{l!}{i!}(yx)^{n + 1 - i}y^{2i} + \sum_{l = 0}^{n}\alpha_l x(yx)^{n - l}y^{2l + 1}.
         \end{align*}
         Hence $\alpha_l = 0$ for all $l$, $0 \leq l \leq n$ and we can write $z=\sum_{l = 0}^{n}\beta_l x(yx)^{n - l}y^{2l}$.
    \end{itemize}
    We have already proved that $X$ generates $\Ker\delta$ and it is easy to see that it is linearly independent.
\end{proof}

By definition there are isomorphisms $E_{1}^{1, 2i + 1} \cong \Ker\partial/ \Ima\delta \cong E_{1}^{2, 2i + 1}$ for all $i \geq 0$ and
$E_{1}^{1, 2i} \cong \Ker\delta/ \Ima\partial \cong E_{1}^{2, 2i}$ for all $i \geq 1$.
Form now on, we will denote by $\left[z\right]$ the class in $E_1^{i, j}$ of an element $z \in E_0^{i, j}$ and
we will write $([a], [b])$ instead of $[(a, b)]$ for $(a, b) \in E_0^{1, 0}$.
\begin{proposition}
    The set $\left\lbrace \left[xy^{2n}\right] | n \geq 0 \right\rbrace$ is a basis of the vector spaces $E_1^{1, 2i}$ and $E_1^{2, 2i}$ for all $i \geq 1$.
\end{proposition}
\begin{proof}
    The proof follows from Proposition \ref{proposition_kerdelta} and Lemma \ref{lemma_impartial}.
\end{proof}
\begin{proposition} \label{prop_basis_eodd}
    The set $\left\{ \sum_{l = 0}^n \frac{n!}{l!}\left[(yx)^{n - l}y^{2l}\right] \ :\ n \geq 0\right\}$
    is a basis of the vector spaces $E_1^{1, 2i + 1}$ and $E_1^{2, 2i + 1}$ for all $i \geq 1$.
\end{proposition}
\begin{proof}
    Again, given $z \in \Ker\partial$ we may suppose that it is homogeneous, and consider two cases.
	\begin{itemize}
		\item $\degree(z) = 2n$,
		\begin{align*}
			z = \sum_{l = 0}^n\alpha_l(yx)^{n - l}y^{2l} + \sum_{i = 1}^n\beta_ix(yx)^{n - i}y^{2i - 1}.
		\end{align*}
		By Lemma \ref{lemma_imdelta}, we know that $x(yx)^{n - i}y^{2i - 1} + \sum_{j = 0}^{i - 1}\frac{(i-1)!}{j!}(yx)^{n - j}y^{2j} \in \Ima\delta$
		for all $i$, $1 \leq i \leq n$.
		Hence, we may assume that $z = \sum_{l = 0}^n\alpha_l(yx)^{n - l}y^{2l}$.
		Since  $z \in \Ker\partial$,
		\begin{align*}
			0 &= \sum_{l = 0}^n\alpha_l\partial((yx)^{n - l}y^{2l})
				= \sum_{l = 0}^{n - 1}\alpha_l\partial((yx)^{n - l}y^{2l}) + \alpha_n\partial(y^{2n})\\
			&= \sum_{l = 0}^{n - 1}\alpha_l x(yx)^{n - l}y^{2l} - \alpha_n\sum_{l = 0}^{n - 1}\frac{n!}{l!}x(yx)^{n - l}y^{2l}
				= \sum_{l = 0}^{n - 1}(\alpha_l - \alpha_n\frac{n!}{l!})x(yx)^{n - l}y^{2l}.
		\end{align*}
		This implies that $\alpha_l = \alpha_n\frac{n!}{l!}$ for all $l$, $0 \leq l \leq n$ and
		\[
			z = 	\sum_{l = 0}^n\alpha_l(yx)^{n - l}y^{2l} = \alpha_n\sum_{l = 0}^n\frac{n!}{l!}(yx)^{n - l}y^{2l}.
		\]
		\item $\degree(z) = 2n + 1$,
		\begin{align*}
			z = \sum_{l = 0}^n\alpha_l(yx)^{n - l}y^{2l + 1} + \sum_{l = 0}^n\beta_lx(yx)^{n - l}y^{2l}.
		\end{align*}
		By Lemma \ref{lemma_imdelta}, $x(yx)^{n - l}y^{2l} \in \Ima\delta$ for all $l$, $0 \leq l \leq n$.
		Therefore, we may assume that $z = \sum_{l = 0}^n\alpha_l(yx)^{n - l}y^{2l + 1}$.
		Since $z \in \Ker\partial$,
		\begin{align*}
			0 &= \sum_{l = 0}^n\alpha_l\partial((yx)^{n - l}y^{2l + 1})
				= \sum_{l = 0}^n\alpha_l\left(-\sum_{i = 0}^l\frac{l!}{i!}(yx)^{n - i + 1}y^{2i} + x(yx)^{n - l}y^{2l + 1}\right)\\
			&= -\sum_{l = 0}^n\sum_{i = 0}^l\alpha_l\frac{l!}{i!}(yx)^{n - i + 1}y^{2i} + \sum_{l = 0}^n\alpha_lx(yx)^{n - l}y^{2l + 1},
		\end{align*}
		from which we obtain that $\alpha_l = 0$ for all $l$, $0 \leq l \leq n$.
	\end{itemize}
	So the set in the statement of the lemma generates $\Ker\partial/ \Ima\delta$ and it is clear that is linearly independent.
\end{proof}

The description of the first page of the spectral sequence ends by observing that $E_1^{0,0} = A$, while $E_1^{1, 0} = \Ker\delta\oplus A$
and $E_{1}^{2, 0} = \Ker\delta$. We summarize this information in the following diagram of $E_1^{\bullet, \bullet}$
\begin{align*}
\xymatrix{
	& & \\
	& \left\langle \sum_{l = 0}^{n}\frac{n!}{l!}\left[(yx)^{n - l}y^{2l}\right] \right\rangle \ar[r]^{d_{(1)}} \ar@{--}[u]
	    & \left\langle \sum_{l = 0}^{n}\frac{n!}{l!}\left[(yx)^{n - l}y^{2l}\right] \right\rangle \ar@{--}[u] \\
    & \left\langle \left[xy^{2n}\right] \right\rangle \ar@{--}[u] \ar[r]^{d_{(1)}} & \left\langle \left[xy^{2n}\right] \right\rangle \ar@{--}[u] \\
    & \left\langle \sum_{l = 0}^{n}\frac{n!}{l!}\left[(yx)^{n - l}y^{2l}\right] \right\rangle \ar[r]^{d_{(1)}} \ar@{--}[u]
	    & \left\langle \sum_{l = 0}^{n}\frac{n!}{l!}\left[(yx)^{n - l}y^{2l}\right] \right\rangle \ar@{--}[u] \\ 
    A \ar[r]^{d^{0}_{(1)}} & \Ker(\delta) \oplus A \ar[r]^{d^{1}_{(1)}} \ar@{--}[u] & \Ker(\delta) \ar@{--}[u]
} 
\end{align*}
where $d^{0}_{(1)}, d^{1}_{(1)}$ and $d_{(1)}$ are the maps induced by $d^{0}, d^{1}$ and $d$ respectively.

Next we compute the second page of the spectral sequence, which will be the last one since, due to the shape of our complex,
$E^{\bullet, \bullet}_2 = E^{\bullet, \bullet}_{\infty}$. For this, we describe in a series of lemmas the images of the maps
$d^{0}_{(1)}, d^{1}_{(1)}$ and $d_{(1)}$.
\begin{lemma}\label{lemma_disnull1}
    The map $d_{(1)}: E^{1, 2i + 1}_{1} \to E^{2, 2i + 1}_{1}$ is null for all $i \geq 0$.
\end{lemma}
\begin{proof}
Let $[w]$ be an element of the basis of $E^{1, 2i + 1}_{1}$. By Proposition \ref{prop_basis_eodd}, we may choose
\[
	w = \sum_{l = 0}^{n}\frac{n!}{l!}(yx)^{n - l}y^{2l}.
\]
\begin{itemize}
    \item If $n = 0$, then $d(w) = d(1) = \left[y^2, 1\right] - xy - yx = - xy - yx$. Thus $d(w)$ belongs to $\Ima\delta$  and $d_{(1)}([w]) = 0$.
    \item If $n \geq 1$, we may assume that $n = m + 1$, where $m \geq 0$ and write
    \[
        w = \sum_{l = 0}^{m + 1}\frac{(m + 1)!}{l!} (yx)^{m + 1 -l}y^{2l} = (m + 1)y^{2m + 1}x + y^{2m + 2}.
    \]
    Applying $d$, we obtain that $d(w) = (m + 1)d\left(y^{2m + 1}x\right) + d\left(y^{2m + 2}\right)$.    
    Next we compute each term separately
    \begin{align*}
        d\left(y^{2m + 1}x\right) &= \left[y^{2}, y^{2m + 1}x\right] - xy^{2m + 2}x - y^{2m + 1}xyx \\
            &= y^{2m + 1}\left[y^{2}, x\right] - y^{2m + 1}xyx\\
        & = y^{2m + 1}\left(y^2x - xy^{2} - xyx\right) = 0,
    \end{align*}
     \begin{align*}
         d\left(y^{2m + 2}\right) &= \left[y^{2}, y^{2m + 2}\right] - xy^{2m + 3} - y^{2m + 3}x\\
         &= xy^{2m + 3} - \sum_{l =0}^{m + 1}\frac{(m + 1)!}{l!}(yx)^{m + 2 - l}y^{2l}.
     \end{align*}
     Since $xy^{2m + 3} - \sum_{l =0}^{m + 1}\frac{(m + 1)!}{l!}(yx)^{m + 2 - l}y^{2l}$ belongs to the image of $\delta$, $d_{(1)}([w]) = 0$.
\end{itemize}
\end{proof}
\begin{lemma}\label{lemma_disnull2}
    The map $d_{(1)}: E^{1, 2i}_{1} \to E^{2, 2i}_{1}$ is null for all $i \geq 1$.
\end{lemma}
\begin{proof}
Let $\left[xy^{2n}\right]$ be an element of the basis of $E^{1, 2i}_{1}$. Since
\begin{align*}
    d(xy^{2n}) &= \left[y^{2}, xy^{2n}\right] - xyxy^{2n} - xy^{2n + 1}x \\
    &= y^{2}xy^{2n} - xy^{2n + 2} - xyxy^{2n} - \sum_{l = 0}^{n}\frac{n!}{l!}x(yx)^{n - l + 1}y^{2l} \\
    &= xy^{2n + 2} + xyxy^{2n} - xy^{2n + 2} - xyxy^{2n} - \sum_{l = 0}^{n}\frac{n!}{l!}x(yx)^{n - l + 1}y^{2l} \\
    &= -\sum_{l = 0}^{n}\frac{n!}{l!}x(yx)^{n - l + 1}y^{2l},
\end{align*}
and $x(yx)^{b}y^{c}$ belongs to $\Ima\partial$ for all $b \geq 1$ and for all $c \geq 0$, we get that
$d_{(1)}(\left[xy^{2n}\right]) = 0$ for all $n \geq 0$.
\end{proof}

We compute now the values of $d^0$ on elements of the basis of $E_1^{0,0} = A$.
By definition, $d^{0}(z) = \left([x,z], [y,z]\right)$ for all $z \in A$. Since $[x,z] = \partial(z)$, we
only need to know the values of $[y, z]$ for all $z = x^a(yx)^by^c \in \mathcal{B}$. There are three cases
to consider:
\begin{itemize}
    \item $a = 1: \qquad [y, z] = (yx)^{b + 1}y^c - x(yx)^by^{c + 1}.$
    \item $a = 0, b = 0: \qquad [y, z] = 0.$
    \item $a = 0, b \geq 1:\qquad [y, z] = x(yx)^{b - 1}y^{c + 2} + bx(yx)^{b}y^c -  (yx)^{b}y^{c + 1}.$
\end{itemize}
Putting together these computations and the ones we made to obtain the image of $\partial$, we get that
if $z = x^a(yx)^by^c \in \mathcal{B}$, then
\begin{itemize}
    \item $a = 0$, $c = 2k$:
        \begin{itemize}
            \item if $b = 0$, then $d^{0}_{(1)}([z]) = \left(-\sum_{l = 0}^{k -1}\frac{k!}{l!}\left[x(yx)^{k - l}y^{2l}\right], 0\right),$
            \item if $b \geq 1$, then
            \[
                d^0_{(1)}([z]) = \left(\left[x(yx)^by^{2k}\right],\left[x(yx)^{b - 1}y^{2k + 2}\right]
                    + b\left[x(yx)^{b}y^{2k}\right] - \left[(yx)^{b}y^{2k + 1}\right]\right).
            \]
        \end{itemize}
    \item $a = 0$, $c = 2k + 1$:
     	\begin{itemize}
            \item if $b = 0$, then $d^{0}_{(1)}([z]) = \left(-\sum_{l = 0}^{k}\frac{k!}{l!}\left[(yx)^{k + 1 - l}y^{2l}\right] + \left[xy^{2k + 1}\right], 0\right),$
            \item if $b \geq 1$, then
            \begin{align*}
                d^{0}_{(1)}([z]) &= \Bigg(-\sum_{l = 0}^{k}\frac{k!}{l!}\left[(yx)^{k + 1 + b - l}y^{2l}\right]
                    + \left[x(yx)^{b}y^{2k + 1}\right],\\
                &\qquad \left[x(yx)^{b - 1}y^{2k + 3}\right]
                    + b\left[x(yx)^{b}y^{2k + 1}\right] - \left[(yx)^{b}y^{2k + 2}\right]\Bigg).
            \end{align*}
        \end{itemize}
     \item $a = 1$, $c = 2k$, $d^{0}_{(1)}([z]) = \left(0, \left[(yx)^{b + 1}y^{2k}\right] - \left[x(yx)^by^{2k + 1}\right] \right),$
     \item $a = 1$, $c = 2k + 1$:
          \[
              d^{0}_{(1)}([z]) = \Bigg(-\sum_{l = 0}^{k}\frac{k!}{l!}\left[x(yx)^{k + 1 + b - l}y^{2l}\right],
                  \left[(yx)^{b + 1}y^{2k + 1}\right] - \left[x(yx)^{b}y^{2k + 2}\right]\Bigg).
          \]      
	\end{itemize}
\begin{lemma} \label{lemma_imd0}
	The set $Y = \left\lbrace \rho_k, \zeta_{b , k}, \sigma_k, \tau_{b , k}, \nu_{b, k},
			\xi_{b, k} \mid b, k \geq 0 \right\rbrace$, where
	\begin{align*}
		\rho_k &= \left(\sum_{l = 0}^{k - 1}\frac{k!}{l!}\left[x(yx)^{k -l}y^{2l}\right], 0\right),\\
    		\zeta_{b,k} &= \left(\left[x(yx)^{b + 1}y^{2k}\right],
    				\left[x(yx)^{b}y^{2k + 2}\right] + (b + 1)\left[x(yx)^{b  + 1}y^{2k}\right] 
        			- \left[(yx)^{b + 1}y^{2k + 1}\right]\right),\\
    		\sigma_{k} &= \left(\sum_{l = 0}^{k}\frac{k!}{l!}\left[(yx)^{k + 1 - l}\right]
    				- \left[xy^{2k + 1}\right], 0\right),\\
    		\tau_{b, k} &= \Bigg(-\sum_{l = 0}^{k}\frac{k!}{l!}\left[(yx)^{k + 2 + b - l}y^{2l}\right]
                    + \left[x(yx)^{b + 1}y^{2k + 1}\right],
                    (b + 1) \left[x(yx)^{b + 1}y^{2k + 1}\right]\Bigg),\\
    		\nu_{b , k} &=\left(0, \left[(yx)^{b + 1}y^{2k}\right] - \left[x(yx)^{b}y^{2k + 1}\right]\right),\\
    		\xi_{b, k} &= \Bigg(-\sum_{l = 0}^{k - 1}\frac{k!}{l!}\left[x(yx)^{k + 1 + b - l}y^{2l}\right],
        			(b + 1)\left[x(yx)^{b + 1}y^{2k}\right]\Bigg).
	\end{align*} 	
 	is a basis of $\Ima d^{0}_{(1)}$.
\end{lemma}
\begin{proof}
	First, notice that
	\begin{align*}
    		d^{0}&\left((yx)^{b + 1}y^{2k + 1} + x(yx)^{b}y^{2(k + 1)}\right) = \\
    		& \Bigg(-\sum_{l = 0}^{k}\frac{k!}{l!}(yx)^{k + 2 + b - l}y^{2l}
                    + x(yx)^{b + 1}y^{2k + 1}, (b + 1) x(yx)^{b + 1}y^{2k + 1}\Bigg)
	\end{align*}
	and
	\begin{align*}
		d^{0}&\left(x(yx)^by^{2k + 1} + (yx)^{b + 1}y^{2k}\right) =
    				\Bigg(-\sum_{l = 0}^{k - 1}\frac{k!}{l!}x(yx)^{k + 1 + b - l}y^{2l},
    				(b + 1)x(yx)^{b + 1}y^{2k}\Bigg).
	\end{align*}
	This implies that $Y$ generates  $\Ima d^{0}_{(1)}$. Let us prove that
	it is linearly independent. For this, suppose that $w$ is a null linear combination
	of elements of $Y$. We may suppose, without loss of generality, that $w$ is homogeneous. If
	$\degree(w) = 2k + 1$, we have
	\begin{align*}
        0 &= w = \alpha \rho_k + \sum_{i = 0}^{k - 1} \beta_i \zeta_{k -i - 1, i}
        			+ \sum_{i = 0}^{k - 1}\gamma_i \xi_{k - i - 1, i}
        = \alpha \rho_k + \sum_{i = 0}^{k - 1}\gamma_i \xi_{k - i - 1, i}\\
        &\qquad + \sum_{i = 0}^{k - 1} \beta_i \Bigg(\left[x(yx)^{k - i}y^{2i}\right],
            \left[x(yx)^{k - i - 1}y^{2i + 2}\right] + (k - i)\left[x(yx)^{k - i}y^{2i}\right]\\
            &\qquad \qquad - \left[(yx)^{k - i}y^{2i + 1}\right]\Bigg).
    \end{align*}
    Since terms of type $(yx)^{k - i}y^{2i + 1}$ do not appear in $\xi_{\ast, \ast}$, neither in $\rho_{\ast}$,
    we conclude that $\beta_i = 0$ for all $i$, $0 \leq i \leq k - 1$. Consequently,
    \begin{align*}
        0 &= \alpha \rho_k +  \sum_{i = 0}^{k - 1}\gamma_i \xi_{k - i - 1, i} = \alpha \rho_k + \sum_{i = 0}^{k - 1}\gamma_i \Bigg(-\sum_{l = 0}^{i - 1}\frac{i!}{l!}\left[x(yx)^{k - l}y^{2l}\right],
                  (k - i)\left[x(yx)^{k - i}y^{2i}\right]\Bigg).
    \end{align*}
    The fact that the second coordinate of $\rho_k$ is null implies that $\gamma_i = 0$ for all $i$,
    $0 \leq i \leq k - 1$, and thus $\alpha = 0$. The case $\degree(z) = 2k$ is analogous.
\end{proof}

Now, we want to compute a basis of $\Ima d_{(1)}^{1} $. In fact we know that
$\Ima d_{(1)}^{1} = \Ima \overline{d_{(1)}^{1}}$, where
$\overline{d_{(1)}^{1}}: \frac{\Ker\delta \oplus A}{\Ima d_{(1)}^{0} } \to \Ker\delta$
is the induced map. Given $z \in \Ker\delta \oplus A$, we will thus reduce it modulo $\Ima d_{(1)}^{0}$.
Using the notations and the proof of the previous lemma, we already know
that the elements $\zeta_{b, k}$, $\sigma_{k}$ and $\tau_{b, k}$ are in the image of $d_{(1)}^{0}$
for all $b, k \geq 0$, which allows us to suppose that $z = \left(xy^{2k}, 0\right)$
or $z = \left(0, x^a(yx)^by^c\right)$. The first possibility is easier, since
\[
    d^1(z) = \left[y^{2}, xy^{2k}\right] - x(yx)y^{2k} - xy^{2k + 1}x = -\sum_{i = 0}^{k}\frac{k!}{i!}x(yx)^{k + 1- i}y^{2i}.
\]
For $z = \left(0, x^a(yx)^by^c\right)$, we study several cases. Suppose first that $a = 0$, then
\[
	d^1(z) = \left[y(yx)^by^c + (yx)^by^{c + 1}, x\right] - x(yx)^by^cx.        
\]
\begin{itemize}
	\item If $b = 0$, then $d^1(z) = 2\left[y^{c + 1}, x\right] - xy^cx = 2\left(y^{c + 1}x - xy^{c + 1}\right) - xy^cx.$
    \begin{itemize}
    		\item If $c = 2k$, then $ d^1(z) = 2\left(y^{2k + 1}x - xy^{2k + 1}\right)
            				= 2\left(\sum_{i = 0}^k \frac{k!}{i!}(yx)^{k + 1 - i}y^{2i} - xy^{2k + 1}\right)$
            \item If $c = 2k + 1$, then
            \begin{align*}
                d^1(z) &= 2\left(y^{2k + 2}x - xy^{2k + 2}\right) - xy^{2k + 1}x\\
                &= 2\sum_{i = 0}^{k}\frac{(k + 1)!}{i!}(yx)^{k + 1 - i}y^{2i}
                			- \sum_{i = 0}^k \frac{k!}{i!}x(yx)^{k + 1 - i}y^{2i}\\
                & = (2k + 1)\sum_{i = 0}^k \frac{k!}{i!}x(yx)^{k + 1 - i}y^{2i}.
            \end{align*}
	\end{itemize}
    \item If $b \geq 1$, then
        \begin{align*}
            d^1(z) &= \left[x(yx)^{b - 1}y^{c + 2} + b x(yx)^{b}y^c + (yx)^{b}y^{c + 1}, x\right] - x(yx)^by^c \\
            &= x(yx)^{b - 1}y^{c + 2}x + bx(yx)^by^cx + (yx)^{b}y^{c + 1}x\\
            &\qquad - x(yx)^{b}y^{c + 1} - x(yx)^by^cx\\
            &= x(yx)^{b - 1}y^{c + 2}x + (b - 1)x(yx)^by^cx + (yx)^{b}y^{c + 1}x - x(yx)^{b}y^{c + 1}.
        \end{align*}
        \begin{itemize}
            \item If $c = 2k$, then $d^1(z) = \sum_{i =0}^{k}\frac{k!}{i!}(yx)^{k + b + 1 - i}y^{2i} - x(yx)^by^{2k + 1}.$
            \item If $c = 2k + 1$, then $d^1(z) = \sum_{i = 0}^{k}\frac{k!}{i!}(k + b)x(yx)^{k +  b + 1 - i}y^{2i}.$
        \end{itemize}
\end{itemize}
For $a =1$,
	\begin{align*}
    		d^1(z) &= \left[(yx)^{b + 1}y^cx  + x(yx)^{b}y^{c + 1}, x\right]\\
        &= (yx)^{b + 1}y^{c}x - x(yx)^{b + 1}y^{c} + x(yx)^b y^{c + 1}x.
    \end{align*}
    \begin{itemize}
        \item If $c = 2k$, then $d^1(z) = \sum_{i = 0}^{k - 1}\frac{k!}{i!} x(yx)^{k + b + 1 -i}y^{2i}.$
        \item If $c = 2k + 1$, then $d^1(z) = \sum_{i = 0}^{k}\frac{k!}{i!} (yx)^{k  + b + 2 -i}y^{2i} - x(yx)^{b + 1}y^{2k + 1}.$
    \end{itemize}
Since $\sum_{i = 0}^{k}\frac{k!}{i!}\left[x(yx)^{k +  b + 1 - i}y^{2i}\right]$ and $\sum_{i = 0}^{k - 1}\left[x(yx)^{k + b + 1 -i}y^{2i}\right]$
belong to $\Ima d^{1}_{(1)}$, the same is true for their difference.
In other words, $\left[x(yx)^{b + 1}y^{2i}\right]$ belongs to $\Ima d^{1}_{(1)}$ for all $b, i \geq 0$. The proof of the following lemma is now clear.
\begin{lemma}\label{lemma_imad11}
	The set
    \[
        \left\lbrace \sum_{i = 0}^{k}\frac{k!}{i!}(yx)^{k + b + 1 - i}y^{2i} - x(yx)^{b}y^{2k + 1},
            x(yx)^{b + 1}y^{2i} \mid b, k \geq 0 \right\rbrace
    \]
    is a basis of $\Ima d^{1}_{(1)} $.
\end{lemma}

We are now ready to describe the second page of the spectral sequence. We will denote by $\ov{z}$ the class in $E_2^{i,j}$ of an element $z \in E_{0}^{i,j}$ and again we will write $(\ov{a}, \ov{b})$ instead of $\ov{(a, b)}$ for $(a, b) \in E_0^{1, 0}$.
\begin{proposition}
	\begin{enumerate}[(i)]
		\item The set $\left\lbrace \sum_{l = 0}^{n}\frac{n!}{l!}\overline{(yx)^{n - l}y^{2l}} | n \geq 0 \right\rbrace$ is a basis of the vector spaces $E_2^{1, 2i +1}$ and $E_2^{2, 2i + 1}$, for all $i \geq 0$.
		\item The set $\left\lbrace \overline{xy^{2n}} | n \geq 0 \right\rbrace$ is a basis of the
		vector spaces $E_2^{2, 0}$, $E_2^{1, 2i}$ and $E_2^{2, 2i}$, for all $i \geq 1$.
	\end{enumerate}
\end{proposition}
\begin{proof}
	It is a direct consequence of Lemmas \ref{lemma_disnull1}, \ref{lemma_disnull2} and \ref{lemma_imad11}.
\end{proof}

Let us now describe $E_2^{1, 0}$. We claim that the set
$\left\lbrace\left(0, \ov{x}\right), \left((2k + 1)\ov{xy^{2k}},
		\ov{y^{2k + 1}}\right) \mid k\geq 0 \right\rbrace$ is a basis of $E^{1, 0}_2$.
In order to prove the claim we may suppose, as always,
that each coordinate of $\left[z\right] = \left(\left[z_1\right], \left[z_2\right]\right)
\in \Ker d_{(1)}^1 $ is homogeneous, both of the same degree and we can also
reduce modulo boundaries to write $\left[z\right]$ as a linear combination of elements
in the set
\[
    \left\lbrace \left(\left[xy^{2n}\right], 0\right), \left(0, \left[x^{a}(yx)^{b}y^c\right]\right)
        \mid n \geq 0, a \in \{0, 1\}, b, c \geq 0  \right\rbrace.
\]
In case $\degree(z) = 2k$, we write $ z = \left(0, \sum_{i = 0}^{k}\alpha_i (yx)^{k - i}y^{2i}
    		+ \sum_{i = 0}^{k - 1}\beta_i x(yx)^{k - i - 1}y^{2i + 1}\right)$ and looking
at the term $(yx)y^{2k}$ in the equality $d^1(z) = 0$, we obtain that $\alpha_k = 0$,
and now the equality is:
\begin{align*}
0 &=  \sum_{l =0}^{k - 1} \left(\sum_{i = l}^{k - 1}\frac{i!}{l!}(\alpha_i + \beta_i)\right)(yx)^{k + 1 -l}y^{2l}
        - \sum_{i = 0}^{k - 1}\left(\alpha_i + \beta_i\right)x(yx)^{k - i}y^{2i + 1}.
\end{align*}
We deduce that $\alpha_i = -\beta_i$ for all $i$, $0 \leq i \leq k - 1$, so that
\[
    [z] = \sum_{i = 0}^{k - 1}\alpha_i\left(0, \left[(yx)^{k - i}y^{2i}\right]
    			-\left[x (yx)^{k - i - 1}y^{2i + 1}\right]\right),
\]
but each term is zero because $\left[(yx)^{k - i}y^{2i}\right] -\left[x (yx)^{k - i - 1}y^{2i + 1}\right]$
belongs to $\Ima d_{(1)}^0 $ for all i, $0 \leq i \leq k - 1$.

In case $\degree(z) = 2k + 1$, we write \[
    z = \sum_{i = 0}^{k}\alpha_i \left(0, (yx)^{k - i}y^{2i + 1}\right) + \sum_{i = 0}^{k}\beta_i \left(0, x(yx)^{k - i}y^{2i}\right)
        +\gamma\left(xy^{2k}, 0\right)
\]
Thus,
\begin{align*}
d^{1}(z) &= \sum_{i = 0}^{k - 1}\sum_{l = 0}^{i}\alpha_i \frac{i!}{l!}kx(yx)^{k + 1 - l}y^{2l}
        + \sum_{i = 0}^{k}\sum_{l = 0}^{i - 1}\beta_i\frac{i!}{l!}x(yx)^{k + 1 - l}y^{2l}\\
        &\qquad + \left(\alpha_k (2k + 1) - \gamma\right)\sum_{i = 0}^{k}\frac{k!}{i!}x(yx)^{k + 1- i}y^{2i}
\end{align*}
which is zero since $\left[z\right] \in \Ker d_{(1)}^1 $. Looking at the term $x(yx)y^{2k}$ we
conclude that $\gamma = \alpha_k(2k + 1)$. Consequently,
\begin{align*}
    0 &= \sum_{i = 0}^{k - 1}\sum_{l = 0}^{i}\alpha_i \frac{i!}{l!}kx(yx)^{k + 1 - l}y^{2l}
        + \sum_{i = 0}^{k}\sum_{l = 0}^{i - 1}\beta_i\frac{i!}{l!}x(yx)^{k + 1 - l}y^{2l}\\
    &=   \sum_{l = 0}^{k - 1}\left(\sum_{i = l}^{k - 1}\alpha_i \frac{i!}{l!}k\right)x(yx)^{k + 1 - l}y^{2l}
        + \sum_{l = 0}^{k - 1}\left(\sum_{i = l + 1}^{k}\beta_i\frac{i!}{l!}\right)x(yx)^{k + 1 - l}y^{2l}.\\
\end{align*}
From this, we get $\sum_{i = l}^{k - 1}\alpha_i \frac{i!}{l!}k = -\sum_{i = l + 1}^{k}\beta_i\frac{i!}{l!} = -\sum_{i = l}^{k - 1}\beta_{i + 1}\frac{(i + 1)!}{l!}$
for all $l$, $0\leq l \leq k - 1$.
Let us denote $c_l = \sum_{i = l}^{k - 1}\alpha_i i!k + \beta_{i + 1}(i + 1)$. Of course
$c_l = 0$, and 
$0 = c_l - c_{l + 1} = \alpha_l l!k + \beta_{l + 1}(l + 1)!$, so $\beta_{l + 1} = -\alpha_l \frac{k}{l + 1}$ for all $l$, $0 \leq l \leq k - 2$. Moreover, since  $0 = c_{k - 1} = \alpha_{k - 1}k! + \beta_k k!$, we conclude
that $\beta_{l + 1} = -\alpha_l \frac{k}{l + 1}$ for all $l$, $0 \leq l \leq k - 1$. We are now able to write
\begin{align*}
    z &= \sum_{i = 0}^{k - 1}\alpha_i \left(0, (yx)^{k - i}y^{2i + 1}\right) - \sum_{i = 1}^{k} \alpha_{i - 1} \frac{k}{i}\left(0, x(yx)^{k - i}y^{2i}\right)
        + \alpha_k\left((2k + 1)xy^{2k}, y^{2k + 1}\right)\\
    &\qquad + \beta_0 \left(0, x(yx)^{k}\right)\\
    &= \sum_{i = 0}^{k - 1}\alpha_i \left(0, (yx)^{k - i}y^{2i + 1} - \frac{k}{i + 1}x(yx)^{k - 1 - i}y^{2i + 2}\right)
        + \alpha_k\left((2k + 1)xy^{2k}, y^{2k + 1}\right)\\
    &\qquad + \beta_0 \left(0, x(yx)^{k}\right),
\end{align*}	
proving that $E_2^{1, 0}$ is generated by the set
\begin{align*}
     \Bigg\lbrace \left(0, \ov{(yx)^{b + 1}y^{2k + 1}} - \frac{b + k + 1}{k + 1}\ov{x(yx)^{b}y^{2k + 2}}\right),
         \left(0, \ov{x(yx)^{k}}\right), \left((2k + 1)\ov{xy^{2k}}, \ov{y^{2k + 1}}\right) \mid k, b \geq 0\Bigg\rbrace.
\end{align*}
Using the notation of Lemma \ref{lemma_imd0}, there are equalities
\begin{align*}
    &\left(0, \left[x(yx)^{k + 1}\right]\right) = \frac{1}{k + 1}\xi_{0, k},\\
    &\left(0, \left[(yx)y^{2k + 1}\right] - \left[xy^{2k + 2}\right]\right) =
    		-\zeta_{0, k} + \xi_{0, k} + \frac{1}{k + 1}\rho_{k + 1},\\
    &\left(0, \left[(yx)^{b + 1}y^{2k + 1}\right] - \frac{b +k + 1}{k + 1}\left[x(yx)^{b}y^{2k + 2}\right]\right)
        = -\zeta_{b, k} + \xi_{b, k} -\frac{1}{k + 1}\xi_{b - 1, k + 1}
\end{align*}
for all $b \geq 1$, $k \geq 0$. Therefore, the set $\left\lbrace\left(0, \ov{x}\right), \left((2k + 1)\ov{xy^{2k}},
		\ov{y^{2k + 1}}\right) \mid k\geq 0 \right\rbrace$ generates $E^{1, 0}_2$.
To prove that it is linearly independent, pick a linear homogeneous combination of these generators
and suppose that it equals zero. Suppose first that $w$ is a linear combination of elements of degree $n$,
so that for $n = 1$, $0 = w = \alpha (0, \ov{x})$ + $\beta(\ov{x}, \ov{y})$ for some $\alpha, \beta \in \field$.
No monomial of degree $1$ appears as second coordinate of the generators of $\Ima d_{(1)}^{0} $,
so $\alpha = \beta = 0$. For $n = 2k + 1 > 1$, we have that
$w = \left((2k + 1)\ov{xy^{2k}}, \ov{y^{2k + 1}}\right)$. As before, $\left[y^{2k + 1}\right]$
does not appear as a second coordinate of any element in $\Ima d_{(1)}^{0} $, so that
the coefficient must be zero.

The only space of the second page of the spectral sequence left to describe is $E_{2}^{0 ,0}$. We do
it in the next proposition.
\begin{proposition}
	The vector space $E_{2}^{0, 0}$ is isomorphic to $\field$.
\end{proposition}
\begin{proof}
	By definition, $E_{2}^{0, 0} = \Ker d_{(1)}^{0} $. Given $\left[z\right] \in \Ker d_{(1)}^{0} $,
	we shall consider two cases, accordingly to the parity of $\degree(z)$. First, suppose that $\degree(z) = 2n$ and write
	\[
		z= \sum_{l = 0}^{n}\alpha_l(yx)^{n - l}y^{2l} + \sum_{l = 0}^{n - 1}\beta_l x(yx)^{n - l - 1}y^{2l + 1} \in \Ker d_{(1)}^{0},
	\]
	that is, it commutes with $y$ and $x$. Using the commutation formulas, we obtain:
	\begin{align*}
	 0 &= [y, z] = \sum_{l = 0}^{n}\alpha_l\left[y, (yx)^{n - l}y^{2l}\right]
            + \sum_{l = 0}^{n - 1}\beta_l\left[y, x(yx)^{n - l - 1}y^{2l + 1}\right] \\
        &= \sum_{l = 0}^{n - 1}\alpha_l \left(x(yx)^{n - l - 1}y^{2l +2} + (n - l)x (yx)^{n - l}y^{2l} - (yx)^{n - l}y^{2l  +1}\right)\\
        &\qquad + \sum_{l = 0}^{n - 1}\beta_l \left((yx)^{n - l}y^{2l + 1} - x(yx)^{n - l - 1}y^{2l + 2}\right)\\
        & = \sum_{l = 0}^{n - 1}\left(\beta_l - \alpha_l\right)(yx)^{n - l}y^{2l + 1} 
            + \sum_{l = -1}^{n - 2}\alpha_{l + 1} (n - l - 1) x(yx)^{n - l - 1}y^{2l + 2}\\
        &\qquad + \sum_{l = 0}^{n -1}\left(\alpha_l - \beta_l \right)x(yx)^{n - l - 1}y^{2l +2}.        	
	\end{align*}
	Notice that terms appearing in the first sum will never cancel with terms appearing in the last two sums, so
	$\alpha_l = \beta_l$ for all $l$, $0 \leq l \leq n - 1$, and the equality
	\[
        \sum_{l = -1}^{n - 2}\alpha_{l + 1} (n - l - 1) x(yx)^{n - l - 1}y^{2l + 2} = 0.
    \]
    implies that $\alpha_l = 0$ for all $l$, $0 \leq l \leq n - 1$. We deduce that $z = \alpha_nxy^{2n}$.
    The equation $\left[x, z\right] = 0$ implies that $\alpha_n = 0$ or $n = 0$, in which case $z$ belongs to $\field$.
    
    Suppose now that $\degree(z) = 2n + 1$, so that $z = \sum_{l = 0}^{n}\alpha_n (yx)^{n - l}y^{2l + 1} + \sum_{l = 0}^{n}\beta_l x (yx)^{n - l}y^{2l}$. The equality $\left[y ,z\right] = 0$ leads to the following system of equations
    \begin{align}
        &\beta_0 = 0, \label{center_1}\\
        &\beta_l - \alpha_{l - 1} = 0, \text{ for all } l, 1 \leq l \leq n, \label{center_2}\\
        &\alpha_0 n - \beta_0 = 0, \label{center_3}\\
        &\alpha_{l - 1} +  \alpha_l(n - l) -\beta_l = 0, \text{ for all } l, 1 \leq l \leq n - 1 \label{center_4}.
    \end{align}
    From (\ref{center_2}) and (\ref{center_4}), we have $\alpha_l = 0$, for all $l$, $1 \leq l \leq n - 1$, while $\alpha_0 = 0$ is a consequence of (\ref{center_1}) and (\ref{center_3}).
    Finally, (\ref{center_1}) and (\ref{center_2}) imply that $\beta_l = 0$, for all $l$, $0\leq l \leq n$, so that $z = \alpha_n y^{2n + 1}$. The equation
    $[x, z] = 0$ has just one solution of this type, that is $z =0$.
\end{proof}

The particular shape of our double complex, provides isomorphisms
\begin{align*}
	\Hy^{0}(A,A) \cong E_2^{0,0} \text{ and } \Hy^{1}(A,A) \cong E_{2}^{1, 0}.
\end{align*}
Moreover, for each $p \geq 1$, there is a short exact sequence
\begin{align*}
\xymatrix@R=0.5pc{
    0 \ar[r] & E^{2, p - 1}_{2} \ar[r]^(0.4){\iota} & \Hy^{p + 1}(A, A) \ar[r]^(.6){\pi} & E^{1, p}_2 \ar[r] & 0\\
    & \ov{a} \ar@{|->}[r] & \ov{\left(0,a\right)}; \ov{(a, b)} \ar@{|->}[r] & \ov{b}. & 
}
\end{align*}
Extending the basis of $\iota\left(E_2^{2, p-1}\right)$ by a set of linear independent elements such that their images by $\pi$ form a basis of $E_{2}^{1, p}$,
we get a basis of $\Hy^{p + 1}(A, A)$. In this way we have proved the following theorem.
\begin{theorem}\label{cohomology_theorem}
	There are isomorphisms
	 \begin{align*}
        &\Hy^0(A,A) \cong \Bbbk, \\
  		&\Hy^1(A,A) \cong \left\langle c, s_n : n \geq 0\right\rangle,\\
  		&\Hy^{2p}(A) \cong \left\langle t_n^{2p}, u_n^{2p}: n \geq 0\right\rangle, \text{ for all } p > 0,\\
  		&\Hy^{2p + 1}(A) \cong \left\langle v_n^{2p + 1}, w_{n}^{2p + 1} : n \geq 0\right\rangle, \text{ for all } p > 0,
    \end{align*}
    where
    \begin{align*}
	&c = \left(0,\ov{x}\right), s_n = \left((2n + 1)\ov{xy^{2n}}, \ov{y^{2n + 1}}\right); \\
	&t_n^{2p} = \left(0, \ov{xy^{2n}}\right), u_n^{2p} = \left(\sum_{i = 0}^{n}\frac{n!}{i!}\ov{(yx)^{n - i}y^{2i}},
		-\ov{y^{2n + 1}}\right);\\
	&v_n^{2p + 1} = \left(0, \sum_{i = 0}^{n}\frac{n!}{i!}\ov{(yx)^{n - i}y^{2i}}\right), w_n^{2p + 1} = \left(\ov{xy^{2n}}, \ov{xy^{2n + 1}}\right).
\end{align*}
\end{theorem}

\begin{remark}
	\begin{enumerate}[(1)]
	\item The element $s_0 \in \Hy^1(A,A)$ is the class of the eulerian derivation associated to the grading such that
	$\degree(x) = 1 = \degree(y)$.
	\item Given $p, q \in \NN$, $\Hy^{2p}(A,A)$ is isomorphic to $\Hy^{2q}(A,A)$. The isomorphism is induced by an isomorphism
	between the $A^e$-modules $P_{2p}A$ and $P_{2q}A$, namely,
	\begin{align*}
 	\xymatrix@=1em{
		&P_{p}A \ar[r] & P_{q}A &\\
	 	1\ox x^{2p}\ox 1 \ar@{|->}[r] & 1\ox  x^{2q} \ox 1, & 1\ox y^2x^{2p - 1} \ox 1 \ar@{|->}[r] & 1\ox y^2x^{2q - 1} \ox 1.\\
 	}
	\end{align*}
	The situation for the odd cohomology spaces is similar.
	\end{enumerate}
\end{remark}

In Section \ref{product} we will describe the multiplicative structure of $\Hy^{\bullet}(A,A) = \oplus_{n \geq 0}\Hy^{n}(A,A)$
and we will prove that, as it happens for $\field[x]/\left(x^2\right)$, the above isomorphisms are given by the cup product
with a basis element of $\Hy^{2}(A,A)$. 
\section{Hochschild homology}
\label{homology}
\addcontentsline{toc}{chapter}{\nameref{homology}}
Now we compute the Hochschild homology of the super Jordan plane.
Applying the functor $A \ox_{A^e} -$ to the resolution \ref{projective_resolution}
and using the canonical isomorphisms of vector spaces $A\ox_{A^e}(A\ox W \ox A) \cong A \ox W$,
we get the following double complex
\begin{align*}
\xymatrix@=2em{
	& \ar[d]^{\partial} & \ar[d]^{-\partial} \\
	& A \ar[d]^{\delta} & A\ar[l]_{d} \ar[d]^{-\delta} \\
	& A \ar[d]^{\partial} & A \ar[l]_{d} \ar[d]^{-\partial} \\
	& A \ar[d]^{\overline{\delta}} & A \ar[l]_{d} \ar[d]^{-\delta} \\
	A & A \oplus A \ar[l]_{d_0} & A \ar[l]_{d_1} \\
}
\end{align*}
with differentials
\begin{align*}
d_0(a,b) &= \left[a,x\right] + \left[b,y\right], &\overline{\delta}(a) &= (ax + xa, 0)\\
d_1(a) &= \left([a,y^2] - (axy + yxa), [x,a]y + y\left[x,a\right] - xax\right), &\delta(a) &= ax + xa,\\
d(a) &= [a,y^2] - (axy + yxa), &\partial(a) &= \left[a, x\right].
\end{align*}
As we have done for cohomology, we shall use the spectral sequence obtained by filtering by columns in order
to compute the homology. We omit the proof of the next proposition, since it is similar to what we
have already proved in the cohomological case.
\begin{proposition}\label{homology_proposition_ima_delta_partial}
The following sets are bases of $\Ima \delta$, $\Ima \ov{\delta}$ and $\Ima \partial$, respectively:
\begin{itemize}
	\item $\left\{\sum_{i = 0}^k\frac{k!}{i!}(yx)^{b + k - i + 1}y^{2i} + x(yx)^by^{2k + 1} :b, k \geq 0 \right\}$,
	\item $\left\{ \left(x(yx)^by^{2k}, 0\right),
			\left(\sum_{i = 0}^k\frac{k!}{i!}(yx)^{b + k - i + 1}y^{2i} + x(yx)^by^{2k + 1},0\right) :b, k \geq 0 \right\}$,
    \item $\left\{ x(yx)^{b + 1}y^{2k},
			\sum_{i = 0}^k\frac{k!}{i!}(yx)^{b + k - i + 1}y^{2i} - x(yx)^by^{2k + 1} :b, k \geq 0 \right\}$.
\end{itemize}
\end{proposition}
We are now ready to describe $E_{1, 0}^1$. We use again brackets to denote the class in $E^1$
of an element in $E^0$.
\begin{proposition} 
	The vector space $E_{1,0}^1 = \frac{A \oplus A}{\Ima \overline{\delta} }$
	has a basis
	\begin{align*}
		\left\{ \left[((yx)^by^c, 0)\right]\ :\ b,c \geq 0\right\}
			\cup \left\{ \left[(0, x^a(yx)^by^c)\right]\ :\ a \in \{0, 1\}\ y\ b,c \geq 0\right\}. 
	\end{align*}
\end{proposition}
\begin{proof}
The description of $\Ima \ov{\delta}$ given in Proposition \ref{homology_proposition_ima_delta_partial} tells us that
\begin{align*}
\left[(x(yx)^by^{2k}, 0)\right] = 0 \text{ and }
	\left[(-x(yx)^by^{2k + 1}, 0)\right] = \sum_{i = 0}^{k}\frac{k!}{i!}\left[(yx)^{b + k - i + 1}y^{2i}, 0)\right],
\end{align*}
so the given set generates $E_{1, 0}^1$, and it is easy to see that it is linearly independent.
\end{proof}
From now on, we will identify $\frac{A \oplus A}{\Ima \ov{\delta}}$ with $\frac{A}{\Ima \delta} \oplus A$
and we will write $([a],[b])$ instead of $[(a, b)]$ for $(a, b) \in E_{1, 0}^0$.
\begin{proposition}
	The set $\left\{ \left[(yx)^by^c\right] \ :\ b,c \geq 0\right\}$ is a basis of $E_{2,0}^1 = \frac{A}{\Ima(\delta)}$.
\end{proposition}
\begin{proof}
	It is similar to the proof of the previous proposition.
\end{proof}
We will next describe completely $E_{\bullet, n}^1$ for $n > 0$. There are, as before, two different cases, namely $n$ odd and $n$ even. We
will deal with each of them in the following two propositions.
\begin{proposition}
	The set $\left\{ \left[xy^{2n}\right] \ :\ n \geq 0\right\}$ is a basis of $E_{1,2i + 1}^1 \cong E_{2, 2i + 1}^1$,
	for all $i \geq 0$.
\end{proposition} 
\begin{proof}
	Fix $i \geq 0$. The isomorphism $E_{1,2i + 1}^1 \cong E_{2, 2i + 1}^1$ is a consequence of the following facts:
	\begin{itemize}
		\item the kernel and the image of $\delta:E^{0}_{1,2i + 1} \to E^{0}_{1, 2i}$ are isomorphic,
		respectively, to the kernel and the image of $-\delta:E^{0}_{2,2i + 1} \to E^{0}_{2, 2i}$,
		\item the same happens for $\partial$ and $-\partial$,
		\item $\Ker \delta \cong \Ker \ov{\delta}$.
	\end{itemize}
	Let us now describe $\Ker \delta$. Given $z \in \Ker \delta$, we suppose without loss of generality
	that $z$ is homogeneous. Two cases arise, according to the parity of $\degree(z)$.
	
	\underline{$1^{\text{st}}$ case}: $\degree(z) = 2n$. We have
	$z = \sum_{l = 0}^n\alpha_l(yx)^{n - l}y^{2l} + \sum_{i = 1}^n\beta_ix(yx)^{n - i}y^{2i - 1}.$ We get rid
	of the second sum since, by Proposition \ref{homology_proposition_ima_delta_partial}, we know that
	$x(yx)^{n - i}y^{2i - 1} - \sum_{j = 0}^{i - 1}\frac{(i-1)!}{j!}(yx)^{n - j}y^{2j}$ belongs to $\Ima \partial$
	for $i$ such that $1 \leq i \leq n$. Now,
	\begin{align*}
			0 &= \delta(z) = \delta\left(\sum_{l = 0}^n\alpha_l(yx)^{n - l}y^{2l}\right) = \sum_{l = 0}^n\alpha_l\delta((yx)^{n - l}y^{2l}) \\
			&= \sum_{l = 0}^{n - 1}\alpha_l\delta((yx)^{n - l}y^{2l}) + \alpha_n\delta(y^{2n}) \\
				&= \sum_{l = 0}^{n - 1}\alpha_lx(yx)^{n - l}y^{2l} + \alpha_n\sum_{l = 0}^n \frac{n!}{l!}x(yx)^{n - l}y^{2l} + \alpha_nxy^{2n} \\
				&= \sum_{l = 0}^{n - 1}(\alpha_l + \alpha_n\frac{n!}{l!}) x(yx)^{n - l}y^{2l} + 2\alpha_nxy^{2n}.
		\end{align*}	
		Looking at the monomial $xy^{2n}$, we obtain that $\alpha_n = 0$ and so $\sum_{l = 0}^{n - 1}\alpha_lx(yx)^{n - l}y^{2l} = 0$.
		But the monomials appearing in the sum are linearly independent, so $\alpha_l = 0$, for all $l$, that is $z = 0$.

        \bigskip
		\underline{$2^{\text{nd}}$ case}: $\degree(z) = 2n + 1$.
		We have $z = \sum_{l = 0}^n\alpha_l(yx)^{n - l}y^{2l + 1} + \sum_{l = 0}^n\beta_lx(yx)^{n - l}y^{2l}.$
		Proposition \ref{homology_proposition_ima_delta_partial} allows us to suppose that in fact,
		$z = \sum_{l = 0}^n\alpha_l(yx)^{n - l}y^{2l + 1} + \beta xy^{2n}$. Now,
		\begin{align*}
			0 &= \delta(z) = \sum_{l = 0}^n\alpha_l\delta((yx)^{n - l}y^{2l + 1}) + \beta \delta(xy^{2n}) \\
			&= \sum_{l = 0}^n\alpha_l \left( \sum_{i = 0}^l\frac{l!}{i!}(yx)^{n + 1 - i}y^{2i} + x(yx)^{n - l}y^{2l + 1} \right) \\
			&= \sum_{l = 0}^n \sum_{i = 0}^l\alpha_l\frac{l!}{i!}(yx)^{n + 1 - i}y^{2i}
				+ \sum_{l = 0}^n\alpha_l x(yx)^{n - l}y^{2l + 1}.
		\end{align*}	
		From the second sum we obtain that $\alpha_l = 0$ for $0 \leq l \leq n$ and $z = \beta xy^{2n}$.
		The only thing left to prove is that the set $\{\left[xy^{2n}\right] : n \geq 0\}$ is linearly independent.
		Suppose that $[w] \in \frac{\Ker \delta }{\Ima \partial}$ is a linear combination of these elements and that
		$[w] = 0$, in other words $w \in \Ima \partial$. We may suppose that $w$ is homogeneous, that is $w = \lambda xy^{2n}$
		for some $\lambda \in \field$, $n \geq 0$. The description of $\Ima \partial$ implies $\lambda = 0$. 
\end{proof}
\begin{proposition}\label{homology_proposition_even_degree}
	For all $i \geq 1$, the vector spaces $E_{1, 2i}^1$ and $E_{2, 2i}^1$ are isomorphic with basis
	\[	
		\left\{ \sum_{l = 0}^n \frac{n!}{l!}\left[(yx)^{n - l}y^{2l}\right] \ :\ n \geq 0\right\}.
	\]
\end{proposition}
\begin{proof}
	The proof of the isomorphism $E_{1, 2i}^1 \cong E_{2, 2i}^1$ is similar to the previous one. As before,
	given $z \in \Ker \partial$, we may assume that it is homogeneous and split the proof in two cases.
	
	\underline{$1^{\text{st}}$ case}: $\degree(z) = 2n$. Reducing modulo boundaries we may write
	$z = \sum_{l = 0}^n\alpha_l(yx)^{n - l}y^{2l}$, the condition $z \in \Ker \partial$ gives
	\begin{align*}
		0 &= \sum_{l = 0}^n\alpha_l\partial((yx)^{n - l}y^{2l})
			= \sum_{l = 0}^{n - 1}\alpha_l\partial((yx)^{n - l}y^{2l}) + \alpha_n\partial(y^{2n})\\
		&= \sum_{l = 0}^{n - 1}\alpha_l(-x(yx)^{n - l}y^{2l}) + \alpha_n\sum_{l = 0}^{n - 1}\frac{n!}{l!}x(yx)^{n - l}y^{2l}
			= \sum_{l = 0}^{n - 1}(-\alpha_l + \alpha_n\frac{n!}{l!})x(yx)^{n - l}y^{2l},
	\end{align*}
	and we conclude that $\alpha_l = \alpha_n\frac{n!}{l!}$ for all $l$, $0 \leq l \leq n$
	so that
	\[
		z = \sum_{l = 0}^n\alpha_l(yx)^{n - l}y^{2l} = \alpha_n\sum_{l = 0}^n\frac{n!}{l!}(yx)^{n - l}y^{2l}.
	\]
	\bigskip
	\underline{$2^{\text{nd}}$ case}: $\degree(z) = 2n + 1$.
	The description of $\Ima \delta$ in Proposition \ref{homology_proposition_ima_delta_partial} allows us to suppose
	that $z = \sum_{l = 0}^n\alpha_l(yx)^{n - l}y^{2l + 1}$, and since $z \in \Ker \partial$, we have
	\begin{align*}
		0 &= \sum_{l = 0}^n\alpha_l\partial((yx)^{n - l}y^{2l + 1})
			= \sum_{l = 0}^n\alpha_l\left(\sum_{i = 0}^l\frac{l!}{i!}(yx)^{n - i + 1}y^{2i} -x(yx)^{n - l}y^{2l + 1}\right)\\
		&= \sum_{l = 0}^n\sum_{i = 0}^l\alpha_l\frac{l!}{i!}(yx)^{n - i + 1}y^{2i} - \sum_{l = 0}^n\alpha_lx(yx)^{n - l}y^{2l + 1}.
	\end{align*}
	The first and the last sum should be zero separately, so $\alpha_l = 0$ for all $l$, $0 \leq l \leq n$.
	The set
	\[
		\left\lbrace \sum_{l = 0}^n \frac{n!}{l!}\left[(yx)^{n - l}y^{2l}\right] \ :\ n \geq 0\right\rbrace
	\]
	generates $E_{1,2i}^1 \cong E_{2, 2i}^1$. Suppose now that $[w]$ is a null linear combination of these elements.
	The degree of $w$ is even. If $w \in \Ima \delta$, then $w$ should be a linear combination of elements of type
	$\sum_{i = 0}^k\frac{k!}{i!}(yx)^{b + k - i + 1}y^{2i} + x(yx)^by^{2k + 1} $, but this is impossible.
\end{proof}
Of course $E_{0,0}^1 = A$, so the first page of our spectral sequence may be pictured as follows, where $d_0^{(1)}$, $d_{1}^{(1)}$
and $d^{1}$ are the morphisms induced by $d_0$, $d_{1}$ and $d$.
\begin{align*}
\xymatrix{
	& & \\
	& \left\langle \left[xy^{2n}\right] \right\rangle \ar@{--}[u] & \left\langle \left[xy^{2n}\right] \right\rangle \ar[l]_{d^{(1)}} \ar@{--}[u]\\
	& \left\langle \sum_{l = 0}^n \frac{n!}{l!}\left[(yx)^{n - l}y^{2l}\right] \right\rangle \ar@{--}[u]
		& \left\langle \sum_{l = 0}^n \frac{n!}{l!}\left[(yx)^{n - l}y^{2l}\right] \right\rangle \ar[l]_{d^{(1)}} \ar@{--}[u]\\
	& \left\langle \left[xy^{2n}\right] \right\rangle \ar@{--}[u] & \left\langle \left[xy^{2n}\right] \right\rangle \ar[l]_{d^{(1)}} \ar@{--}[u]\\
	A & \left\langle \left(\left[(yx)^by^c\right], 0\right) \right\rangle \bigoplus A \ar[l]_(.7){d_0^{(1)}} \ar@{--}[u]
		& \left\langle \left[(yx)^by^c\right] \right\rangle \ar[l]_{d_1^{(1)}} \ar@{--}[u] \\
}
\end{align*}

Again, the second page will be the last one. We will next compute it, starting by describing the images of
$d_0^{(1)}$, $d_{1}^{(1)}$ and $d^{1}$. Knowing the images, we will provide bases of $E_{\bullet, n}^2$ for all $n \geq 1$.
Finally, we will describe $E_{i, 0}^2$ for $i = 0, 1, 2$. We will write $\ov{a}$ for the class in $E_{\bullet, \bullet}^2$ of
an element $[a] \in E_{\bullet,\bullet}^1$.
\begin{proposition} \label{homology_proposition_d1}
    For all $n \geq 1$, the map $d^{(1)}: E_{2, n}^1 \to E_{1, n}^1$ is zero.
\end{proposition}
\begin{proof}
	We shall consider separately the cases $n$ odd and $n$ even.
	
	\underline{$1^{\text{st}}$ case}: $n$ odd. 
		Given a basis element $\left[xy^{2n}\right]$ of $E_{2, 2k + 1}^1$,
	\begin{align*}
		d\left(xy^{2n}\right) &= \left[ xy^{2n}, y^2 \right] - xy^{2n}xy
			= \left[ xy^{2n}, y^2 \right] - x\sum_{l = 0}^n\frac{n!}{l!}x(yx)^{n - l}y^{2l + 1}  \\
		&= \left[ xy^{2n}, y^2 \right] = \left[x, y^2 \right]y^{2n} + x\left[y^{2n}, y^2\right] =  \left[x, y^2 \right]y^{2n} \\
		&= \left(xy^2 - y^2x\right)y^{2n} = (-xyx)y^{2n}. 
	\end{align*}
	This last element belongs to $\Ima \partial$ and so $d^{(1)}\left(\left[xy^{2n}\right]\right) = 0$.
	
	\bigskip
	\underline{$2^{\text{nd}}$ case}: $n$ even. Fix $w = \sum_{l = 0}^n \frac{n!}{l!}(yx)^{n - l}y^{2l}$ such that 
	$[w]$ belongs to the basis of $E_{2,2k}^1$. If $n = 0$, then
	$d(w) = d(1) = \left[1, y^2 \right] - (xy + yx) = -xy -yx$ and $d(w) \in \Ima \delta$. Suppose now that $n > 0$
	and write $n = m + 1$ with $m \geq 0$. Using the commutation rules we write $w = (m + 1)y^{2m + 1}x + y^{2m + 2}$.
	So, $d(w) = d((m + 1)y^{2m + 1}x + y^{2m + 2}) = (m + 1)d(y^{2m + 1}x) + d(y^{2m + 2})$. Computing each
	term separately we get $d(y^{2m + 1}x) = -\sum_{j = 0}^m(m - j + 2)\frac{m!}{j!}(yx)^{m -j + 2}y^{2j}$
	and $d(y^{2m + 2}) = -\sum_{l = 0}^{m + 1}\frac{(m + 1)!}{l!}x(yx)^{m - l + 1}y^{2l + 1} - yxy^{2m + 2}$.
	Therefore,  
	\begin{align*}
		d^{(1)}(\left[w\right]) &=
				-\sum_{j = 0}^m(m - j + 2)\frac{(m + 1)!}{j!}\left[(yx)^{m -j + 2}y^{2j}\right]\\
			&\qquad -\sum_{l = 0}^{m + 1}\frac{(m + 1)!}{l!}\left[x(yx)^{m - l + 1}y^{2l + 1}\right]
				- \left[yxy^{2m + 2}\right]\\
		&= -\sum_{j = 0}^{m + 1}(m - j + 2)\frac{(m + 1)!}{j!}\left[(yx)^{m -j + 2}y^{2j}\right]\\
		& \qquad-\sum_{l = 0}^{m + 1}\frac{(m + 1)!}{l!}\left[x(yx)^{m - l + 1}y^{2l + 1}\right].	
	\end{align*}
	Since $\sum_{i = 0}^k\frac{k!}{i!}(yx)^{b + k - i + 1}y^{2i} + x(yx)^by^{2k + 1} \in \Ima \delta$ for all
	$b, k \geq 0$, we choose $k = l$ and $b = m - l + 1$ and we obtain that
	$-\left[x(yx)^{m - l + 1}y^{2l + 1}\right]$ equals $\sum_{j = 0}^l\frac{l!}{j!}\left[(yx)^{m - j + 2}y^{2j}\right]$,
	and using this equality we conclude that $d^{(1)}([w]) = 0$.
\end{proof}
Next we apply $d_1^{(1)}$ to the elements in the basis of $E_{2, 0}^1$. For this, let us denote by $f$ and $g$ the compositions
of $d_1$ with the first and second canonical projections, respectively, so that
\[
	d_1(a) =  \left(\left[a,y^2\right] - (axy + yxa), \left[x,a\right]y + y\left[x,a\right] - xax\right)
			= (f(a), g(a)), \text{ for all } a \in A.
\]
We will also denote by $f$ and $g$ the induced maps in homology. Given $b, c \geq  0$ and $z = (yx)^by^c$,
\begin{align*}
	f(z) &= \left[(yx)^by^c, y^2\right] - \left((yx)^by^cxy + yx(yx)^by^c\right)\\
	&= \left[(yx)^b, y^2\right]y^c - (yx)^b(y^cx)y - (yx)^{b + 1}y^c \\
	&= (yx)^by^{c + 2} - y^2(yx)^by^c - (yx)^b(y^cx)y - (yx)^{b + 1}y^c \\
	&= (yx)^by^{c + 2} - ((yx)^by^2 + b(yx)^{b + 1})y^c - (yx)^b(y^cx)y - (yx)^{b + 1}y^c \\
	&= - (b + 1)(yx)^{b + 1}y^c - (yx)^b(y^cx)y.
\end{align*}
There are two cases: $c$ even and $c$ odd. If $c = 2k$, then
\[
	f([z]) = \left\{\begin{array}{cc}
			\sum_{l = 0}^{k - 1}(k - l + 1)\frac{k!}{l!}\left[(yx)^{k - l + 1}y^{2l}\right] &\text{ if } b = 0,\\
			&\\
			-(b + 1)\left[(yx)^{b + 1}y^{2k}\right] &\text{ if } b \geq 1.
	\end{array}\right.
\]
If $c = 2k + 1$, then
$f([z]) = -(b + 2)\left[(yx)^{b + 1}y^{2k + 1}\right]
	- \sum_{i  = 0}^{k - 1}\frac{k!}{i!}\left[(yx)^{k + b - i + 1}y^{2i + 1}\right]$.

\bigskip
For $b, c$ and $z$ as before, we compute now the image of $g$, obtaining that
\[
	g(z) = \left\{\begin{array}{cc}
			- \sum_{i = 0}^{k - 1}\frac{k!}{i!}x(yx)^{k - i}y^{2i + 1}
					-  \sum_{i = 0}^{k - 1}\frac{k!}{i!}(yx)^{k + 1 - i}y^{2i} &\text{ if } c = 2k, b = 0,\\
			&\\
			x(yx)^by^{2k + 1} + (yx)^{b + 1}y^{2k} &\text{ if } c = 2k, b \geq 1,\\
			&\\
			- \sum_{i = 0}^{k - 1}\frac{k!}{i!}(yx)^{k + 1 - i}y^{2i + 1}
                    - \sum_{i = 0}^{k}\frac{(k + 2)k!}{i!}x(yx)^{k + 1 -i}y^{2i} &\text{ if } c = 2k + 1, b \geq 0,\\
			&\\
			-\sum_{i = 0}^{k - 1}\frac{k!}{i!}(yx)^{b + k + 1 - i}y^{2i + 1}
                    - \sum_{i = 0}^{k}\frac{(k + b + 2)k!}{i!}x(yx)^{b + k + 1 - i}y^{2i} &\text{ if }  c = 2k + 1, b \geq 1.
	\end{array}\right.
\]

We introduce some notation
\begin{align*}
    \theta^{1}_{b, k} &= -(b + 1)\left[(yx)^{b + 1}y^{2k}\right],\\
    \theta^{2}_{b, k} &= \left[x(yx)^by^{2k + 1}\right] + \left[(yx)^{b + 1}y^{2k}\right],\\
    \lambda^{1}_{b, k} &= (b + 1)\left[(yx)^{b}y^{2k + 1}\right]
        +\sum_{i  = 0}^{k - 1}\frac{k!}{i!}\left[(yx)^{k + b - i}y^{2i + 1}\right]\\
    \lambda^{2}_{b, k} &= \sum_{i = 0}^{k - 1}\frac{k!}{i!}\left[(yx)^{b + k - i}y^{2i + 1}\right]
        + \sum_{i = 0}^{k}\frac{(k + b + 1)k!}{i!}\left[x(yx)^{b + k - i}y^{2i}\right].
\end{align*}
\begin{proposition}\label{homology_proposition_imad11}
    The set $\left\lbrace\left(\theta^{1}_{b, k}, \theta^{2}_{b, k}\right),
        \left(\lambda^{1}_{b, k}, \lambda^{2}_{b, k}\right) \mid k \geq 0, b\geq 1\right\rbrace$
        is a basis of $\Ima d_1^{(1)}$.
\end{proposition}
\begin{proof}
Given $k \geq 0$, consider the elements
\begin{align*}
    \eta^{1}_k &= \sum_{l = 0}^{k - 1}(k - l + 1)\frac{k!}{l!}\left[(yx)^{k - l + 1}y^{2l}\right],\\
    \eta^{2}_k &= -\sum_{i = 0}^{k - 1}\frac{k!}{i!}\left[x(yx)^{k - i}y^{2i + 1}\right]
        - \sum_{i = 0}^{k - 1}\frac{k!}{i!}\left[(yx)^{k + 1 - i}y^{2i}\right].
\end{align*}
We already know that 
    $\left\lbrace \left(\theta^{1}_{b, k}, \theta^{2}_{b, k}\right),
            \left(\lambda^{1}_{b, k}, \lambda^{2}_{b, k}\right),  \left(\eta^{1}_{k}, \eta^{2}_{k}\right)
                \mid k \geq 0, b\geq 1\right\rbrace$
generates $\Ima d_1^{(1)}$. Notice that $\eta^1_0 = 0 = \eta^2_0$ and that
$\left(\eta^{1}_k, \eta^{2}_k\right) = \sum_{l = 0}^{k - 1}\left(\theta^{1}_{k - l, l}, \theta^{2}_{k - l, l}\right)$
for $k \geq 1$. Moreover, it is easy to see that the set $\left\lbrace\left(\theta^{1}_{b, k}, \theta^{2}_{b, k}\right),
        \left(\lambda^{1}_{b, k}, \lambda^{2}_{b, k}\right) \mid k \geq 0, b\geq 1\right\rbrace$ is linearly independent
        and we obtain the result.
\end{proof}
Our next step will be to compute a basis of $\Ima d_0^{(1)}$. We recall that $E_{1, 0}^1$ is generated by the set
$ \left\lbrace \left(\left[(yx)^by^c\right], 0\right);
		\left(0, \left[x^a(yx)^by^c\right]\right) : a \in \{0, 1\}, b,c \geq 0\right\rbrace$.
For $z = ((yx)^by^c, 0) \in A \oplus A$,
\[
	d_0(z) = \partial(z) = \left\{\begin{array}{cc}
			\sum_{i = 0}^{k - 1}\frac{k!}{i!}x(yx)^{k - i}y^{2i} &\text{ if } c = 2k \text{ and } b = 0,\\
			&\\
			-x(yx)^by^{2k} &\text{ if } c = 2k \text{ and } b \geq 1,\\
			&\\
			\sum_{i = 0}^k\frac{k!}{i!}(yx)^{k + b + 1 - i}y^{2i} - x(yx)^by^{2k + 1}	&\text{ if } c = 2k + 1.
	\end{array}\right.
\]
while for $z = \left(0, x^a(yx)^by^c\right)$, we know that
$d_0(z) = d_0((0, x^a(yx)^by^c)) = x^a(yx)^by^{c + 1} - yx^a(yx)^{b}y^c$, consequently
\[
	d_0(z) = \left\{\begin{array}{cc}
			x(yx)^by^{c + 1} - (yx)^{b + 1}y^c &\text{ if } a =1,\\
			&\\
			0 &\text{ if } a = 0 = b,\\
			&\\
			(yx)^{b}y^{c + 1} - x(yx)^{b - 1}y^{c + 2} - bx(yx)^{b}y^c	&\text{ if } a = 0 \text{ and } b \geq 1.
	\end{array}\right.
\]
In this way we have obtained a set of generators of $\Ima d_0^{(1)}$, from which we extract a basis.
\begin{proposition} \label{homology_proposition_imad01}
    The set $\left\lbrace \left[x(yx)^{b + 1}y^c\right], \left[(yx)^{b + 2}y^c\right],
    		\left[xy^{c + 1} - (yx)y^c\right] \mid b,c \geq 0 \right\rbrace$    
    is a basis of $\Ima d_0^{(1)}$.
\end{proposition}
\begin{proof}
	Let us write
	\begin{align*}
   		\eta_k &= \sum_{i = 0}^{k - 1}\frac{k!}{i!}x(yx)^{k - i}y^{2i},\\
    		\theta_{b, k} &= x(yx)^{b + 1}y^{2k},\\
    		\lambda_{b, k} &= \sum_{i = 0}^k \frac{k!}{i!}(yx)^{k + b + 1 - i}y^{2i} - x(yx)^by^{2k + 1},\\
    		\mu_{b, c} &= x(yx)^by^{c + 1} - (yx)^{b + 1}y^c,\\
    		\nu_{b, c} &= (yx)^{b + 1}y^{c + 1} - x(yx)^{b}y^{c + 2} -(b + 1)x(yx)^{b + 1}y^c.
	\end{align*}
	Notice that $\mu_{b, c + 1} + \nu_{b, c} = -(b + 1)x (yx)^{b + 1}y^c$ for all $b, c \geq 0$, implying that
	$x(yx)^{b + 1}y^c \in \Ima d_0^{(1)}$ and that
	$x(yx)^{b + 1}y^{c + 1} - \mu_{b + 1, c}$, that equals $(yx)^{b + 2}y^c$, also belongs to $\Ima d_0^{(1)}$.
\end{proof}
We are now ready to describe the second page of the spectral sequence. The proof of the following proposition
is a direct consequence of Propositions \ref{homology_proposition_d1} and \ref{homology_proposition_imad01}.
\begin{proposition}
	\begin{enumerate}[(i)]
		\item For all $i \geq 0$, the set $\left\{ \ov{xy^{2n}} \ :\ n \geq 0\right\}$
		is a basis of $E_{1,2i + 1}^2 \cong E_{2, 2i + 1}^2$.
		\item For all $i \geq 0$, the set $\left\{ \sum_{l = 0}^n \frac{n!}{l!}\ov{(yx)^{n - l}y^{2l}} \ :\ n \geq 0\right\}$ 
		is a basis of $E_{1,2i}^2 \cong E_{2, 2i}^2$.
		\item The set $\left\lbrace \ov{xy^{n}}, \ov{y^n} \ :\ n \geq 0\right\rbrace$ is a basis of $E^{2}_{0, 0}$.
	\end{enumerate}
\end{proposition}
We will now complete the description of $E_{\bullet, \bullet}^2$ by exhibiting bases of $E_{1, 0}^2$ and $E_{2, 0}^2$.
\begin{proposition}
	The set $\left\lbrace \sum_{l = 0}^n \frac{n!}{l!}\ov{(yx)^{n - l}y^{2l}} \ :\ n \geq 0\right\rbrace$ is a basis of $E_{2,0}^2$.
\end{proposition}
\begin{proof}
	Given $[z] \in \Ker d_{1}^{(1)}$, we may suppose that $z$ is homogeneous of degree $r$ and then treat separately the cases
	$r$ even and $r$ odd. Computations similar to those in the proof of Proposition \ref{homology_proposition_even_degree}
	prove that the set we proposed generates $E_{2, 0}^2$
	and they are clearly linearly independent, taking into account their degrees.
\end{proof}
\begin{proposition}
    The set
    \begin{align*}
    \Bigg\lbrace &\sum_{l = 0}^n \frac{n!}{l!}\left(\ov{(yx)^{n - l}y^{2l}}, 0\right), 
        \left(-\ov{(yx)^{2n}}, \ov{xy^{2n + 1} + (yx)y^{2n}}\right), \left(0, \ov{y^{n}}\right),\\
        &\qquad \left(\ov{y^{2n + 1}}, \sum_{l = 0}^n \frac{n!}{l!}\frac{n + 1}{n + 1 - l}\ov{x(yx)^{n - l}y^{2l}} 
            + \sum_{l = 0}^n \frac{n!}{l!}\frac{1}{n - l}\ov{(yx)^{n - l}y^{2l + 1}} \right), n \geq 0 \Bigg\rbrace.
    \end{align*}
    is a basis of $E_{1, 0}^2$.
\end{proposition}
\begin{proof}
	Given $\left([z], [w]\right) \in \Ker d_{0}^{(1)}$, we can still suppose $w$ and $z$ homogeneous of the same degree.
  	Again, we treat separately the even and odd cases.
    \begin{itemize}
        \item Suppose $\degree(z) = \degree(w) = 2n$,
            \[
                ([z], [w]) = \left( \sum_{l = 0}^{n} \alpha_l\left[(yx)^{n - l}y^{2l}\right],
                    \sum_{l = 0}^{n - 1}\beta_l \left[x (yx)^{n - l - 1}y^{2l + 1}\right]
                        +  \sum_{l = 0}^{n} \gamma_l\left[(yx)^{n - l}y^{2l}\right]\right)               
            \]
        so
        \begin{align*}
            0 &= d_{0}^{(1)}([z], [w]) = \alpha_n \sum_{l = 0}^{n - 1}\frac{n!}{l!}\left[x(yx)^{n - l}y^{2l}\right]
               - \sum_{l = 0}^{n - 1}\alpha_l \left[x(yx)^{n - l}y^{2l}\right]\\
            &\qquad  + \sum_{l = 0}^{n - 1}\beta_l \left(\left[x(yx)^{n - l - 1}y^{2l + 2}\right] - \left[(yx)^{n - l}y^{2l + 1} \right] \right)\\
            &\qquad + \sum_{l = 0}^{n - 1}\gamma_l \left(\left[(yx)^{n - l}y^{2l + 1} \right]
                - \left[x(yx)^{n - l - 1}y^{2l + 2}\right] - (n - l)\left[x(yx)^{n - l}y^{2l} \right]\right)\\
            &= \sum_{l = 0}^{n - 1}(\gamma_l - \beta_l)\left[(yx)^{n - l}y^{2l + 1}\right]
                + \sum_{l = 0}^{n - 1}(\beta_l - \gamma_l)\left[ x(yx)^{n - l - 1}y^{2l + 2} \right] \\
            &\qquad + \sum_{l = 0}^{n - 1}\left( \alpha_n \frac{n!}{l!} - \alpha_l -(n - l)\gamma_l \right)\left[x(yx)^{n - l}y^{2l}\right]
        \end{align*}
        from which we know that $\beta_l = \gamma_l$ and
        $\alpha_l = \alpha_n \frac{n!}{l!} - (n - l)\gamma_l$ for all $l$, $0 \leq l \leq n - 1$. We can
        now rewrite $ ([z], [w])$ as follows, using that
        \[
        		\left( -(n - l)\left[(yx)^{n - l}y^{2l}\right],
                    \left[x(yx)^{n - l - 1}y^{2l + 1}\right] + \left[(yx)^{n - l}y^{2l}\right] \right) \in \Ima d_{1}^{(1)}
        \]
        for all $l$, $0 \leq l \leq n-2$, and therefore
        \begin{align*}
            ([z], [w]) &= \alpha_n \sum_{l = 0}^n\frac{n!}{l!}\left( \left[(yx)^{n - l}y^{2l} \right], 0 \right)
                + \gamma_n\left(0, \left[y^{2n}\right]\right)\\
            &\qquad + \gamma_{n - 1} \left( -\left[(yx)y^{2n - 2}\right],
                    \left[xy^{2n - 1}\right] + \left[yxy^{2n - 2}\right] \right).       
        \end{align*}
    \item Suppose $\degree(z) = \degree(w) = 2n + 1$, 
    \[
        ([z], [w]) = \left(\sum_{l = 0}^{n}\alpha_l \left[(yx)^{n - l}y^{2l + 1}\right],
            \sum_{l = 0}^{n}\beta_l \left[x(yx)^{n - l}y^{2l}\right]
                + \sum_{l = 0}^{n}\gamma_l\left[(yx)^{n - l}y^{2l + 1}\right]\right),   
    \]
    reducing modulo boundaries
    \begin{align*}
        0 &= d_{0}^{(1)}([z], [w]) = d_0^{(1)}\left(\alpha_n \left[y^{2n + 1}\right],
        			\sum_{l = 0}^{n}\beta_l \left[x(yx)^{n - l}y^{2l}\right]
                + \sum_{l = 0}^{n}\gamma_l\left[(yx)^{n - l}y^{2l + 1}\right] \right)\\
        &= \alpha \sum_{l = 0}^{n}\frac{n!}{l!}\left[(yx)^{n +1 - l}y^{2l}\right] - \alpha \left[xy^{2n + 1}\right]\\
        &\qquad +\sum_{l = 0}^{n}\beta_l \left( \left[x(yx)^{n - l}y^{2l + 1}\right] - \left[(yx)^{n - l + 1}y^{2l} \right] \right)\\
        &\qquad + \sum_{l = 0}^{n - 1}\gamma_l \left( \left[(yx)^{n - l}y^{2l + 2}\right]
            - \left[x(yx)^{n - l - 1}y^{2l + 3}\right] - (n - l)\left[x(yx)^{n - l}y^{2l+ 1}\right]\right)\\
        &= \left(\alpha n! - \beta_0\right)\left[(yx)^{n + 1}\right]
            + \sum_{l = 1}^{n}\left(\frac{n!}{l!}\alpha - \beta_l + \gamma_{l - 1}\right)\left[(yx)^{n + 1 - l}y^{2l}\right]\\
        &\qquad + \left(\beta_0 - \gamma_0 n\right)\left[x(yx)^{n}y\right]
            + \sum_{l = 0}^{n - 1}\left(\beta_l - \gamma_{l - 1} - \gamma_l(n - l)\right)\left[x(yx)^{n - l}y^{2l + 1}\right]\\
        &\qquad + \left(\beta_n - \gamma_{n - 1} - \alpha\right)\left[xy^{2n + 1}\right].
    \end{align*}
    The linear independence of the monomials appearing in the expression gives:
    \begin{align*}
        \beta_0 &= n! \alpha,\\
        \beta_l &= \frac{n!}{l!}\alpha + \gamma_{l - 1} \text{ for all }l, 1 \leq l \leq n,\\
        \beta_l &= \gamma_{l - 1} + \gamma_l(n - l) \text{ for all } l, 1\leq n \leq n - 1,\\
        \gamma_0 &= \frac{1}{n}\beta_0,\\
        \gamma_{n - 1} &= \beta_n - \alpha.     
    \end{align*}
    We deduce that $\gamma_l = \alpha \frac{n!}{l!} \frac{1}{n - l}$ for all $l$, $0 \leq l \leq n - 1$ and
    that $\beta_l  = \alpha \frac{n!}{l!} \frac{n + 1}{n - l + 1}$ for all $l$, $0 \leq l \leq n$. As a consequence,
    \begin{align*}
     ([z], [w]) = &\alpha \Bigg(\left[y^{2n + 1}\right],
            \sum_{l = 0}^{n}\frac{n!}{l!} \frac{n + 1}{n - l + 1}\left[x(yx)^{n -l}y^{2l}\right]
                + \sum_{l = 0}^{n}\frac{n!}{l!}\frac{1}{n - l}\left[(yx)^{n - l}y^{2l + 1}\right]\Bigg)\\
      &\qquad + \gamma_n \left(0, \left[y^{2n + 1}\right]\right).
    \end{align*}
    \end{itemize}
    So our set generates $E_{1, 0}^2$ and it is clearly linearly independent.
\end{proof}
In the homological case we also have a short exact sequence
\begin{align*}
\xymatrix{
    0 \ar[r] & E_{1, p}^{2} \ar[r]^(.4){\iota} & \Hy_{p + 1}(A, A) \ar[r]^{\pi} & E_{2, p - 1}^{2} \ar[r] & 0.
}
\end{align*}
from which we can deduce the proof of the next theorem. Notice that all the homology spaces are infinite dimensional and that starting
from degree 3, they are periodic of period 2.
\begin{theorem} \label{homology_theorem}
	There are isomorphisms:
	\begin{align*}
    		&\Hy_0(A, A) \cong \left\langle \ov{xy^n}, \ov{y^n} \mid n \geq 0 \right\rangle, \\
    		&\Hy_1(A, A) \cong \\
    		    &\qquad\left\langle \left(\ov{y^{2n + 1}},
    			\sum_{i = 0}^n\frac{n!}{i!} \left(\frac{n + 1}{n -(i - 1)}\ov{x(yx)^{n - i}y^{2i}}
    				+ \frac{1}{n - i}\ov{(yx)^{n - i}y^{2i + 1}} \right) \right) \mid n \geq 0 \right\rangle\\
    			&\qquad\oplus \left\langle \left(\sum_{i = 0}^n\frac{n!}{i!}\ov{(yx)^{n - i}y^{2i}}, 0\right),
    				\left(0, \ov{y^{n}}\right), \left(\ov{-yxy^{2n}},\ov{xy^{2n + 1}} + \ov{yxy^{2n}}\right) \mid n \geq 0 \right\rangle, \\
    		&\Hy_2(A, A) \cong \left\langle \left(\ov{xy^{2n}}, 0\right), \left(0, \sum_{i = 0}^n\frac{n!}{i!}\ov{(yx)^{n - i}y^{2i}}\right) \mid n \geq 0 \right\rangle, \\
    		&\Hy_{2p + 1}(A, A) \cong \left\langle \left(\sum_{i = 0}^n\frac{n!}{i!}\ov{(yx)^{n - i}y^{2i}}, 0\right) ,
    			\left(-\ov{(yx)y^{2n}}, \ov{xy^{2n}}\right) \mid n \geq 0 \right\rangle, \\
    		&\Hy_{2p + 2}(A, A) \cong \left\langle \left(\ov{xy^{2n}}, 0 \right) ,
    			\left(\sum_{i = 0}^n\frac{n!}{i!}\ov{(yx)^{n - i}y^{2i + 1}}, \sum_{i = 0}^n\frac{n!}{i!}\ov{(yx)^{n - i}y^{2i}}\right)\mid
    				n \geq 0 \right\rangle. \\
    \end{align*}
\end{theorem}
\section{The ring structure of Hochschild cohomology}
\label{product}
\addcontentsline{toc}{chapter}{\nameref{product}}
We aim to describe the structure of the Hochschild cohomology of the
super Jordan plane as associative graded algebra. By the general
theory we already know that it is graded commutative, and since 
$\Hy^{0}(A,A) \cong \field$ and $\Hy^{1}(A,A)$ is infinite
dimensional we also know that this algebra cannot be finitely
generated.

The description obtained in this section will be also useful
to compute in Section \ref{representations} the Gerstenhaber structure of $\Hy^{\bullet}(A, A)$,
using the fact that the Gerstenhaber bracket is a graded derivation with respect
to the cup product. Of course, since we have obtained $\Hy^{\bullet}(A,A)$
using the minimal resolution of $A$, we need to do some work before actually
computing the cup product. We shall thus construct comparison maps between
the minimal resolution $P_{\bullet}A$ and the bar resolution $B_{\bullet}A$,
that is, morphisms of complexes $f_{\bullet}: P_{\bullet}A \to B_{\bullet}A$
and  $g_{\bullet}: B_{\bullet}A \to P_{\bullet}A$ such that they lift
the identity $P_{0}A \leftrightarrows B_{0}A$. As usual, the map $f_{\bullet}$
is easier to describe since $P_{\bullet}A$ is "smaller".
\begin{proposition}
Let $f_{\bullet} = (f_n)_{n \geq 0}$, $f_n : P_nA \to B_nA$ be the following sequence of morphisms of $A^e$-modules:
	\begin{itemize}
		\item $f_0 : A \ox A \to A \ox A, \quad f_0 = id_{A\ox A},$
		\item $f_1 : A \ox \field \left\lbrace x, y \right\rbrace \ox A \to A \ox A \ox A$, 
				$\quad f_1(1 \ox v \ox 1) = 1 \ox v \ox 1, \text{ for } v \in \field \{x, y\},$
		\item for $n \geq 2 $, $f_n : A \ox \field \left\lbrace x^{n}, y^2x^{n - 1} \right\rbrace \ox A \to A \ox A^{\ox n} \ox A$,
		\begin{align*}
			&f_n(1 \ox x^{n} \ox 1) = 1 \ox x^{\ox n} \ox 1,\\
			&f_n(1 \ox y^2x^{n - 1} \ox 1) = y \ox y \ox x^{\ox n - 1}\ox 1 + 1 \ox y \ox yx \ox x^{ \ox n - 2}\ox 1\\
			&\qquad - x \ox y \ox y \ox x^{\ox n - 2}\ox 1 - 1 \ox x \ox y^{2}\ox x^{\ox n - 2} \ox 1\\
			&\qquad -x \ox y \ox x^{\ox n - 1} \ox 1 - 1 \ox x \ox yx \ox x^{\ox n - 2} \ox 1\\
			&\qquad + \sum_{i = 0}^{n - 3}(-1)^{i}\left(1\ox x^{\ox 2 + i} \ox y^2 \ox x^{\ox n - 3 - i} \ox 1
				+ 1 \ox x^{\ox 2 + i} \ox yx \ox x^{\ox n - 3 - i} \ox 1\right),
		\end{align*}
		where we interpret the last sum as $0$ when $n = 2$.
	\end{itemize}
	The map $f_{\bullet}$ is a morphism of complexes, lifting the identity.
\end{proposition}
\begin{proof}
	We only have to prove that for all $n \geq 0$, $b_n f_{n+1} = f_n d_n$, where $b_{\bullet}$ and $d_{\bullet}$ are the
	differentials of the bar resolution and of the minimal resolution, respectively. The proof is recursive and it consists of a straightforward
	computation that we omit.
\end{proof}
Notice that the formulas of the components $f_n$ of $f_{\bullet}$ come from the terms
appearing when applying the rewriting rules to the elements of $\mathcal{A}_n$, replacing products
by tensors.

We will not provide the complete description of the maps $g_n$ such that $g_{\bullet}: B_{\bullet}A \to P_{\bullet}A$
is a morphism of complexes of $A^e$-modules lifting the identity, since we do not need them. We will just describe in detail
those expressions which we need for the cup product and the Gerstenhaber bracket. Notice that since
$B_{\bullet}A$ is a free resolution of $A$ as a $A^e$-module, if we define $g_n$ on a linearly independent subset
$B_n \subseteq A\ox A^{\ox n} \ox A$ in such a way that $d_n g_{n + 1} = g_n b_n$ for all $z \in B_{n + 1}$ -
a necessary condition is that $b_n(B_{n + 1}) \subseteq B_n$ -, then it is possible to extend
each $g_n$ to $A\ox A^{\ox n} \ox A$ as a map of $A^e$-modules in such a way that we get a morphism of complexes.
The following proposition provides the maps we need. Notice that the idea for $g_n$ is that we make
the product of the intermediate tensors and we look for elements in $\mathcal{A}_n$, before and after rewriting.
\begin{proposition}
	There is a morphism of complexes $g_{\bullet}: B_{\bullet}A \to P_{\bullet}A$ such that:
	\begin{itemize}
		\item $g_0 = id_{A\ox A}$,
		\item for all $a \in \{0, 1\}, b, c \in \NN_0$,
		\begin{align*}
			&g_1\left(x^{a}(yx)^{b}y^c\right) = a\ox x \ox (yx)^by^c
				+ \sum_{i = 0}^{b - 1} x^{a}(yx)^{i}\ox y \ox x(yx)^{b - 1 - i}y^c\\
			&\qquad + \sum_{i = 0}^{b - 1}x^{a}(yx)^{i}y \ox x \ox (yx)^{b - 1 - i}y^c
				+ \sum_{i =0}^{c - 1}x^{a}(yx)^{b} y^{i}\ox y \ox y^{c - 1 - i},
		\end{align*}
		\item for all $n \geq 2$,
	\begin{itemize}
		\item $g_n\left(1\ox y \ox x^{\ox n - 1}\ox 1\right) = 0$,
		\item $g_n\left(1\ox y \ox yx \ox x^{ \ox n - 2} \ox 1\right) = 1\ox y^2 x^{n - 1} \ox 1$,
		\item $g_n\left(1\ox y \ox y \ox x^{\ox n - 2} \ox 1\right) = 0$,
		\item $g_n\left(1\ox x^{\ox n} \ox 1\right) = 1\ox x^n \ox 1$,
		\item $g_n\left(1\ox y^2 \ox x^{\ox n - 1}\ox 1\right) = 1 \ox y^{2}x^{n - 1 \ox 1}$,
		\item $g_n\left(1\ox x^{\ox i} \ox y^{2} \ox x^{n - 1 - i} \ox 1\right) =  0,\quad$
			for all $i$, $1 \leq i \leq n  - 1$,
		\item $g_n\left(1\ox yx \ox x^{n - 1} \ox 1\right) = y\ox x^{n}\ox 1$,
		\item $g_n\left(1\ox x^{\ox i} \ox yx \ox x^{\ox n - 1 - i}\ox 1\right) = 0,\quad$
			for all $i$, $1 \leq i \leq n  - 1$,
		\item $g_n\left(1 \ox xy^{2} \ox x^{\ox n - 1}\ox 1\right) = x \ox y^2x^{n - 1} \ox 1
			+ 1 \ox x^{n} \ox y^{2} + 1 \ox x^{n} \ox yx$,
		\item $g_n\left(1 \ox x^{\ox i}\ox xy^{2} \ox x^{\ox n - 1 - i}\ox 1\right) =
			1 \ox x^{n} \ox y^{2} + 1 \ox x^{n} \ox yx,\quad$ for all $i$, $1 \leq i \leq n  - 2$,
		\item $g_n\left(1 \ox x^{\ox n - 1}\ox xy^{2} \ox 1\right) = 1 \ox x^{n} \ox y^{2}$,
		\item $g_n\left(1 \ox xyx \ox x^{\ox n - 1}\ox 1\right) = xy \ox x^{n} \ox 1$,
		\item $g_n\left(1 \ox x^{\ox i}\ox xyx \ox x^{\ox n - 1 - i}\ox 1\right) = 0,\quad$
			for all $i$, $1 \leq i \leq n  - 2$,
		\item $g_n\left(1 \ox x^{\ox n - 1}\ox xyx \ox 1\right) = 1 \ox x^{n} \ox yx$.
	\end{itemize}
	\end{itemize}
\end{proposition}
\begin{proof}
The proof is, as before, straightforward.
\end{proof}
The induced morphisms in cohomology, $f^{\bullet}$ and $g^{\bullet}$, are thus inverse
isomorphisms, so the cup product may be computed as follows:
given $\varphi \in \Hom_{A^e}\left(P_n A, A\right)$
and $\phi \in \Hom_{A^e}\left(P_m A, A\right)$, the cup product of their cohomology classes
$\ov{\varphi}$ and $\ov{\phi}$ is:
\[
	\ov{\varphi} \smile \ov{\phi} = \ov{\left(\varphi g_n \smile \phi g_m\right) f_{n + m}}.
\]
From now on we omit the bar in the notation of cohomology classes, since we will always work in cohomology. The differential in
$\Hom_{A^e}\left(P_{\bullet}A, A\right)$ will be denoted as usual by $d^{\bullet}$. 

Of course, if $\lambda \in \Hy^{0}(A,A)$ and $\varphi \in  \Hy^{n}(A,A)$, then
$\varphi \smile \lambda = \lambda \smile \varphi = \lambda \varphi$. Suppose now that $n, m \geq 1$
and that $\varphi \in \Hy^{n}(A,A)$ and $\phi \in \Hy^{m}(A,A)$. Their cup product
$\varphi \smile \phi$ belongs to $\Hy^{n + m}(A,A)$, and since the cup product is graded commutative,
we know that $\varphi \smile \phi = (-1)^{nm}\phi \smile \varphi$. In order to
write these elements in terms of a basis of $\Hy^{n + m}(A,A)$ we need to know their
value on $1\ox x^{n + m}\ox 1$ and on $1 \ox y^2x^{n + m - 1}\ox 1$. In each case
we will use either $\varphi \smile \phi$ of $\phi \smile \varphi$, according to our
convenience.

Let us start by describing the product of $c\in \Hy^{1}(A,A)$ with all the other
generators. The graded commutativity implies $c^2 = 0$, we will see that in fact all the products involving $c$
are zero.

We define a family of cochains that we will use in several of the proofs of the following propositions. For $r  \geq 1$
let $\alpha_{m - i, j}^{r} \in \Hom_{A^{e}}(P_{r}A, A)$ be defined by
\[
	\alpha_{m - i, j}^{r}\left(1 \ox x^r \ox 1\right) = 0,\quad
		\alpha_{m - i, j}^{r}\left(1 \ox y^2x^{r - 1} \ox 1\right) = -(yx)^{m - i}y^{2j}.
\]

\begin{proposition}
	The product of $c$  with any other generator of $\Hy^{\bullet}(A,A)$ is zero.
\end{proposition}
\begin{proof}
	We will denote with the same symbol the cohomology class and its representatives.
	
	\bigskip
	$\boldsymbol{1) c\smile s_n:}$
	\begin{align*}
	c &\smile s_n\left(1\ox x^{2} \ox 1\right) = \left(c g_1 \smile s_n g_1\right)f_2\left(1\ox x^2 \ox 1\right)
		= \left(c g_1 \smile s_n g_1\right)(1 \ox x^{\ox 2}\ox 1)\\
	&= c(1\ox x \ox 1)s_n(1 \ox x \ox 1) = 0 (2n + 1)xy^{2n} = 0,
\end{align*}
\begin{align*}
	c &\smile s_n\left(1\ox y^2x \ox 1\right) = \left(c g_1 \smile s_n g_1\right)f_2\left(1\ox y^2x \ox 1\right)\\
	&= \left(c g_1 \smile s_n g_1\right)\bigg(y \ox y \ox x \ox 1 + 1 \ox y \ox yx \ox 1 
		- x \ox y \ox y \ox 1\\ &\qquad - 1\ox x \ox y^2 \ox 1 - x \ox y \ox x \ox 1 
		- 1\ox x \ox yx \ox 1\bigg) \\
	&= y c(1 \ox y \ox 1)s_n(1 \ox x \ox 1) + c(1\ox y \ox 1)s_n(y \ox x \ox 1 + 1 \ox y \ox x)\\
	&\qquad -xc(1\ox y \ox 1)s_n(1 \ox y \ox 1) -c(1\ox x \ox 1)s_n(y \ox y \ox 1 + 1\ox y \ox y)\\
	&\qquad -xc(1\ox y \ox 1)s_n(1 \ox x \ox 1) -c(1\ox x \ox 1)s_n(y \ox x \ox 1 + 1\ox y \ox x)\\
	&= yx(2n + 1)xy^{2n} + x\left(y(2n + 1)xy^{2n} + y^{2n + 1}x\right) - x^2y^{2n + 1} \\
	&\qquad - 0\left(y^{2n + 2} + y^{2n + 2}\right) - x^2(2n + 1)xy^{2n}
		- 0\left(y(2n + 1)xy^{2n} + y^{2n + 1} x\right)\\
	&= (2n + 1)xyxy^{2n} + \sum_{i = 0}^{n}\frac{n!}{i!}x(yx)^{n  + 1 -i}y^{2i}.
\end{align*}
We will see that reducing modulo coboundaries, this last expression is zero. Defining
the $1$-cochains $\eta$ and $\nu$ such that
\begin{align*}
	&\eta(1\ox x \ox 1) = -xy^{2n},& &\eta(1\ox y \ox 1) = 0,\\
	&\nu(1\ox x \ox 1) = 0,& &\nu(1 \ox y \ox 1) = y^{2n + 1} - (2n + 1)xy^{2n},
\end{align*}
we obtain that $\left( c \smile s_n - d^1(\eta) - d^1(\nu)\right)\left(1\ox y^2x \ox 1\right) = 0$,
and the computation for $1\ox x^2 \ox 1$ remains valid, since neither $d^1(\eta)$ nor $d^1(\nu)$ touches
$1 \ox x^2 \ox 1$.

\bigskip
$\boldsymbol{2) t_m^{2p} \smile c:}$
\begin{align*}
	t_m^{2p}&\smile c(1 \ox x^{2p + 1} \ox 1) = t_m^{2p}\left(1\ox x^{2p} \ox 1\right)c(1 \ox x \ox 1) = 0
\end{align*}
while	
\begin{align*}
	t_m^{2p}&\smile c(1 \ox y^2x^{2p} \ox 1) = \left(t_m^{2p}g_{2p}\smile cg_1\right)
		\left(f_{2p + 1}\left(1 \ox y^2x^{2p}\right)\right) = 0,
\end{align*}	
since $c(1\ox x \ox 1) = 0 = t_m^{2p}\left(1 \ox x^{2p} \ox 1\right)$.

\bigskip
$\boldsymbol{3) u_m^{2p} \smile c:}$ here we compute directly that
$ u_m\smile c =  \sum_{i =0}^{m}\frac{m!}{i!}d^{2p}\left(\alpha_{m - i, i}^{2p}\right)$.

\bigskip
$\boldsymbol{4) v_m^{2p + 1} \smile c:}$ it is straightforward to verify that it is zero
both on $1\ox x^{2p + 2}\ox 1$ and  on\\ $1\ox y^2x^{2p + 1}\ox 1$.

\bigskip
$\boldsymbol{5) w_m^{2p + 1} \smile c:}$
\begin{align*}
	w_m^{2p + 1} \smile c\left(1 \ox x^{2p + 2} \ox 1\right)
		= w_m^{2p + 1}\left(1 \ox x^{2p + 1} \ox 1\right)c(1 \ox x \ox 1) = 0,
\end{align*}
\begin{align*}
	w_m^{2p + 1}& \smile c\left(1 \ox y^2x^{2p + 1} \ox 1\right)
		= w_m^{2p + 1}\left(1 \ox y^2 x^{2p} \ox 1\right)c(1 \ox x \ox 1)\\
	&\qquad- w_m^{2p + 1}\left(1 \ox x^{2p + 1} \ox 1\right)c(y \ox y \ox 1 + 1 \ox y \ox y)\\
	&\qquad -w_m^{2p + 1}\left(1 \ox x^{2p + 1} \ox 1\right)c(y \ox x \ox 1 + 1 \ox y \ox x)\\
	&= -xy^{2m}(yx+ xy) = \sum_{i = 0}^{m}\frac{m!}{i!}x(yx)^{m + 1 - i}y^{2i}.
\end{align*}
Notice that $w_m^{2p + 1} \smile c =  \sum_{i = 0}^{m}\frac{m!}{i!}d^{2p + 1}\left(\alpha_{m + 1 -i, i}^{2p + 1}\right)$.
\end{proof}
Now we are going to prove that the product of two different elements of the infinite
family of generators of $\Hy^{1}(A,A)$ is a scalar multiple of a generator of $\Hy^{2}(A,A)$.
\begin{proposition}\label{product_proposition_smcupsn}
	For all $n ,m \in \NN_{0}$, $s_m\smile s_n = 4(n - m)t_{n + m + 1}^2$.
\end{proposition}
\begin{proof}
	Since $s_m \in \Hy^{1}(A,A)$, we know that $s_m \smile s_m = 0$. Suppose $m \neq n$,
	\begin{align}
	s_m &\smile s_n\left(1\ox x^{2} \ox 1\right) = \left(s_m g_1 \smile s_n g_1\right)f_2\left(1\ox x^2 \ox 1\right)
	= \left(s_m g_1 \smile s_n g_1\right)\left(1 \ox x^{\ox 2}\ox 1\right)\nonumber\\
	&= s_m(1\ox x \ox 1)s_n(1 \ox x \ox 1) = (2m + 1)xy^{2m} (2n + 1)xy^{2n} = 0.
	\label{sm_sn_x2}
\end{align}
On the other hand,
\begin{align*}
	s_m &\smile s_n\left(1\ox y^2x \ox 1\right) = (4n + 2)\sum_{i = 0}^{m + 1}\frac{(m + 1)!}{i!}x(yx)^{m + 1 - i}y^{2(i + n)}\\
	&\qquad+ \sum_{i = 0}^{m + n + 1}\frac{(m + n + 1)!}{i!}x(yx)^{m + n + 1 - i}y^{2i}
			-(4m + 3)xy^{2(m + n + 1)}\\
	&\qquad
		- (2m + 2)(2n + 1)\sum_{i = 0}^{m + 1}\frac{(m + 1)!}{i!}x(yx)^{m + 1 - i}y^{2(i + n)}\\ 
	&\qquad -(2m + 1)\sum_{i = 0}^{m + n}\frac{(m + n)!}{i!}x(yx)^{m + n + 1 - i}y^{2i}.
\end{align*}
Even if this expression is not zero, while trying to write $s_m \smile s_n$ in terms of the basis,
we notice that the same $\nu$ of the previous proof allows us to get rid of terms of type
$x(yx)y^{2i}$ in this last expression by subtracting $d^1(\nu)$ to $s_m 
\smile s_n$ without changing the value on $1\ox x^2 \ox 1$. We still need to deal
with the terms of type $x(yx)^{b + 1}y^{2i}$, without touching (\ref{sm_sn_x2}).
Let $\gamma_{b, i} \in \Hom_{A^e}\left(A\ox \field\{x, y\}\ox A, A\right)$ be such that
\[
	\gamma_{b, i}(1\ox x \ox 1) = 0 \text{ and }
		\gamma_{b, i}(1 \ox y \ox 1) = \frac{1}{b + i}(yx)^{b}y^{2i + 1} - x(yx)^{b}y^{2i},
\]
so 
$d^1(\gamma_{b, i})\left(1\ox x^2 \ox 1\right) = 0
		\text{ and } d^1(\gamma_{b, i})\left(1\ox y^2x \ox 1\right) = x(yx)^{b + 1}y^{2i}$.
Finally, subtracting $d^1\left(\gamma_{b,i}\right)$ for convenient values from $b$ and $i$
to $s_m \smile s_n$ we obtain that $s_m \smile s_n = 4(n - m)t_{n + m + 1}^2$.
\end{proof}
The previous proposition proves that all the generators $t_n^2$ of $\Hy^{2}(A,A)$,
except $t_0^2$ and $t_1^2$ may be obtained from products of generators of $\Hy^{1}(A,A)$.
Before continuing with the cup product we prove a technical lemma.
\begin{lemma}\label{lemma_product_binom}
	For all $n, m\geq 0$, $\sum_{i = 0}^{m}\frac{(n + i)!}{i!} = \frac{(m + n + 1)!}{m!(n + 1)}$.
\end{lemma}
\begin{proof}
	Dividing by $n!$ on both sides, we see that the equality is equivalent to
	$
		\sum_{i = 0}^{m}\binom{n + i}{i} = \binom{m + n + 1}{m},
	$
	which can be proved without difficulty by induction on $m$.
\end{proof}
We now compute the products of a generator $s_n$ with all generators in degree
$2p$ and $2p + 1$.
\begin{proposition}\label{product_proposition_sn}
	Given $n, m\geq 0$ and $p \geq 1:$
	\begin{multicols}{2}
	\begin{enumerate}[(i)]
	\item $t_m^{2p} \smile s_n = 0,$ \label{product_proposition_sn_item}
	\item $u_m^{2p} \smile s_n = (2n + 1)w_{n + m}^{2p + 1} + 2v_{n + m + 1}^{2p + 1}$
		\label{product_proposition_sn_item2},
	\item $v_m^{2p + 1} \smile s_n = (2n + 1)t_{n + m}^{2p + 2}$,
	\item $w_m^{2p + 1} \smile s_n = -2t_{n + m + 1}^{2p + 2}$.
	\end{enumerate}
	\end{multicols}
\end{proposition}
\begin{proof}
	\begin{enumerate}[(i)]
		\item It is straightforward:
		\begin{align*}
			t_m^{2p} &\smile s_n\left(1 \ox x^{2 p + 1}\ox 1\right)
					= t_m^{2p}\left(1 \ox x^{2p} \ox 1\right) s_n(1 \ox x \ox 1) = 0,\\
			t_m^{2p} &\smile s_n\left(1 \ox y^2x^{2 p}\ox 1\right) =
			t_m^{2p}\left(1 \ox y^{2}x^{2p - 1}\ox 1\right)s_n(1 \ox x \ox 1)\\
			&\qquad + t_m^{2p}\left(1 \ox x^{2p} \ox 1\right)s_n(y \ox y \ox 1 + 1 \ox y \ox y)\\
			&\qquad + t_m^{2p}\left(1 \ox x^{2p} \ox 1\right)s_n(y \ox x \ox 1 + 1 \ox y \ox x)\\
			&= xy^{2m}(2n + 1)xy^{2n} = 0.
		\end{align*}
		\item The first part is quite direct, since
		\begin{align*}
			u_m^{2p}& \smile s_n\left(1 \ox x^{2p + 1} \ox 1\right)
					= u_m^{2p}\left(1 \ox x^{2p} \ox 1\right)s_n(1 \ox x \ox 1)\\
			&= \sum_{i = 0}^{m}(yx)^{m  -i}y^{2i}(2n + 1)xy^{2n}
						= (2n + 1)\sum_{i = 0}^{m}\frac{m!}{i!} x(yx)^{m  -i}y^{2(i + n)},
		\end{align*}
		and setting $\eta_i \in \Hom_{A^e}\left(P_{2p}A, A\right)$ for $i$, $0 \leq i < m$,
		such that
		\[
			\eta_i\left(1\ox x^{2p} \ox 1\right) =
					(yx)^{m - i}y^{2(i + n)}
					\text{ and } \eta_i\left(1\ox y^2x^{2p - 1} \ox 1\right) = 0,
		\]
		it turns out that $d^{2p}(\eta_i)\left(1 \ox x^{2p + 1}\ox 1\right) = x(yx)^{m - i}y^{2(i + n)}$, so
		\begin{align*}
			\left(u_m^{2p} \smile s_n - \sum_{i =0}^{m - 1} (2n + 1)\frac{m!}{i!}d^{2p}(\eta_i)\right)
			&\left(1 \ox x^{2p + 1} \ox 1\right) = (2n + 1)xy^{2(n + m)}\\
			&= (2n + 1)w_{n + m}^{2p + 1} \left(1 \ox x^{2p + 1} \ox 1\right).
		\end{align*}
		Notice that
		\begin{align*}
			d^{2p}&(\eta_i)\left(1 \ox y^2x^{2p}\ox 1\right) = (m  - i)(yx)^{m - i + 1}y^{2(i + n)}
				- \sum_{l = 0}^{i + n}\frac{(i + n)!}{l!}(yx)^{m + n + 1 - l}y^{2l},
		\end{align*}
		and so this reduction will also have an effect on $1 \ox y^2x^{2p} \ox 1$.
		Let us now compute the value on this element. We will write $\lambda_l$
		instead of $(yx)^{m + n + 1 - l}y^{2l}$.
		
		\begin{align*}
			u_m^{2p}& \smile s_n\left(1 \ox y^2x^{2p} \ox 1\right)
					= u_m^{2p}\left(1 \ox y^2x^{2p - 1} \ox 1\right)s_n(1 \ox x \ox 1)\\
			&\qquad + u_m^{2p}\left(1 \ox x^{2p} \ox 1\right)s_n(y \ox y \ox 1 + 1 \ox y \ox y)\\
			&\qquad + u_m^{2p}\left(1 \ox x^{2p} \ox 1\right)s_n(y \ox x \ox 1 + 1 \ox x \ox y)\\
			&= -y^{2m + 1}(2n + 1)xy^{2n}
					+ 2\sum_{i = 0}^{m}\frac{m!}{i!}(yx)^{m - i}y^{2i}y^{2n + 2}\\
			&\qquad + \sum_{i = 0}^{m}\frac{m!}{i!}(yx)^{m - i}y^{2i}\left((2n + 1)yxy^{2n}
					+ y^{2n + 1}x\right)\\
			&= 2\sum_{i = 0}^{m}\frac{m!}{i!}(yx)^{m - i}y^{2(i + n + 1)}
					+ \sum_{i = 0}^{m - 1}\frac{(2n + 1)m!}{i!}(yx)^{m - i}y^{2i + 1}xy^{2n}\\
			&\qquad+ \sum_{i = 0}^{m}\frac{m!}{i!}(yx)^{m - i}y^{2(i + n) + 1}x\\
			&= 2\sum_{l = n + 1}^{m + n + 1}\frac{m!}{(l - n - 1)!}\lambda_l 
					+ \sum_{l = n}^{m + n - 1}\frac{(2n + 1)m!(m + n - l)}{(l - n)!}
						\lambda_l
						%\\
			\end{align*}
			
			\begin{align*}
			&\qquad + \sum_{l = 0}^{n}\frac{m!}{l!}\left(
						\sum_{i = 0}^{m}\frac{(i + n!)}{i!}\right)\lambda_l
					+\sum_{l = n + 1}^{m + n}\frac{m!}{l!}\left(
						\sum_{i = l - n}^{m}\frac{(i + n!)}{i!}\right)\lambda_l.
		\end{align*}		
		By Lemma \ref{lemma_product_binom}, we have that
		\begin{align*}
			\sum_{l = 0}^{n} &\frac{m!}{l!}\left(
						\sum_{i = 0}^{m}\frac{(i + n!)}{i!}\right)\lambda_l
					+\sum_{l = n + 1}^{m + n}\frac{m!}{l!}\left(
						\sum_{i = l - n}^{m}\frac{(i + n!)}{i!}\right)\lambda_l\\
			&= \sum_{l = 0}^{n + m}\frac{m!}{l!}
					\left(\frac{(n + m + 1)!}{m!(n + 1)}\right)\lambda_l
						-\sum_{l = n + 1}^{m + n}
					\frac{m!}{l!}\left(\frac{l!}{(l - n - 1)!(n + 1)}\right)\lambda_l,
		\end{align*}
		so
		\begin{align*}
		    u_m^{2p}& \smile s_n\left(1 \ox y^2x^{2p} \ox 1\right) =
		        2\sum_{l = n + 1}^{m + n + 1}\frac{m!}{(l - n - 1)!}\lambda_l
		        +\sum_{l = 0}^{n + m}\frac{(n + m + 1)!}{l!(n + 1)}\lambda_l\\
	        &\qquad+ (2n + 1)\sum_{l = n}^{m + n - 1}\frac{m!}{(l - n)!}(m + n - l)\lambda_l
		        -\sum_{l = n + 1}^{m + n}\frac{m!}{(l - n - 1)!(n + 1)}\lambda_l,
		\end{align*}
		and reducing modulo coboundaries, we obtain
		\begin{align*}
			u_m^{2p}& \smile s_n\left(1 \ox y^2x^{2p} \ox 1\right) =
				\left(u_m^{2p} \smile s_n - \sum_{i =0}^{m - 1}\frac{(2n + 1)m!}{i!} d^{2p}(\eta_i)\right)
				\left(1 \ox y^2x^{2p} \ox 1\right)\\
			&=2\sum_{l = n + 1}^{m + n + 1}\frac{m!}{(l - n - 1)!}\lambda_l
				+\sum_{l = 0}^{n + m}\frac{(n + m + 1)!}{l!(n + 1)}\lambda_l\\
			&\qquad -\sum_{l = n + 1}^{m + n}\frac{m!}{(l - n - 1)!(n + 1)}\lambda_l
				+(2n + 1)\sum_{i = 0}^{m - 1}\frac{m!}{i!}
						\sum_{l = 0}^{i + n}\frac{(i + n)!}{l!}\lambda_l\\
			&= 2\sum_{l = n + 1}^{m + n + 1}\frac{m!}{(l - n - 1)!}\lambda_l
				+\sum_{l = 0}^{n + m}\frac{(n + m + 1)!}{l!(n + 1)}\lambda_l
	 			-\sum_{l = n + 1}^{m + n}\frac{m!}{(l - n - 1)!(n + 1)}\lambda_l\\
			&\qquad+(2n + 1)\sum_{i = 0}^{m}\sum_{l = 0}^{i + n}
				\frac{m!}{i!}\frac{(i + n)!}{l!}\lambda_l
				-(2n + 1)\sum_{l = 0}^{m + n}\frac{(m + n)!}{l!}\lambda_l\\
			&= 2\sum_{l = n + 1}^{m + n + 1}\frac{m!}{(l - n - 1)!}\lambda_l
				+(2n + 2)\sum_{l = 0}^{n + m}\frac{(n + m + 1)!}{l!(n + 1)}\lambda_l\\
			&\qquad -(2n + 2)\sum_{l = n + 1}^{m + n}\frac{m!}{(l - n - 1)!(n + 1)}\lambda_l
				-(2n + 1)\sum_{l = 0}^{m + n}\frac{(m + n)!}{l!}\lambda_l\\
			&= 2\lambda_{n + m + 1} + (2m + 1)\sum_{l = 0}^{n + m}\frac{(n + m)!}{l!}\lambda_l\\
			&= 2\left(\sum_{l = 0}^{n + m + 1}\frac{(n + m + 1)!}{l!} \lambda_l
				- \sum_{l = 0}^{n + m}\frac{(n + m + 1)!}{l!} \lambda_l\right)
				+ (2m + 1)\sum_{l = 0}^{n + m}\frac{(n + m)!}{l!}\lambda_l\\
			&= 	2\sum_{l = 0}^{n + m + 1}\frac{(n + m + 1)!}{l!} \lambda_l
				- (2n + 1)\sum_{l = 0}^{n + m}\frac{(n + m)!}{l!}\lambda_l.
		\end{align*}
		Choosing $\beta \in \Hom_{A^e}\left(P_{2p}A, A\right)$ such that
		\[
			\beta\left(1\ox x^{2p} \ox 1\right) = 0
				\text{ and } \beta\left(1\ox y^2x^{2p - 1} \ox 1\right) = -y^{2(n + m)+ 1},
		\]
		we have
		\[
			d^{2p}(\beta)(1\ox x^{2p + 1} \ox 1) =0 \text{ and }
			d^{2p}(\beta)(1\ox y^2x^{2p} \ox 1) =
				\sum_{l = 0}^{n + m}\frac{(n + m)!}{l!}\lambda_l + xy^{2(n + m) + 1}
		\]
		and using that $v_{n + m}^{2p + 1}\left(1\ox x^{2p + 1} \ox 1\right) = 0$, we obtain
		\[
			u_m^{2p} \smile s_n = u_m^{2p} \smile s_n
				- \sum_{i =0}^{m - 1}\frac{(2n + 1)m!}{i!} d^{2p}(\eta_i) + d^{2p}((2n + 1)\beta)
				= 2v_{n + m + 1}^{2p + 1} + (2n + 1)w_{n + m}^{2p + 1}.
		\]
		\item Clearly $v_m^{2p + 1} \smile s_n\left(1 \ox x^{2p + 2} \ox 1\right) = 0$.
		We compute
		\begin{align*}
			v_m^{2p + 1}& \smile s_n\left(1 \ox y^2x^{2p + 1} \ox 1\right)
				= v_m^{2p + 1}\left(1 \ox y^2 x^{2p} \ox 1\right)s_n(1 \ox x \ox 1)\\
			&\qquad- v_m^{2p + 1}\left(1 \ox x^{2p + 1} \ox 1\right)
				s_n(y \ox y \ox 1 + 1 \ox y \ox y)\\
			&\qquad -v_m^{2p + 1}\left(1 \ox x^{2p + 1} \ox 1\right)
				s_n(y \ox x \ox 1 + 1 \ox y \ox x)\\
			&= \sum_{i = 0}^{m}\frac{m!}{i!}(yx)^{m - i}y^{2i}(2n + 1)xy^{2n}
				=(2n + 1)\sum_{l = 0}^{m}\frac{m!}{l!}x(yx)^{m - l}y^{2(l + n)},
		\end{align*}
		and we get
		\[
			v_m^{2p + 1} \smile s_n = v_m^{2p + 1} \smile s_n
				-(2n + 1)\sum_{l = 0}^{m - 1}\frac{m!}{i!}d^{2p + 1}
				\left(\alpha_{m - l, l + n}^{2p + 1}\right)
				= (2n + 1)t_{n + m}^{2p + 2}.
		\]
		\item The first part is straightforward, since
		\begin{align*}
			w_m^{2p + 1} &\smile s_n\left(1 \ox x^{2p + 2} \ox 1\right)
				= w_m^{2p + 1}\left(1 \ox x^{2p + 1} \ox 1\right)s_n(1 \ox x \ox 1)\\
			&= xy^{2m}(2n + 1)xy^{2n} = 0.
		\end{align*}	
		On the other hand,
		\begin{align*}
			w_m^{2p + 1}& \smile s_n\left(1 \ox y^2x^{2p + 1} \ox 1\right)
				= w_m^{2p + 1}\left(1 \ox y^2 x^{2p} \ox 1\right)s_n(1 \ox x \ox 1)\\
			&\qquad -w_m^{2p + 1}\left(1 \ox x^{2p + 1} \ox 1\right)
				s_n(y \ox y \ox 1 + 1 \ox y \ox y)\\
			&\qquad -w_m^{2p + 1}\left(1 \ox x^{2p + 1} \ox 1\right)
				s_n(y \ox x \ox 1 + 1 \ox y \ox x)\\
			&= xy^{2m + 1}(2n + 1)xy^{2n} -2xy^{2m}y^{2n + 2}
				- xy^{2m}\left((2n + 1)yxy^{2n} + y^{2n + 1}x\right)\\
			&= -2xy^{2m}y^{2n + 2} -xy^{2(n + m) + 1}x\\
			&= -2xy^{2(m + n + 1)}
				-\sum_{i = 0}^{n +m}\frac{(n + m)!}{i!}x(yx)^{m  + n  + 1 -i}y^{2i}.
		\end{align*}
		Using again the cochains $\alpha_{m - i, j}^{2p + 1}$, we can
		remove the sum from the last expression. Thus, 
		\[
			w_m^{2p + 1} \smile s_n = -2t_{n + m + 1}^{2p + 2}.
		\]		 
	\end{enumerate}
\end{proof}
We have already proved in Proposition \ref{product_proposition_sn} \ref{product_proposition_sn_item} that the product $t_{m}^{2p} \smile s_n$
is zero for all $m, n \geq 0$, $p \geq 1$. In fact, the product of $t_m^{2p}$ with almost all other generators annihilates,
except for one case, where we obtain another generator $t_{s}^{2r}$.
\begin{proposition}
For all $m, n \geq 0$, $p, q\geq 1$ we have:
\begin{multicols}{2}
	\begin{enumerate}[(i)]
	    \item $t_m^{2p} \smile t_n^{2q} = 0,$
	    \item $u_m^{2p} \smile t_n^{2q} = t_{n + m}^{2(p + q)}$,
	    \item $t_m^{2p + 1} \smile v_n^{2p + 1} = 0$,
	    \item $t_m^{2p + 1} \smile w_n^{2p + 1} = 0$.
	\end{enumerate}
\end{multicols}
\end{proposition}
\begin{proof}
    \begin{enumerate}[(i)]
        \item The equality follows from:
        \begin{align*}
	        t_m^{2p} &\smile t_n^{2q}\left(1\ox x^{2p + 2q} \ox 1\right)
		        = t_m^{2p}\left(1\ox x^{2p}\ox 1\right)t_n^{2q}\left(1\ox x^{2q}\ox 1\right) = 0,\\
	        t_m^{2p} & \smile t_n^{2q}\left(1\ox y^2x^{2p + 2q - 1} \ox 1\right)
		        = t_m^{2p}\left(1\ox y^2x^{2p - 1}\ox 1\right)t_n^{2q}\left(1\ox x^{2q}\ox 1\right)\\
	        &\qquad + t_m^{2p}\left(1 \ox x^{2p} \ox 1\right)t_n^{2q}\left(1\ox y^2x^{2q - 1}\ox 1\right)\\
	        &\qquad + t_m^{2p}\left(1 \ox x^{2p} \ox 1\right)t_n^{2q}\left(y\ox x^{2q}\ox 1\right) = 0.
        \end{align*}
        \item Clearly $u_m^{2p} \smile t_n^{2q}\left(1 \ox x^{2p + 2q} \ox 1\right) = 0$, while
		\begin{align*}
		    \left(u_m^{2p} \smile t_n^{2q} - \sum_{l =0}^{m - 1}\frac{m!}{l!} d^{2(p + q) - 1}
		    		\left(\alpha_{m - l, l + n}^{2(p + q) - 1}\right)
		        \right)\left(1 \ox y^2x^{2(p + q) - 1} \ox 1\right) = xy^{n + m}.
		\end{align*}
		\item The computations are straightforward:
		\begin{align*}
	        t_m^{2p} & \smile v_n^{2q + 1}\left(1\ox x^{2p + 2q + 1} \ox 1\right)
		        = t_m^{2p}\left(1\ox x^{2p}\ox 1\right)v_n^{2q + 1}\left(1\ox x^{2q + 1}\ox 1\right) = 0,\\
	        t_m^{2p} & \smile v_n^{2q + 1}\left(1\ox y^2x^{2p + 2q} \ox 1\right)
		        = t_m^{2p}\left(1\ox y^2x^{2p - 1}\ox 1\right)v_n^{2q + 1}\left(1\ox x^{2q + 1}\ox 1\right)\\
	        &\qquad + t_m^{2p}\left(1 \ox x^{2p} \ox 1\right)v_n^{2q + 1}\left(1\ox y^2x^{2q}\ox 1\right)\\
	        &\qquad + t_m^{2p}\left(1 \ox x^{2p} \ox 1\right)v_n^{2q + 1}\left(y\ox x^{2q +1}\ox 1\right) = 0.
        \end{align*}
        \item Again, it is clear that:
        \begin{align*}
	        t_m^{2p} & \smile w_n^{2q + 1}\left(1\ox x^{2p + 2q + 1} \ox 1\right)
		        = t_m^{2p}\left(1\ox x^{2p}\ox 1\right)w_n^{2q + 1}\left(1\ox x^{2q + 1}\ox 1\right) = 0, \\
	        t_m^{2p} & \smile w_n^{2q + 1}\left(1\ox y^2x^{2p + 2q} \ox 1\right)
		        = t_m^{2p}\left(1\ox y^2x^{2p - 1}\ox 1\right)w_n^{2q + 1}\left(1\ox x^{2q + 1}\ox 1\right)\\
	        &\qquad + t_m^{2p}\left(1 \ox x^{2p} \ox 1\right)w_n^{2q + 1}\left(1\ox y^2x^{2q}\ox 1\right)\\
	        &\qquad + t_m^{2p}\left(1 \ox x^{2p} \ox 1\right)w_n^{2q + 1}\left(y\ox x^{2q +1}\ox 1\right)
	            = xy^{2m}xy^{2m} = 0.
        \end{align*}
    \end{enumerate}
\end{proof}
Notice that the only case where the product is non zero in the previous proposition is the product with
$u_m^{2p}$, which gives another generator of type $t$. In fact, the elements of type $u_m^{2p}$ act on other
families of generators, as we shall see. Before proving this we need another technical lemma.
\begin{lemma}
    For all $m, n \geq 0$, there is an equality
    \begin{align*}
	    \left(\sum_{i = 0}^{m}\frac{m!}{i!}(yx)^{m - i}y^{2i}\right)
		    \left(\sum_{j = 0}^{n}\frac{n!}{j!}(yx)^{n - j}y^{2j}\right)
			= \sum_{k = 0}^{m + n}\frac{(m + n)!}{k!}(yx)^{m + n - k}y^{2k}.
	\end{align*}
\end{lemma}
\begin{proof}
    Using the commutation rules the product on the left equals
    \begin{align*}
        \sum_{j = 0}^{n - 1}\sum_{l = 0}^{m}
			\left(\sum_{r = 0}^{m - l}\binom{n - 1 - j + r}{r}\right)\frac{m!n!}{l!j!}
			(yx)^{m + n - (j + l)}y^{2(l + j)}
		     + \sum_{i = 0}^{m}\frac{m!}{i!}(yx)^{m - i}y^{2(i + n)}.\\
    \end{align*}
		    It follows from Lemma \ref{lemma_product_binom} that for all $m, n$ and $j$, there is an equality 
    $\sum_{r = 0}^{m - l}\binom{n - 1 - j + r}{r} = \binom{m + n - (j + l)}{n - j}$,
    whence the previous expression is
    \begin{align}
        \sum_{j = 0}^{n}\sum_{l = 0}^{m}
			\binom{m}{l}\binom{n}{j}(n + m - (j + l))!(yx)^{m + n - (j + l)}y^{2(l + j)}\label{product_lemma_u_n_u_m_expression}.
    \end{align}
		The formula we want to prove is symmetric in $m$ and $n$, so we may suppose 
    		that $n \geq m$ and write $m = n + t$, with $t \geq 0$. Setting $k = j + l$
			and $\lambda_k = (yx)^{m + n - k}y^{2k}$ in (\ref{product_lemma_u_n_u_m_expression}), we get
		\begin{align*}
        \sum_{k = 0}^{n}\sum_{j = 0}^{k} &
			\binom{n + t}{k - j}\binom{n}{j}(2n + t -k)!\lambda_k
			 +\sum_{k = n + 1}^{n + t}\sum_{j = 0}^{n}
			\binom{n + t}{k - j}\binom{n}{j}(2n + t -k)!\lambda_k\\
		&\qquad  +\sum_{k = n + t  + 1}^{2n + t}\sum_{j = k - (n + t)}^{n}
			\binom{n + t}{k - j}\binom{n}{j}(2n + t -k)!\lambda_k.
	\end{align*}
	Now, the only things left to finish the proof are the equalities:
	\begin{align*}
		\sum_{j = 0}^{k}\binom{n + t}{k - j}\binom{n}{j} &= \binom{2n + t}{k}, \text{ for all }k, 0\leq k \leq n,\\
		\sum_{j = 0}^{n}\binom{n + t}{k - j}\binom{n}{j} &= \binom{2n + t}{k},
			\text{ for all }k, n + 1 \leq k \leq n + t,\\
		\sum_{j = k - (n + t)}^{n}\binom{n + t}{k - j}\binom{n}{j} &= \binom{2n + t}{k},
			\text{ for all }k, n + t + 1 \leq k \leq 2n + t.
	\end{align*}
	The term on the right counts the numbers of subsets of $k$ elements
	that can be obtained from a set of $2n + t$ elements. This number can be computed splitting our set
	in two subsets: one with $n$ elements and one with $n + t$ elements, and then choosing $j$ from the first subset
	and $k - j$ from the second one for all possible values of $j$.
\end{proof}
We are now ready to prove our claim.
\begin{proposition}
    Given $m \geq 0$ and $p \geq 1$, the element $u_m^{2p} \in \Hy^{2p}(A,A)$ acts, via the cup product, on the families
    of generators $\left\lbrace u_n^{2q}\right\rbrace$, $\left\lbrace v_n^{2q + 1}\right\rbrace$ and
    $\left\lbrace w_n^{2q + 1}\right\rbrace$ as follows. For all $n \geq 0$ and $q \geq 1$,
    \begin{itemize}
        \item $u_m^{2p} \smile u_n^{2q} = u_{m + n}^{2(p + q)}$,
        \item $u_m^{2p} \smile v_n^{2q + 1} = v_{m + n}^{2(p + q) + 1}$,
        \item $u_m^{2p} \smile w_n^{2q + 1} = w_{m + n}^{2(p + q) + 1}$.
    \end{itemize}
\end{proposition}
\begin{proof}
    \begin{enumerate}[(i)]
    \item $ u_m^{2p} \smile u_n^{2q}(1 \ox x^{2p + 2q} \ox 1) = \left(\sum_{i = 0}^{m}\frac{m!}{i!}(yx)^{m - i}y^{2i}\right)
			\left(\sum_{j = 0}^{n}\frac{n!}{j!}(yx)^{n - j}y^{2j}\right)$.
	The previous lemma implies that this expression equals  $\sum_{k = 0}^{m + n}\frac{(m + n)!}{k!}(yx)^{m + n - k}y^{2k}$.
	Also,
	\begin{align*}
	    u_m^{2p} &\smile u_n^{2q}(1 \ox y^2x^{2p + 2q - 1} \ox 1) = 
		    u_m^{2p}\left(1\ox y^2x^{2p - 1}\ox 1\right)u_n^{2q}\left(1\ox x^{2q}\ox 1\right)\\
	    &\qquad + u_m^{2p}\left(1 \ox x^{2p} \ox 1\right)u_n^{2q}\left(1\ox y^2x^{2q - 1}\ox 1\right)\\
	    &\qquad + u_m^{2p}\left(1 \ox x^{2p} \ox 1\right)u_n^{2q}\left(y\ox x^{2q}\ox 1\right)\\
	    &= -\sum_{l = 0}^{n}\frac{n!}{l!}y^{2m + 1}(yx)^{n - l}y^{2l}
		    -\sum_{l = 0}^{m}\frac{m!}{l!}(yx)^{n - l}y^{2l}y^{2n + 1}\\
	    &\qquad + \sum_{l = 0}^{m}\frac{m!}{l!}(yx)^{m - l}y^{2l + 1}\sum_{i = 0}^{n}\frac{n!}{i!}(yx)^{n - i}y^{2i}\\
	    &= -y^{2(m + n) + 1}.
    \end{align*}
    We have thus proved that $u_m^{2p} \smile u_n^{2q} = u_{m + n}^{2(p + q)}$.
    \item The result follows without difficulty, since
    \begin{align*}
	    v_n^{2q + 1} &\smile u_m^{2p}(1 \ox x^{2q + 2p + 1} \ox 1) = 0, \text{ while }\\
	    v_n^{2q + 1} &\smile u_m^{2p}(1 \ox y^2x^{2q + 2p} \ox 1)
		    = v_n^{2q + 1}\left(1\ox y^2x^{2q}\ox 1\right)u_m^{2p}\left(1\ox x^{2p}\ox 1\right)\\
	    &\qquad - v_n^{2q + 1}\left(1 \ox x^{2q + 1}  \ox 1\right)u_m^{2p}\left(1\ox y^2x^{2p - 1}\ox 1\right)\\
	    &\qquad - v_n^{2q + 1}\left(1 \ox x^{2q + 1} \ox 1\right)u_m^{2p}\left(y\ox x^{2p}\ox 1\right)\\
	    &= \left(\sum_{i = 0}^{n}\frac{n!}{i!}(yx)^{n - i}y^{2i}\right)
		    \left(\sum_{j = 0}^{m}\frac{n!}{j!}(yx)^{m - j}y^{2j}\right)\\
	    &= \sum_{i = 0}^{m + n}\frac{(m + n)!}{i!}(yx)^{m + n - i}y^{2i},
    \end{align*}
    and we know that $ u_m^{2p} \smile v_n^{2q + 1} = v_n^{2q + 1} \smile u_m^{2p}$.
    \item Again, it is easy to prove that
    \begin{align*}
	    u_m^{2p} & \smile w_n^{2q + 1}\left(1\ox x^{2p + 2q + 1} \ox 1\right) = \sum_{l = 0}^{m}\frac{m!}{l!}x(yx)^{m - l}y^{2(l + n)}.
    \end{align*}
    The computation for $ u_m^{2p} \smile w_n^{2q + 1} \left(1\ox y^2x^{2p + 2q} \ox 1\right)$ is more involved:
    \begin{align*}
        u_m^{2p} &\smile w_n^{2q + 1}\left(1\ox y^2x^{2p + 2q} \ox 1\right)
            =\sum_{l = 0}^{m}\frac{m!}{l!}x(yx)^{m - l}y^{2(l + n) + 1}\\
	    &\qquad + \sum_{i = 0}^{m - 1}\frac{m!(m - l)}{l!}(yx)^{m + 1 - l}y^{2(l + n)}.
    \end{align*}
    Reducing modulo coboundaries without modifying the other coordinate, we obtain that
    \begin{align*}
	    u_m^{2p} & \smile w_n^{2q + 1}\left(1\ox y^2x^{2p + 2q} \ox 1\right) = xy^{2(m + n) + 1}\\
		    &\qquad - \sum_{l = 0}^{m - 1}\sum_{i = 0}^{l + n}\frac{m!(l + n)!}{l!i!}(yx)^{m + n + 1 -i}y^{2i}
		    +\sum_{l = 0}^{m - 1}\frac{m!(m - l)}{l!}(yx)^{m + 1 - l}y^{2(l + n)}.
    \end{align*}
    Defining $\nu_l \in \Hom_{A^{e}}\left(P_{2p + 2q}A, A\right)$, for $l$ such that $0 \leq l \leq m - 1$,
    by the formulas
    \[
	    \nu_l\left(1 \ox x^{2p + 2q} \ox 1\right) = (yx)^{m -l}y^{2(l + n)} \text{ and }
		    \nu_l\left(1 \ox y^2x^{2p + 2q - 1} \ox 1\right) = 0,
    \]
    like in Proposition \ref{product_proposition_sn} \ref{product_proposition_sn_item2} we get
    \begin{align*}
	    &d^{2p + 2q}(\nu_l)\left(1 \ox x^{2p + 2q + 1} \ox 1\right) = x(yx)^{m - l}y^{2(l + n)},\\
	    &d^{2p + 2q}(\nu_l)\left(1 \ox y^2x^{2p + 2q - 1} \ox 1\right)
		    =(m - l)(yx)^{m  + 1 - l}y^{2(l + n)}\\
	    &\qquad- \sum_{i = 0}^{l + n}\frac{(l + n)!}{i!}(yx)^{m + n  +1 -l}y^{2l}.
    \end{align*}
    Finally,
    \[
	    u_m^{2p}  \smile w_n^{2q + 1} = u_m^{2p}  \smile w_n^{2q + 1}
		    - \sum_{l = 0}^{m - 1}\frac{m!}{l!}d^{2p + 2q}(\nu_l) = w_{n + m}^{2p + 2q + 1}.
    \]
    \end{enumerate}
\end{proof}
The last cup products of generators that we have to describe are those amongst the generators in odd degrees
greater than 1. Most of them are zero except for some cases where we also obtain some of the $t$ generators. We finish the description of the cup product
with the following proposition.
\begin{proposition}
    For all $m, n \geq 0$, $p, q \geq 1$, there are equalities:
    \begin{enumerate}[(i)]
        \item $v_m^{2p + 1} \smile v_n^{2q + 1} = 0$, \label{product_proposition_odd_degrees_1}
        \item $w_m^{2p + 1} \smile w_{n}^{2q + 1} = 0$, \label{product_proposition_odd_degrees_2}
        \item $v_m^{2p + 1} \smile w_n^{2q + 1} = t_{n + m}^{2(p + q + 1)}$. \label{product_proposition_odd_degrees_3}
    \end{enumerate}
\end{proposition}
\begin{proof}
    The proof of \ref{product_proposition_odd_degrees_1} is straightforward and we omit it.
    
    For \ref{product_proposition_odd_degrees_2},  $w_m^{2p + 1} \smile w_{n}^{2q + 1}\left(1 \ox x^{2(p + q + 1)} \ox 1\right)$ is zero,
    but the other coordinate needs a reduction modulo coboundaries. The cochain we will use is, as usual,
    $\alpha_{b, k}^{2(p + q) + 1} \in \Hom_{A^e}\left(P_{2(p + q) + 1}A, A\right)$ defined by
		\begin{align*}
			\alpha_{b, k}^{2(p + q) + 1}\left(1 \ox x^{2(p + q) + 1} \ox 1\right) = 0 \text{ and }
			\alpha_{b, k}^{2(p + q) + 1}\left(1 \ox y^2x^{2(p + q)} \ox 1\right) = -(yx)^{b}y^{2k}.
		\end{align*}
	Subtracting $d^{2(p + q) + 1}\left(\alpha_{b,k}\right)$ for convenient values of $b$ and $k$
    from $w_m^{2p + 1} \smile w_{n}^{2q + 1}$ we obtain that the cup product between both generators is zero.
    
    To prove \ref{product_proposition_odd_degrees_3} notice that
    \begin{align*}
	    v_m^{2p + 1} &\smile w_n^{2q + 1}\left((1\ox x^{2p + 2q + 2} \ox 1\right)
		    = v_m^{2p + 1}\left(1\ox x^{2p + 1}\ox 1\right)w_n^{2q + 1}\left(1\ox x^{2q + 1}\ox 1\right) = 0,
    \end{align*}
    while    
    \begin{align*}
	    v_m^{2p + 1} &\smile w_n^{2q + 1}\left(1\ox y^2x^{2p + 2q  + 1} \ox 1\right)
		    = v_m^{2p + 1}\left(1\ox y^2x^{2p}\ox 1\right)w_n^{2q} \left(1\ox x^{2q + 1}\ox 1\right)\\
	    &\qquad - v_m^{2p + 1}\left(1 \ox x^{2p + 1} \ox 1\right)
		    w_n^{2q + 1}\left(1\ox y^2x^{2q}\ox 1\right)\\
	    &\qquad - v_m^{2p + 1}\left(1 \ox x^{2p + 1} \ox 1\right)
		    w_n^{2q + 1}\left(y\ox x^{2q + 1}\ox 1\right)\\
    	    &= \sum_{i = 0}^{m}\frac{m!}{i!}(yx)^{m - i}y^{2i}xy^{2n}
	    	    =  \sum_{l = 0}^{m}\frac{m!}{l!}x(yx)^{m - l}y^{2(l + n)}.
    \end{align*} 
    Reducing modulo coboundaries, we obtain the result.
\end{proof}
We shall summarize all these propositions in the following table.
\begin{theorem}\label{product_theorem_table}
    The cup product endows $\Hy^{\bullet}(A,A)$ with an associative graded commutative algebra structure.
    The next table describes the products amongst generators
    \begin{table}[h!]
        \begin{center}
        \begin{tabular}{ |c|c|c|c|c|c|c| } 
            \hline
            & $c$ & $s_n$ & $t_n^{2q}$ & $u_n^{2q}$ & $v_n^{2q + 1}$ & $w_n^{2q + 1}$ \\
            \hline
            $c$ & $0$ & $0$ & $ 0$ & $ 0$ & $0$ & $0$ \\
            \hline
            $s_m$ & $0$ & $4(n - m)t^{2}_{n + m + 1}$ & $0$ &
	            $2v^{2q +1}_{n + m + 1} $
		        & $-(2n + 1)t_{n + m}^{2p + 2}$ & $2t_{n + m + 1}^{2p + 2}$ \\
                &&&&$+ (2n + 1)w_{n + m}^{2q + 1}$&&\\
            \hline
            $t_m^{2p}$ & $0$ & $0$ & $0$ & $t_{m + n}^{2p + 2q}$ & $0$ & $0$ \\
            \hline
            $u_m^{2p}$ & $0$ &$2v^{2p +1}_{n + m + 1} $ & $t_{n + m}^{2p + 2q}$ &
	            $u_{m + n}^{2p + 2q}$
		        & $v_{n + m}^{2p + 2q + 1}$ & $w_{n + m}^{2p + 2q +1}$ \\
                &&$+ (2n + 1)w_{n + m}^{2p + 1}$&&&&\\
            \hline
            $v_m^{2p + 1}$ & $0$ &$(2n + 1)t_{n + m}^{2p + 2}$ & $0$ &
	            $v_{m + n}^{2p + 2q + 1}$
		        & $0$ & $t_{m + n}^{2p + 2q + 2}$ \\
            \hline
            $w_m^{2p + 1}$ & $0$ &$-2t_{n + m + 1}^{2p + 2}$ & $0$ &
	            $w_{m + n}^{2p + 2q + 1}$
		        & $-t_{m +n}^{2p + 2q + 2}$ & $0$\\
            \hline
        \end{tabular}
        \label{table:1}
        \end{center}
    \end{table}
\end{theorem}
\begin{corollary}
    The graded algebra $\Hy^{\bullet}(A,A)$ is generated by the set
    $\left\lbrace 1, c, t_0^{2}, t_1^{2}, s_n, u_n^{2}, v_n^3 : n\geq 0\right\rbrace$.
\end{corollary}
Notice that, as we have mentioned before, the cohomology spaces $\Hy^{i}(A,A)$ are periodic
of period $2$ for $i \geq 2$ and that this periodicity comes from multiplication by $u_{0}^{2p}$.
We state this fact more precisely in the next corollary.
\begin{corollary}
For all $p, q \geq 1$, there are isomorphisms
\begin{align*}
    H^{2q}(A,A) \cong H^{2(q + p)}(A,A) \text{ and }  H^{2q + 1}(A,A) \cong H^{2(q + p) + 1}(A,A),
    		\text{ given by }
\end{align*}
\begin{align*}
	&\varphi \in H^{2q}(A,A) \mapsto \varphi \smile u_{0}^{2p} \in H^{2(q + p)}(A,A),\\
	&\varphi \in H^{2q + 1}(A,A) \mapsto \varphi \smile u_{0}^{2p} \in H^{2(q + p) + 1}(A,A).
\end{align*}
\end{corollary}
\section[Derivations]{$\Hy^{1}(A,A)$ as a Lie algebra}
\label{derivations}
\addcontentsline{toc}{chapter}{\nameref{derivations}}
Given an algebra $A$, the Gerstenhaber bracket endows $\Hy^{1}(A,A)$ with a Lie algebra structure.
The aim of this section is to describe it when $A$ is the super Jordan plane. We already have
the necessary comparison maps between the bar resolution and ours.
\begin{proposition}
    The generator $c$ is central.
\end{proposition}
\begin{proof}
    We must prove that $\left[c, s_n\right] = 0$ for all $n \in \NN_{0}$.
    Straightforward computations using Lemma \ref{lemma_product_binom} lead
    us to the equalities:
    \begin{align*}
        c \circ s_n(1 \ox x \ox 1) &= (2n + 1)\sum_{k = 0}^{n -1}\frac{n!}{l!(n - l)} x(yx)^{n - l}y^{2l},\\
        c \circ s_n(1 \ox y \ox 1) &= \sum_{i = 0}^{n}\frac{(n + 1)!}{i!(n -i + 1)} x(yx)^{n - i}y^{2i}
		    + \sum_{i = 0}^{n - 1}\frac{n!}{i!(n -i)} (yx)^{n- i}y^{2i +1},\\
        s_n \circ c(1 \ox x \ox 1) &= 0,\\
        s_n \circ c(1 \ox y \ox 1) &= (2n + 1)xy^{2n},
    \end{align*}
    so that
    \begin{align*}
	    &\left[c, s_n\right](1\ox x \ox 1)
		    = (2n + 1)\sum_{k = 0}^{n -1}\frac{n!}{l!(n - l)} x(yx)^{n - l}y^{2l},\\
	    &\left[c, s_n\right](1\ox y \ox 1)
		    = \sum_{l = 0}^{n}\frac{(n + 1)!}{l!(n -l + 1)} x(yx)^{n - l}y^{2l}
		    + \sum_{l = 0}^{n - 1}\frac{n!}{l!(n -l)} (yx)^{n- l}y^{2l +1} - (2n +  1)xy^{2n}.
    \end{align*}
    We will make use of the elements $\zeta_{n - l - 1, l}$, $\xi_{n - l - 1, l}$, $\rho_n$,
    with $0 \leq l \leq n - 1$, of Proposition \ref{lemma_imd0}. They belong to $\Ima d^0$.
    It is long but not difficult to verify that
    \[
	    \left[c, s_n\right] + \sum_{l = 0}^{n - 1}\frac{n!}{l!(n - l)}\zeta_{n - l - 1, l}
	         - \sum_{l = 0}^{n - 1}\frac{2(n + 1)!}{l!(n - l + 1)(n -l)}\xi_{n - l - 1, l}= 2(n + 1)\rho_n,
    \]
    which proves that $\left[c, s_n\right]$ is zero in $\Hy^{1}(A,A)$.
\end{proof}
The bracket $\left[s_m, s_n\right]$ does not annihilate for $n \neq m$. The following computation is easier and we omit it.
For details, see \cite[p. 95]{R}
\begin{proposition}
    Given $n, m \geq \in \NN_{0}$, we have $\left[s_m, s_n\right] = 2(n - m)s_{n + m}$.
\end{proposition}
Recall that the Virasoro algebra, denoted $\Vir$, is the only central extension of the Witt algebra and it 
is the Lie algebra with basis $\left\lbrace C \right\rbrace \cup \left\lbrace L_n : n \in \ZZ \right\rbrace$,
subject to the relations
\begin{align*}
	&\left[L_m, C\right] = 0,\\
	&\left[L_m, L_n\right] = (n - m)L_{m + n} + \delta_{m, -n}\frac{m^3 - m}{12}C.
\end{align*}
The triangular decomposition of $\Vir$ is the following:
\begin{align*}
    \Vir = Vir^{+} \oplus \mathfrak{h} \oplus \Vir^{-}, \text{ where }
	\Vir^{+} = \bigoplus_{n = 1}^{\infty}\field L_n,\quad
	\mathfrak{h} = \field C \oplus \field L_0,\quad
	\Vir^{-} = \bigoplus_{n = 1}^{\infty}\field L_{-n}.
\end{align*}
The Lie algebra $\Hy^{1}(A,A)$ can be identified to a Lie subalgebra of $\Vir$.
\begin{theorem}\label{derivations_theorem}
    There exists an isomorphism of Lie algebras
    \[
	    \Hy^{1}(A,A) \cong  \Vir^{+} \oplus \mathfrak{h}.
    \]
\end{theorem}
\begin{proof}
    We define the map $\Hy^{1}(A,A) \to \Vir^{+} \oplus \mathfrak{h}$:
    \begin{align*}
        c &\mapsto C,\\
        s_m &\mapsto 2^{m + 1}L_m, \text{ for all } n \in \NN_{0}.
    \end{align*}
    $\Vir$ and $\Hy^{1}(A,A)$ being Lie algebras yields the fact that the map is well-defined.
    It is clearly an isomorphism of vector spaces, and consequently it is an isomorphism of Lie algebras.
    \end{proof}
\begin{remark}
    Any Lie module $M$ over $\Vir^{+} \oplus \mathfrak{h}$ provides, by induction, a Lie module over $\Vir$. Given $n \in \NN$,
    the vector space $\Hy^{n}(A,A)$ is a representation of $\Hy^1(A,A)$. Using the previous theorem, it induces a representation
    of the Virasoro algebra. The next section will be devoted to the description of these representations.
\end{remark}
\section[Representations]{Hochschild cohomology spaces as $\Hy^{1}(A,A)$- Lie modules}
\label{representations}
\addcontentsline{toc}{chapter}{\nameref{representations}}
From now on $\field = \CC$.
In this section we will compute the Gerstenhaber brackets $\left[\Hy^{1}(A,A), \Hy^{n}(A,A)\right]$
for $n > 1$. These brackets describe the structure of $\Hy^{1}(A, A)$-Lie module of each space $\Hy^{n}(A,A)$ and thus,
their induced structure as $\Vir$-modules.

Representation theory of the Virasoro algebra is a very active subject of research since this algebra is important
both in mathematics and in physics, in particular in conformal field theory and string theory. Recall that a weight-module $M$
is a $\Vir$-module such that the Cartan subalgebra $\mathfrak{h}$ of $\Vir$ acts diagonally on $M$. If, additionally,
all the weight spaces of a weight module are finite dimensional, then $M$ is called a Harish-Chandra module.

Intermediate series modules are modules $V_{a,b}$ with basis  $\{v_n\}_{n \in \ZZ}$, $a, b \in \field$
and action given by $L_s \cdot v_n =(n + as + b)v_{n + s}$ and $c \cdot v_n = 0$. The $\Vir$-module $V_{a,b}$
is simple if $a \neq 0, 1$ or when $b \in \CC \setminus \ZZ$. Notice that according to Mathieu \cite{M}, in
case $V_{a, b}$ is not simple, the unique simple quotient of $V_{a, b}$ is called an intermediate series module.

There are two classical families of simple modules: highest (lowest) weight modules and intermediate
series modules, and every simple weight Harish-Chandra module belongs to one of these families \cite{M}. In
\cite{CGZ} all $V_{a,b}$ are called intermediate series modules. Even if the algebra considered is
$W^{+} := \Vir^{+} \oplus \mathfrak{h}$, the corresponding modules $V_{a, b}$ are called intermediate series modules.
We will use this terminology.

We will now start the description of the action of $\Hy^{1}(A,A)$ on $\Hy^{n}(A,A)$.
\begin{proposition}
    For all $n \geq 2$, the element $c \in \Hy^{1}(A,A)$ acts trivially on $\Hy^{n}(A,A)$.
\end{proposition}
\begin{proof}
    We will prove that for all $n \geq 0$ and $p \geq 1$,
    \begin{align*}
        \left[c, t_n^{2p}\right] = \left[c, u_n^{2p}\right] = \left[c, v_n^{2p + 1}\right] = \left[c, w_n^{2p + 1}\right] = 0.    
    \end{align*}
    The element $c$ is the class of the derivation in $A$ sending $x$ to $0$ a $y$ to $x$. We will use Suárez-Álvarez's method \cite{S}
    to compute the brackets, so we need liftings of $c$ as in the following diagram.
    \begin{align*}
        \xymatrix{
	        \cdots \ar[r]^(0.25){d_2}
		        & A \ox \field \{x^2, y^2x\} \ox A \ar[r]^{d_1} \ar[d]^{c_2}
		        & A \ox \field \{x, y\} \ox A \ar[r]^(0.6){d_0} \ar[d]^{c_1} & A\ox A \ar[r]^(0.55){m_A} \ar[d]^{c_0}& A \ar[r] \ar[d]^{c} & 0 \\
	        \cdots \ar[r]^(0.25){d_2}
		        & A \ox \field \{x^2, y^2x\} \ox A \ar[r]^{d_1} \ar[r]
		        & A \ox \field \{x, y\} \ox A \ar[r]^(0.6){d_0} & A\ox A \ar[r]^(0.55){m_A} & A \ar[r] & 0 \\
        }
    \end{align*}
    Let us briefly recall some facts about this approach: given a derivation $\delta$ of an algebra $B$,
    $\delta^{e} = \delta \ox id + id \ox \delta$ is a derivation of $B\ox B^{op}$. For any $B$-bimodule X,
    a $\field$-linear map $f:X \to X$ is a $\delta^{e}$-operator if
    $f((a\ox b) \cdot m) = (a \ox b) \cdot f(m) + \delta^{e}(a\ox b) \cdot m$
    for all $a, b \in B$ and $x \in X$. Here, we need the maps $c_i$ to be $c^e$-operators such
    that the squares are commutative. In this situation, it is enough to determine the values
    of each $c_n$ in elements of type $1 \ox v \ox 1 \in P_nA$. It is not difficult
    to verify that the maps $\{c_n\}_{n \geq 0}$ defined as
    \begin{align*}
	    &c_0(1 \ox 1) = 0,\\
	    &c_1(1 \ox x \ox 1) = 0,& &c_1(1 \ox y  \ox 1) = 1 \ox y \ox 1,\\
	    &c_n(1 \ox x^n \ox 1) = 0,& &c_n(1 \ox y^2x^{n -1} \ox 1) = y \ox x^n \ox 1 + (-1)^{n + 1} \ox x^n \ox y - x \ox x^n \ox 1,
    \end{align*}
    satisfy all the required conditions.

    Let us compute $\left[c, t_n^2\right]$. By Proposition \ref{product_proposition_smcupsn}, we know that $s_0 \smile s_n = 4nt^{2}_{n + 1}$
    and due to the Poisson identity, we get that
    \begin{align*}
        \left[c, t_{n + 1}^2\right] = \frac{1}{4n}\left[c, s_0 \smile s_n\right] = \frac{1}{4n}\left(\left[c, s_0\right]\smile s_n
            +s_0 \smile \left[c, s_n\right]\right) = 0,
    \end{align*}
    for all $n \geq 1$. It only remains to prove that $\left[c, t_0^2\right] = \left[c, t_1^2\right] = 0$. On the one hand,
    \begin{align*}
        \left[c, t_0^2\right]&\left(1 \ox x^2 \ox 1\right) = c\left(t_0^{2}\left(1 \ox x^2 \ox 1\right)\right)
            - t_0^{2}\left(c_2\left(1 \ox x^2 \ox 1\right)\right) = 0,\\
        \left[c, t_0^2\right]&\left(1 \ox y^2x \ox 1\right) = c\left(x\right)
            - t_0^2\left(y \ox x^2 \ox 1 + (-1)^{n + 1} \ox x^2 \ox 1 - x \ox x^2 \ox 1\right) = 0.
    \end{align*}
    On the other hand
    \begin{align*}
        \left[c, t_1^2\right]&\left(1 \ox x^2 \ox 1\right) = c\left(t_1^{2}\left(1 \ox x^2 \ox 1\right)\right)
            - t_1^{2}\left(c_2\left(1 \ox x^2 \ox 1\right)\right) = 0,\\
        \left[c, t_1^2\right]&\left(1 \ox y^2x \ox 1\right) = c\left(xy^2\right) = xyx.
    \end{align*}
    Reducing modulo coboundaries in the same way we did while computing $s_m \smile s_n$, we obtain that 
    $\left[c, t_0^2\right]$ and $\left[c, t_1^2\right]$ are zero in $\Hy^1(A, A)$, and so the bracket
    $\left[c, t_n^2\right]$ is null for all $n$.
    
    We now turn our attention to $\left[c, u_n^2\right]$. Straightforward computations and the usual
    reductions modulo coboundaries show that
    \begin{align*}
        \left[c, u_n^{2}\right]&\left(1\ox x^2 \ox 1\right) = \sum_{r =0}^{n - 1}\frac{n!}{r!(n - r)}x(yx)^{n - 1 - r}y^{2r + 1}
            +\sum_{r =0}^{n - 1}\frac{n!}{r!}\left(\sum_{k = 1}^{n - r}\frac{1}{k}\right)(yx)^{n - r}y^{2r},\\
        \left[c, u_n^{2}\right]&\left(1\ox y^2x \ox 1\right) = -2nxy^{2n} + \sum_{r = 0}^{n - 1}\frac{n!(n - r- 1)}{r!(n-r)}(yx)^{n -r}y^{2r + 1}.
    \end{align*}
    For each $r$, $0 \leq r < n$, let $\gamma_r \in \Hom_{A^{e}}\left(A\ox \field \left\lbrace x, y\right\rbrace \ox A, A\right)$
    be defined by
    \begin{align*}
        \gamma_r(1\ox x \ox 1) = (yx)^{m - 1- r}y^{2r + 1} \quad \text{and} \quad \gamma_r(1 \ox y \ox 1) = 0.
    \end{align*}
    It is not difficult to check the equality $\sum_{r = 0}^{n - 1}\frac{n!}{r!(n - r)}d^1\left(\gamma_r\right) =  \left[c, u_n^{2}\right]$,
    so the bracket is zero in cohomology.
    
    Once we have proved that $\left[c, t_n^2\right] = 0 = \left[c, u_n^2\right]$ for all $n \geq 0$, the identity
    $t_n^{2p} \smile u_0^{2q} = t_{n}^{2(p + q)}$ and the fact that $\left[c, -\right]: \Hy^{\bullet}(A,A) \to \Hy^{\bullet}(A,A)$
    is a derivation with respect to the cup product  allow us to conclude that $\left[c, t_n^{2p}\right] = 0$
    for all $n \geq 0$, $p \geq 1$. Using a similar argument and the equality $u_n^{2p} \smile u_0^{2q} = u_{n}^{2(p + q)}$, we prove that
    $ \left[c, u_n^{2p}\right] = 0$ for all $n \geq 0$, $p \geq 1$.
    
    The equalities  $\left[c, v_n^{2p + 1}\right] = 0 = \left[c, w_n^{2p + 1}\right]$ can be obtained through analogous methods.
\end{proof}

Our next step will be to describe the action of $\Hy^{1}(A,A)$ on the even cohomology spaces. For this, we need to lift
the derivations $s_0, s_1$ and $s_2$ as we have done for $c$. The equalities $\left[s_m, s_n\right] = 2(n - m)s_{m + n}$,
assure that these liftings will be enough. The liftings
of $s_0$, $s_1$ and $s_2$ will be provided by the morphisms of complexes $\alpha_{\bullet}$, $\beta_{\bullet}$ and $\gamma_{\bullet}$,
respectively. They are defined by the following formulas:
\begin{align*}
	\alpha_0&(1 \ox 1) = \beta_0(1 \ox 1) = \gamma_0(1 \ox 1) = 0,\\
	\alpha_1&(1 \ox x \ox 1) = 1\ox x \ox 1 \text{ and } \alpha_1(1 \ox y \ox 1) = 1 \ox y \ox 1,\\
	\alpha_n&(1\ox x^n \ox 1) = n \ox x^n \ox 1, \text{ for all } n \geq 2,\\
	\alpha_n&(1\ox y^2x^{n - 1} \ox 1) = (n + 1) \ox y^2x^{n -1} \ox 1, \text{ for all } n \geq 2.
\end{align*}
The formulas for the higher $\beta$'s and $\gamma$'s are more complicated and because of the periodicity
of the cohomology we will only need the liftings up to degree 3:
\begin{align*}
	\beta_0&(1 \ox 1) = 0,\\
	\beta_1&(1 \ox x \ox 1) = 3(xy \ox y \ox 1 + x \ox y \ox y + 1 \ox x \ox y^2),\\
 	\beta_1&(1 \ox y  \ox 1) = y^2 \ox y\ox 1 + y \ox y \ox y + 1 \ox y \ox y^2,\\
 	\beta_2&(1 \ox x^2 \ox 1) = 3x \ox y^2x \ox 1 + 6 \ox x^2 \ox y^2 + 3 \ox x^2 \ox yx,\\
 	\beta_2&(1 \ox y^2x \ox 1) = 2y^2 \ox y^2x\ox 1 + 5 \ox y^2x \ox y^2 + 2 \ox y^2x \ox yx - 2xy \ox y^2x \ox 1,\\
 	\beta_3&(1 \ox x^3 \ox 1) = 3x \ox y^2x^2 \ox 1 + 9 \ox x^3 \ox y^2 + 6 \ox x^3 \ox yx,\\
 	\beta_3&(1 \ox y^2x^2 \ox 1) = 2y^2 \ox y^2x^2\ox 1 + 8 \ox y^2x^2 \ox y^2 + 5 \ox y^2x^2 \ox yx - 2xy \ox y^2x^2 \ox 1,
\end{align*}
\begin{align*}
	\gamma_0&(1 \ox 1) = 0,\\
	\gamma_1&(1 \ox x \ox 1) = 5(xy^3 \ox y \ox 1 + xy^2 \ox y \ox y + xy \ox y \ox y^2 +x \ox y \ox y^3 + 1 \ox x \ox y^4),\\
 	\gamma_1&(1 \ox y \ox 1) = y^4 \ox y\ox 1 + y^3 \ox y \ox y + y^2 \ox y \ox y^2 + y \ox y \ox y^3 + 1 \ox y \ox y^4,\\
 	\gamma_2&(1 \ox x^2 \ox 1) = 10 \ox x^2 \ox y^4 + 5xy^2 \ox y^2x \ox 1 + 5 x \ox y^2x \ox y^2\\
 		&\qquad +5x \ox y^2x \ox yx + 10 \ox x^2 \ox yxy^2 + 10 \ox x^2 \ox (yx)^2,\\
 	\gamma_2&(1 \ox y^2x \ox 1) = 2y^4 \ox y^2x \ox 1 + 2y^2 \ox y^2x \ox y^2 + 2 y^2\ox y^2x \ox yx\\
 		&\qquad +7 \ox y^2x \ox y^4 + 4 \ox y^2x \ox yxy^2 + 4 \ox y^2x \ox (yx)^2\\
 		&\qquad -4xy^3 \ox y^2x \ox 1 - 2xy \ox y^2x \ox y^2 -2xy \ox y^2x \ox yx,\\
 	\gamma_3&(1 \ox x^3 \ox 1) = 15 \ox x^3 \ox y^4 + 5 xy^2 \ox y^2x^2 \ox 1 + 5 x \ox y^2x^2 \ox y^2\\
 		&\qquad 5x \ox y^2x^2 \ox yx + 20 \ox x^3 \ox yxy^2 + 20 \ox x^3 \ox (yx)^2,\\
 	\gamma_3&(1 \ox y^2x^2 \ox 1) = 2y^4 \ox y^2x^2 \ox 1 + 2y^2 \ox y^2x^2 \ox y^2 \\
 	    &\qquad+ 2 y^2\ox y^2x^2 \ox yx +12 \ox y^2x^2 \ox y^4 + 14 \ox y^2x^2 \ox yxy^2\\
 		&\qquad  + 14 \ox y^2x^2 \ox (yx)^2 -4xy^3 \ox y^2x^2 \ox 1 - 2xy \ox y^2x^2 \ox y^2 \\
 		&\qquad -2xy \ox y^2x^2 \ox yx.
\end{align*}
We will omit the verifications, since they are direct, long and tedious.
\begin{theorem}\label{representations_theorem_evendegree}
    For all $p \geq 1$ and $m, n \geq 0$    
    \begin{enumerate}[(i)]
        \item $\left[c,  t_n^{2p}\right] = 0 = \left[c,  u_n^{2p}\right]$, \label{representations_theorem_2p_1}
        \item $\left[s_m, t_n^{2p}\right] = 2\Big(n - (2p - 1)m - p\Big)t_{n + m}^{2p}$, \label{representations_theorem_2p_2}
        \item $\left[s_m,  u_n^{2p}\right] = 2\Big((n - 2pm - p)u_{n + m}^{2p} + p m (2m + 1)t_{n + m}^{2p}\Big)$. \label{representations_theorem_2p_3}
     \end{enumerate}    
\end{theorem}
Before proceeding to the proof of Theorem \ref{representations_theorem_evendegree} we need a technical lemma. We will not prove
it since the computations needed are not complicated but long.
\begin{lemma}\label{representations_lemma_s_2_u_n_2}
	For all $n \neq 0$,
    \begin{align*}
    	s_2\left(\sum_{l = 0}^{n}\frac{n!}{l!}(yx)^{n - l}y^{2l}\right)
    			= \sum_{l = 0}^{n + 1}\frac{n!\left(2n^3 + n^2 - n + 5\left(l^2 - l\right)\right)}{l!}(yx)^{n + 2 - l}y^{2l} + 2ny^{2n + 4}.
        \end{align*}
    \end{lemma}
\begin{proof}[Proof of Theorem \ref{representations_theorem_evendegree}]
    The equalities in \ref{representations_theorem_2p_1} have already been proved.
    
    \ref{representations_theorem_2p_2} We will prove it by induction
    on $p$. Fix $p = 1$. Using now induction on $m$ we will show that $\left[s_m, t_n^{2}\right] = 2(n - m - 1)t_{n + m}^{2}$.
    Suppose first that $m = 0$:
    \begin{align*}
	    \left[s_0, t_n^2\right]&\left(1\ox x^2 \ox 1\right)= s_0\left(t_n^2\left(1\ox x^2 \ox 1\right)\right)
			    - t_n^2\left(\alpha_2\left(1 \ox x^2	\ox 1\right)\right) = 0,\\
	    \left[s_0, t_n^2\right]&\left(1\ox y^2x \ox 1\right)= s_0\left(xy^{2n}\right) - t_n^2\left(3\ox y^2x \ox 1\right)
	        = 2(n - 1)xy^{2n}.
    \end{align*}
    Thus $\left[s_0, t_n^2\right] = 2(n - 1)t_{n}^{2}$. Supposing now that $m = 1$, we have
    \begin{align*}
	    \left[s_1, t_n^2\right]&\left(1\ox x^2 \ox 1\right)= -t_n^2\Big(3x \ox y^2x \ox 1 + 6 \ox x^2 \ox y^2 + 3 \ox x^2 \ox yx\Big) = 0,\\
	    \left[s_1, t_n^2\right]&\left(1\ox y^2x \ox 1\right)= s_1\left(xy^{2n}\right) - t_n^{2}\Big(2y^2 \ox y^2x\ox 1 \\
		&\qquad+ 5 \ox y^2x \ox y^2 + 2 \ox y^2x \ox yx - 2xy \ox y^2x \ox 1\Big)\\
	    &= 3xy^{2n + 2} + \sum_{i = 0}^{2n - 1}xy^{i}y^3y^{2n - 1 - i} - 2y^2xy^{2n} - 5xy^{2n + 2} - 2xy^{2n + 1}x + 2xyxy^{2n}\\
	    &=2(n - 2)xy^{2(n + 1)} - 2\sum_{j =0}\frac{n!}{j!}x(yx)^{n + 1 - j}y^{2j}.
    \end{align*}
    Reducing modulo coboundaries we can omit the terms of type $x(yx)^{m + 1 - j}y^{2j}$ from the second coordinate, so we obtain that
    $\left[s_1, t_n^2\right] = 2(n - 2)t_{n + 1}^2$. Finally, suppose $m = 2$. Just like before, we get
    \begin{align*}
	    \left[s_2, t_n^2\right]&\left(1\ox x^2 \ox 1\right) = 0,
    \end{align*}
    while
    \begin{align*}
	    \left[s_2, t_n^2\right]&\left(1\ox y^2x \ox 1\right)  = s_2\left(xy^{2n}\right) - t_n^2\Big(
	        2y^4 \ox y^2x \ox 1 + 2y^2 \ox y^2x \ox y^2 + 2 y^2\ox y^2x \ox yx \\
		&\qquad +7 \ox y^2x \ox y^4 + 4 \ox y^2x \ox yxy^2 + 4 \ox y^2x \ox (yx)^2 -4xy^3 \ox y^2x \ox 1\\
		&\qquad - 2xy \ox y^2x \ox y^2 -2xy \ox y^2x \ox yx\Big)\\
	    &= 5xy^{2n +4} + 2nxy^{2n + 4} - 2y^4 xy^{2n} - 2y^2 xy^{2n + 2} - 2 y^2xy^{2n + 1}x -7 xy^{2n + 4}  \\
	    &\qquad - 4 xy^{2n + 1}xy^2 - 4 xy^{2n + 1}xyx  + 4xy^3xy^{2n} + 2xy xy^{2n +2 } +2xyxy^{2n + 1}x.
    \end{align*}
    Again, reducing modulo coboundaries we conclude that $\left[s_2, t_n^2\right] = 2(n - 3)t_{n + 2}^2$.
    
    Given $m \geq 2$, the Jacobi identity reads:
    \begin{align*}
        \left[\left[s_1, s_m\right], t_n^2\right]&= \left[s_1, \left[s_m, t_n^2\right]\right]
		    - \left[s_m, \left[s_1, t_n^2\right]\right]
    \end{align*}
    and using the inductive hypothesis and the equality $\left[s_1, s_m\right] = 2(m - 1)s_{m + 1}$ we get the desired result.
    
    Assume now that $\left[s_m, t_n^{2p}\right] = 2\Big(n - (2p - 1)m - p\Big)t_{n + m}^{2p}$. We will compute:
    \begin{align*}
        \left[s_m, t_n^{2(p + 1)}\right] &= \left[s_m, u_0^2\right] \smile t_n^{2p} + u_0^2 \smile \left[s_m, t_n^{2p}\right]\\
        &=2\left((-2m -1)u_m^2 + m(2m + 1)t_m^2\right) \smile t_n^{2p} + 2 u_0^2 \smile \left(n - (2p - 1)m - p\right)t_{n}^{2p}\\
        &= 2\Big(n - (2p + 1)m - (p + 1\Big)t_{n + m}^{2(p + 1)}.
    \end{align*}
    
    We now turn our attention to the brackets $\left[s_m, u_n^{2p}\right]$. Again, the equality
    $\left[s_1, s_m\right] = 2(m - 1)s_{m + 1}$ reduces the computations to those of the brackets $\left[s_i, u_m^{2p}\right]$
    with $i = 0, 1, 2$. We proceed by induction on $p$. Set $p = 1$, $i = 0$ and $n \geq 0$:
    \begin{align*}
        \left[s_0, u_n^2\right]&\left(1 \ox x^2 \ox 1\right)
            = s_0\left(\sum_{l = 0}^{n}\frac{n!}{l!}(yx)^{n -l}y^{2l}\right) - u_n^2\left(2 \ox x^2 \ox 1\right)\\
        &=2n  \sum_{l = 0}^{n}\frac{n!}{l!}(yx)^{n -l}y^{2l} - 2 \sum_{l = 0}^{n}\frac{n!}{l!}(yx)^{n -l}y^{2l}\\
        &=2(n - 1)\sum_{l = 0}^{n}\frac{n!}{l!}(yx)^{n -l}y^{2l}.
    \end{align*}
    Also, 
    \begin{align*}
        \left[s_0, u_n^2\right]&\left(1 \ox y^2x \ox 1\right) = s_0\left(-y^{2n + 1}\right) - u_n^2\left(3 \ox y^2x \ox 1\right)
            =-2(n - 1)y^{2n + 1}.
    \end{align*}
    
    Calculating $\left[s_i, u_n^2\right]$ for $i = 1, 2$ is considerably longer. The idea is to compute the values
    of these elements of $\Hy^{1}(A,A)$ on $1 \ox x^2 \ox 1$ and $1 \ox y^2x \ox 1$ and then reduce modulo coboundaries.
    We will just indicate which coboundaries can be used, since the procedure does not differ from what we have already
    done in other cases.
    
    For $i = 1$, 
    \begin{align}
        \left[s_1, u^2_{n}\right]&\left(1\ox x^2 \ox 1\right)= \sum_{l = 0}^{n}\frac{n!(2n^2 - 4n - 3)}{l!}(yx)^{n + 1 - l}y^{2l}
            + 2(n - 3) y^{2(n + 1)} + 3xy^{2n + 1},\label{representations_equation1}\\
        \left[s_1, u_n^2\right]&\left(1\ox y^2x \ox 1\right) = -2(n - 3)y^{2n + 3} + 2 \sum_{l = 0}^{n}\frac{(n + 1)!}{i!}x(yx)^{n + 1 - i}y^{2i}.
        \label{representations_equation2}
    \end{align}
    Taking $\gamma = \left(y^{2n + 1}, 0\right) \in \Hom_{A^e}\left(A \ox \field\left\lbrace x,y \right\rbrace \ox A, A\right)$, gives
    \begin{align*}
        d^1(\gamma)\left(1\ox x^2 \ox 1\right) &= xy^{2n + 1} + \sum_{l = 0}^{n}\frac{n!}{l!}(yx)^{n + 1 - l}y^{2l},\\
        d^1(\gamma)\left(1 \ox y^2x \ox 1\right) &= -2xy^{2n + 2} - \sum_{l = 0}^n\frac{(n + 1)!}{l!}x(yx)^{n + 1 - l}y^{2l}.
    \end{align*}
    With this coboundary we delete the term $3xy^{2n + 1}$ from (\ref{representations_equation1}) and after removing the terms of shape
    $x(yx)^{b + 1}y^{2i}$ from the second coordinate, we get that
    \begin{align*}
        \left[s_1, u_n^2\right] = \left[s_1, u_n^2\right] - 3d^1(\gamma) = 2\left((n - 3)u_{n + 1}^2 + 3t_{n + 1}^2\right).
    \end{align*}
	To obtain the value of $s_2 \circ u_n^2$ on $1 \ox x^2 \ox 1$ we use Lemma \ref{representations_lemma_s_2_u_n_2}.
	The other term is $u^2_n\left(\gamma_2(1 \ox x^2 \ox 1)\right)$, and after some direct computations
	the result is
	\begin{align*}
		u^2_n\left(\gamma_2(1 \ox x^2 \ox 1)\right) =
			-10xy^{2n + 3} + 5\sum_{l = 0}^{n + 2}\frac{n!}{l!}\left(n^2 + 3n + 2 + l^2 - l\right)(yx)^{n + 2 - l}y^{2l}.
	\end{align*}		
	Combining both results,
	\begin{align*}
		\left[s_2, u_n^2\right]&\left(1 \ox x^2 \ox 1\right)
			= \sum_{l = 0}^{n + 1}\frac{n!(2n^3 - 4n^2 - 16n -10)}{l!}(yx)^{n + 2 - l}y^{2l} + 2(n - 5)y^{2n +4}\\
		&\qquad+ 10 xy^{2n + 3}.
	\end{align*}
	Moreover
	\begin{align*}
		\left[s_2, u_n^2\right]&\left(1 \ox y^2x \ox 1\right) = s_2\left(-y^{2n + 1}\right) - u_n^2\Big(
			2y^4 \ox y^2x \ox 1 + 2y^2 \ox y^2x \ox y^2 + 2 y^2\ox y^2x \ox yx\\
 		&\qquad +7 \ox y^2x \ox y^4 + 4 \ox y^2x \ox yxy^2 + 4 \ox y^2x \ox (yx)^2\\
 		&\qquad -4xy^3 \ox y^2x \ox 1 - 2xy \ox y^2x \ox y^2 -2xy \ox y^2x \ox yx\Big)\\
 		&= -(2n + 1)y^{2n + 5} + 2y^{2n + 5} +2y^{2n +5} + 2y^{2n + 4}x + 7y^{2n + 5} + 4y^{2n + 2}xy^2\\
 		&\qquad + 4y^{2n + 3}xyx -4xy^{2n + 4} - 2xy^{2n + 4} - 2xy^{2n + 3}x.
	\end{align*}
	The commutation rules and the elimination of the terms $x(yx)^{b + 1}y^{2i}$ lead to the equality
	\begin{align*}
		\left[s_2, u_n^2\right]&\left(1 \ox y^2x \ox 1\right) = -2(n - 5)y^{2n + 5}.
	\end{align*}   
    Setting now
    $\gamma = \left(y^{2n + 3}, 0\right) \in \Hom_{A^e}\left(A \ox \field\left\lbrace x,y \right\rbrace \ox A, A\right)$, we get
    \begin{align*}
	    \left[s_2, u_n^2\right]&\left(1 \ox x^2 \ox 1\right) = \left(\left[s_2, u_n^2\right]
	    		- 10d^1(\gamma)\right)\left(1 \ox x^2 \ox 1\right)\\
	    &= 2(n - 5)\sum_{l = 0}^{n + 2}\frac{(n + 2)!}{l!}(yx)^{n + 2 - l}y^{2l};
    \end{align*}
    after deleting again terms of type $x(yx)^{b + 1}y^{2i}$ of the second coordinate, the result is:
    \begin{align*}
	    \left[s_2, u_n^2\right]&\left(1 \ox y^2x \ox 1\right) = \left(\left[s_2, u_n^2\right]
			    - 10d^1(\gamma)\right)\left(1 \ox y^2x \ox 1\right)\\
	    &=2\left(-(n - 5)y^{2n + 5} + 10xy^{2n + 4}\right)
    \end{align*}
    and the equality $\left[s_2, u_n^2\right] = 2\left((n - 5)u_{n + 2}^2 + 10t_{n + 2}^2\right)$ is thus proved.
    
    The formula $\left[s_m, u_n^2\right] = 2\left((n - 2m - 1)u_{n + m}^2 + m(2m + 1)t_{n + m}^2\right)$ is obtained
    recursively on $m$, as follows. Let $m \geq 3$ and $n \geq 2$, using the Jacobi identity we get that
	\begin{align*}
		\left[\left[s_1, s_m\right], u_n^2\right] &= \left[s_1, \left[s_m, u_n^2\right]\right] - \left[s_m,\left[s_1, u_n^2\right]\right]\\
		&= 2\left[s_1, (n - 2m - 1)u_{n + m}^2 + m(2m + 1)t_{n + m}^2\right] - 2\left[s_m, (n - 3)u_{n + 1}^2 + 3t_{n + 1}^2\right]\\
		&= 4(n - 2m - 1)\left((n + m - 3)u_{n + m  + 1}^2  +3t_{n + m + 1}^2\right)\\
			&\qquad+ 4m(2m + 1)(n + m - 2)t_{n + m + 1}^2\\
			&\qquad - 4(n - 3)\left((n - 2m)u_{n + m + 1}^2 + m (2m + 1)t_{n + m + 1}^2\right)\\
			&\qquad - 12(n - m)t_{n + m + 1}^2\\
		& = 4(m - 1)(n - 2m - 3)u_{n + m + 1}^2 + 4(m - 1)(m + 1)(2m + 3)t_{n + m + 1}^2.
	\end{align*}
	On the other hand, we have $\left[\left[s_1, s_m\right], u_n^2\right] = 2(m - 1)\left[s_{m + 1}, u_n^2\right]$,
	so
	\begin{align*}
		\left[s_{m + 1}, u_n^2\right] = 2(n - 2m - 3)u_{n + m + 1}^2 + 2(m + 1)(2m + 3)t_{n + m + 1}^2.
	\end{align*}
	Now, using once more the facts that for all $\varphi \in \Hy^{1}(A,A)$, the bracket $\left[\varphi, -\right]$ is a derivation
	with respect to the cup product and that $u_0^{2} \smile u_n^{2p} = u_{n}^{2(p + 1)}$ for all $p \geq 1$,
	we are ready to prove by induction on $p$ that
	\begin{align*}
	    \left[s_m,  u_n^{2p}\right] = 2\Big((n - 2pm - p)u_{n + m}^{2p} + p m (2m + 1)t_{n + m}^{2p}\Big).
	\end{align*}
	The case $p = 1$ has already been proved, we suppose now  $p \geq 2$:
	    \begin{align*}
        \left[s_m, u_n^{2p}\right] &= \left[s_m, u_0^2\right] \smile u_n^{2(p - 1)} + u_0^2 \smile \left[s_m, u_n^{2(p - 1)}\right]\\
        &= 2\left((-2m -1)u_m^2 + m(2m + 1)t_m^2\right) \smile u_n^{2(p - 1)}\\
        &\qquad + 2 u_0^2 \smile 2\Big((n - 2(p - 1)m - (p - 1))u_{n + m}^{2(p - 1)} + (p - 1) m (2m + 1)t_{n + m}^{2(p - 1)}\Big)\\
        &= 2\Big((n - 2pm - p)u_{n + m}^{2p} + p m (2m + 1)t_{n + m}^{2p}\Big).
    \end{align*}
\end{proof}
The only brackets missing are those corresponding to the action of $\Hy^{1}(A,A)$ on the odd degree Hochschild cohomology spaces.
\begin{proposition}
    For all $p \geq 1$ and $m, n \geq 0$, $\Hy^{1}(A,A)$ acts on $\Hy^{2p + 1}(A,A)$ as follows:
    \begin{enumerate}[(i)]
        \item $\left[s_m, \nu_n^{2p  + 1}\right] = \Big(n - (2p + 1)m - (1 + p)\Big)v_{n + m}^{2p + 1}
    		    - 4m(2m + 1)\omega_{n + m -1}^{2p + 1}$,
        \item $\left[s_m, \omega_n^{2p + 1}\right] = \Big(n - 2pm - p\Big)\omega_{n + m}^{2p + 1}$.
    \end{enumerate}        
\end{proposition}
\begin{proof}
    For $p = 1$, we will not prove the corresponding formulas, since the verifications follow the same lines of previous
    computations, while for $p > 1$, the periodicity isomorphisms given by the cup product with $u_0^2$, and the fact that
    $\left[s_m, - \right]$ is a derivation with respect to the cup product give the result.
\end{proof}
We change the bases of the Hochschild cohomology spaces to make the description of the representations clearer. From now
on we will use the following notations: $L_m = s_m / 2^{m + 1}$; $\tau_n^{2p} = t_n^{2p} / 2^{n + 1}$;  $\mu_n^{2p} =  u_n^{2p} / 2^{n + 1}$;
 $\nu_n^{2p + 1} = v_n^{2p + 1} / 2^{n + 1}$;  $\omega_n^{2p + 1} = w_n^{2p + 1} / 2^{n + 1}$.

Given $p \geq 1$, let $I_{2p}$ be the ideal generated by $\{\tau_n^{2p}\}_{n \geq 0}$. 
We recall some facts
about intermediate series modules. As we have already mentioned, it is an important class amongst the simple Harish-Chandra modules of $\Vir$ \cite{M}. 
Given $a, b \in \CC$, the intermediate series $\Vir$-module $V_{a, b}$ is generated by a family
$\left\lbrace v_{n}\right\rbrace_{n \in \ZZ}$ and action
\begin{align*}
    C \cdot v_n = 0 \text{ and } L_s \cdot v_n &=(n + as + b)v_{n + s}. 
\end{align*}
Since $W^{+} := \Vir^{+} \oplus \mathfrak{h}$ is a Lie subalgebra of $\Vir$, we have that $V_{a, b}$ is also a $W^{+}$-module
containing a submodule $V_{a, b}^{+}$ generated by $\{v_{n}\}_{n \geq 0}$.
The ideal $I_{2p}$ is isomorphic, as $\Hy^{1}(A,A)$-Lie module, to the intermediate series representation $V_{-(2p - 1), -p}^{+}$.

We know that if $b \notin \ZZ$ or $a \neq 0, 1$,
then $V_{a, b}$ is simple \cite{M}. This is not true for $V_{a, b}^{+}$ because the subspace generated by $\{v_{n}\}_{n \geq k}$ is a submodule
of $V_{a, b}^{+}$, for all $k \in \NN$.
\begin{proposition}
    For $a, b \in \NN$, the $W^{+}$-module $V_{-a,-b}^{+}$ is indecomposable.
\end{proposition}
\begin{proof}
    We claim that given $v \in V^{+}_{a,b}$, $v \neq 0$, and $k_0 \geq 0$, there exists $m \geq 0$ such that $L_m \cdot v$ belongs
    to $\langle v_n \rangle_{n \geq k_0}$ and $L_m \cdot v \neq 0$. For this, we write $v$ as a sum of elements in the basis with non zero
    coefficients: $v = \sum_{i = 1}^{k}\gamma_i v_{\varphi(i)}$, so that
    \begin{align*}
        L_m \cdot v = \sum_{i = 1}^{k}\gamma_i L_m \cdot v_{\varphi(i)} = \sum_{i = 1}^{k}\gamma_i(\varphi(i) - am - b)v_{\varphi(i) + m}.
    \end{align*}
    The integer $m$ can be chosen big enough to ensure that $\varphi(i) - am - b \neq 0$
    and $\varphi(i) + m \geq k_0$ for all $i$, and our assertion is proved.
    
    Let $S$ be a non trivial submodule of $V_{-a, -b}^{+}$. We have to prove that $S$ is not a direct summand of $V_{-a, -b}^{+}$.
    The previous claim shows that it is sufficient to prove that there exists $k_0 \in \NN$ such that $v_n \in S$, for all $n \geq k_0$.
    Choosing a non zero element $v$ of $S$ and writing $v = \sum_{i = 1}^{k}\gamma_i v_{\varphi(i)}$ as before, we aim to show
    that there exists $j_0 \in \NN$ such that $v_{j_0} \in S$. If $k = 1$, then we are done. Suppose now that $k > 1$ and that the result
    holds for elements that are linear combinations of less than $k$ elements of the basis $\{v_n\}_{n \geq 0}$. Suppose first that
    there exists $m \geq 0$ such that $\varphi(i) - am - b = 0$ for some $i$, $1 \leq i \leq k$. In this case
    \begin{align*}
        L_m \cdot v = \sum_{j = 1}^{k}\gamma_j L_m \cdot v_{\varphi(j)} = \sum_{j = 1, j \neq i}^{k}\gamma_j(\varphi(j) - am - b)v_{\varphi(j) + m}
    \end{align*}
    and we use the inductive hypothesis for $L_m \cdot v \in S$. In case $\varphi(i) - am - b \neq 0$ for all $i$, we "shift" $v$, acting
    by $L_s$ for some $s$ so as to return to the previous situation:
    \begin{align*}
        L_m \cdot L_s \cdot v = \sum_{j = 1}^{k}\gamma_j(\varphi(j) - am - b)(\varphi(j) + s - am - b)v_{\varphi(j) + s + m}.
    \end{align*}
    Since $a, b \in \NN$, the equation $\varphi(j) + s =  am + b$ has solutions $j, m , s$.
	Finally, since there exists $j_0 \in \NN$ such that $v_{j_0} \in S$ and $L_m \cdot v_{j_0} = (j_0 - am - b)v_{j_0 + m}$,
	it follows that there exists $k_0 \in \NN$ such that $\langle v_n \rangle_{n \geq k_0}$ is a subspace of $S$.
    The proposition is thus proved.
\end{proof}
Our next result concerns isomorphisms between two of such modules.
\begin{proposition}
    Given $a, b, \alpha, \beta \in \NN$, if $V_{-a, -b}^{+}$ is isomorphic to $V_{-\alpha, -\beta}^{+}$ as $W^{+}-modules$, then $a = \alpha$.
\end{proposition}
\begin{proof}
    We shall denote $\{v_n\}$ and $\{w_n\}$ the generators of $V_{-a, -b}^{+}$ and $V_{-\alpha, -\beta}^{+}$, respectively. Let
    $f: V_{-a, -b}^{+} \to V_{-\alpha, -\beta}^{+}$ be an isomorphism of $W^{+}$-modules.
    Since $b \in \NN$, $v_{b} \in V_{-a, -b}^{+}$ and we write $f(v_b) =  \sum_{j = 1}^{k}\gamma_j w_{\varphi(j)}$.
    Notice that $L_0 \dot v_b = 0$, which implies
    \begin{align*}
        0 = f\left(L_m \cdot v_b \right) = L_0 \cdot f(v_b) = \sum_{j = 1}^{k}\gamma_j L_0 \cdot w_{\varphi(j)}
        = \sum_{j = 1}^{k}\gamma_j(\varphi(j) - \beta)w_{\varphi(j) + m}.
    \end{align*}
    We conclude from this that there exists $j_0$, with $1 \leq j_0 \leq k$ such that $\varphi(j_0) = \beta$ and $\gamma_j = 0$
    for all $j \neq j_0$, so $f(v_b) = \gamma w_{\beta}$. As a consequence,
    \begin{align*}
        f\left(L_m \cdot v_b\right) = L_m \cdot f(v_b) = -\alpha \gamma m w_{\beta + m}
    \end{align*}
    for all $m \geq 0$. On the other hand,
    \begin{align*}
        f\left(L_m \cdot v_b\right) = f\left(-amv_{b + m}\right) = - am f(v_{b + m}).
    \end{align*}
    This means that there exists $d \in \CC^{\ast}$ such that $f(v_{b + m}) = d w_{\beta + m}$. If we choose $m = 1$,
    then $f(v_{b + 1}) = d w_{\beta + 1}$. This implies that, for all $s \geq 0$,
     \begin{align*}
        f(L_s \cdot v_{b + 1}) = L_s\cdot d w_{\beta + 1} = d(1 - \alpha s) w_{\beta + 1 + s}.
    \end{align*}
    Also $f(L_s \cdot v_{b + 1}) =(1 - as)d w_{\beta + 1 + s}$ and we conclude that $\alpha = a$.
\end{proof}
We have already mentioned that the $\Hy^{1}(A,A)$-module $I_{2p}$ is isomorphic to $V^{+}_{-(2p -1), p}$.
Let us now introduce for $p \geq 1$ the family of $\Hy^{1}(A,A)$-modules
$I_{2p + 1} = \left\langle \omega_{n}^{2p + 1}\right\rangle_{n \geq 0}$. Note that $I_{2p + 1}$ is an ideal of $\Hy^{2p + 1}(A,A)$
and it is isomorphic to $V_{-2p, p}^{+}$ as $\Hy^{1}(A,A)$-modules.
The previous proposition allows us to prove that the representations we have obtained are pairwise non isomorphic.
\begin{corollary}
    Given $p, q \geq 1$,
    \begin{itemize}
        \item the $\Hy^{1}(A,A)$-modules $I_{2p}$ and $I_{2q}$ are isomorphic if and only if $p = q$,
        \item The $\Hy^{1}(A, A)$-modules $\Hy^{2p}(A,A)/I_{2p} \cong V_{-2p, p}^{+}$ and $\Hy^{2q}(A,A)/I_{2q} \cong V_{-2q, q}^{+}$
        		are isomorphic if and only if $p = q$,
        \item $I_{2p}$ is never isomorphic to $\Hy^{2q}(A,A)/I_{2q}$, 
    \end{itemize}
    and similarly for $I_{2p + 1}$.
\end{corollary}
\section{Yoneda algebra}
\label{yoneda}
\addcontentsline{toc}{chapter}{\nameref{yoneda}}
The augmentation map $A\to \field$ sending $x$ and $y$ to $0$ makes $\field$ an $A$-bimodule. We will now compute
the Hochschild cohomology of $A$ with coefficients in $\field$, together
with its cup product. Since $\Ext_{A^e}^{\bullet}(A, \field)$ is isomorphic as algebra to
$\Ext_{A}^{\bullet}(\field, \field)$ \cite{St}, we will obtain the Yoneda algebra $\yoneda(A)$ of $A$.
For this, we are going to use again the projective resolution $P_{\bullet}A$ of $A$ as $A$-bimodule, to which we will apply the
functor $\Hom_{A^{e}}(-, \field)$. Since $P_{\bullet}A$ is the minimal resolution of the connected graded algebra $A$,
the differentials
obtained in $\Hom_{A^e}(P_nA, \field)$ are all zero and hence $\Hy^{n}(A, \field)$
is isomorphic to $\Hom_{A^e}(P_nA, \field)$ for all $n \geq 0$ and we obtain the following description
of the cohomology spaces.
\begin{proposition}
\begin{itemize}
    \item The space $\Hy^{0}(A, \field)$ is isomorphic to $\field$ with basis
    $\left\lbrace e: 1\ox 1 \to 1 \right\rbrace$.
    \item $\Hy^{1}(A, \field)$ is $2$-dimensional with basis $\left\lbrace \eta^1, \omega^1 \right\rbrace$
    defined by 
    \begin{align*}
	    \eta^1(1\ox x \ox 1) = 1, \quad \eta^1(1\ox y \ox 1) = 0,\\
	    \omega^1(1\ox x \ox 1) = 0, \quad \omega^1(1\ox y \ox 1) = 1. 
    \end{align*}
    \item For all $n \geq 2$, the space $\Hy^{n}(A, \field)$ is $2$-dimensional with basis $\left\lbrace \eta^n, \omega^n \right\rbrace$
    defined by:
    \begin{align*}
        	\eta^n(1\ox x^n \ox 1) = 1, \quad \eta^n(1\ox y^2x^{n -1 } \ox 1) = 0,\\
	    \omega^n(1\ox x^n \ox 1) = 0, \quad \omega^n(1\ox y^2x^{n -1} \ox 1) = 1. 
    \end{align*}
\end{itemize}
\end{proposition}
The Hilbert series of $\oplus_{i \geq 0}\Hy^{i}(A, \field)$ is $h(t) = 1 + 2\sum_{i \geq 1} t^i = \frac{1 + t}{1 - t}$.
We are interested in the cup product of $\yoneda(A)$. We will use again the comparison morphisms
in order to compute it. Of course $e \in \Hy^{0}(A, \field)$ is the unit of the cohomology algebra.
We will prove that, denoting $\eta^0 := e$, the subalgebra $\oplus_{i \geq 0}\field \eta^i$ is 
a polynomial algebra in one variable
and that the products amongst the $\omega^{i}$'s are zero. More precisely:
\begin{theorem}
    Using the notations of the previous proposition, we describe as follows the products amongst the
    generators of $\Hy^{\bullet}(A, \field)$:
    \begin{itemize}
        \item $\eta^p \smile \eta^q = \eta^{p + q}$, for all  $p, q \geq 1$,
        \item $\omega^p \smile \omega^q = 0$,  for all $p,q \geq 1$,
        \item $\omega^1 \smile \eta^q = 0 = \eta^q \smile \omega^1$, for all $q \geq 1$,
        \item $\omega^p \smile \eta^q = \omega^{p  + q} = (-1)^{q}\eta^q \smile \omega^p$, for all $p \geq 2$ and $q \geq 1$.
    \end{itemize}
    As a consequence, the algebra $\Hy^{\bullet}(A,\field)$ is generated by $\left\lbrace e, \eta^1, \omega^1, \omega^2 \right\rbrace$.
    Moreover, $\Hy^{\bullet}(A,\field)$ is isomorphic to the graded algebra
    \[
        \field \left\langle \eta^1, \omega^1, \omega^2 \right\rangle \Big/
            \left(\left(\omega^1\right)^2, \left(\omega^2\right)^2, \omega^1 \omega^2, \omega^2 \omega^1, \omega^1 \eta^1, \eta^1 \omega^1,
                \omega^2 \eta^1 + \eta^1 \omega^2\right).
    \]
\end{theorem}
\begin{proof}
We shall consider four cases and within each case we will compute the product of two generic elements.

\underline{$1^{\text{st}}$ case}: $\Hy^1(A,\field) \smile \Hy^1(A,\field)$.
Let $\varphi, \psi \in \Hy^{1}(A, \field)$,
\begin{align*} 
	\varphi &\smile \psi\left(1\ox x^2 \ox 1\right) = \varphi g_1 \smile \psi g_1 \left(f_2\left(1\ox x^2\ox 1\right)\right)
	    = \varphi g_1 \smile \psi g_1\left(1 \ox x^{\ox 2} \ox 1\right)\\
	&= \varphi g_1(1\ox x \ox 1) \psi g_1(1 \ox x \ox 1) = \varphi (1\ox x \ox 1) \psi (1 \ox x \ox 1),\\
	\varphi &\smile \psi\left(1\ox y^2x \ox 1\right) = \varphi g_1 \smile \psi g_1 \left(f_2\left(1\ox y^2x \ox 1\right)\right)
	    = \varphi g_1 \smile \psi g_1 \Big(y \ox y \ox x \ox 1 \\
	&\qquad + 1 \ox y \ox yx \ox 1 
		- x \ox y \ox y \ox 1 - 1\ox x \ox y^2 \ox 1 - x \ox y \ox x \ox 1 
		- 1\ox x \ox yx \ox 1\Big) 
	%\\
	\end{align*}
	
	\begin{align*}
	&= y \cdot \varphi(1 \ox y \ox 1) \psi(1 \ox x \ox 1) + \varphi(1 \ox y \ox 1) \psi g_1(1 \ox yx \ox 1)\\
	&\qquad -x \cdot \varphi(1 \ox y \ox 1) \psi(1\ox y \ox 1) - \varphi(1\ox x \ox 1) \psi g_1(1\ox y^2 \ox 1)\\
	&\qquad -x \cdot \varphi(1 \ox y \ox 1) \psi(1\ox x \ox 1) - \varphi(1\ox x \ox 1) \psi g_1(1\ox yx \ox 1)\\
	&= \varphi(1\ox y \ox 1)\left(y\cdot \psi(1\ox x \ox 1) + \psi(1 \ox y \ox 1)\cdot x\right)\\
	&\qquad -\varphi(1\ox x \ox 1)\left(y\cdot \psi(1\ox y \ox 1) + \psi(1 \ox y \ox 1)\cdot y\right)\\
	&\qquad -\varphi(1\ox x \ox 1)\left(y\cdot \psi(1\ox x \ox 1) + \psi(1 \ox y \ox 1)\cdot x\right).
\end{align*}
If we look at the generators, we obtain that $\eta^1 \smile \eta^1 = \eta^2$ and all other products are zero.

\bigskip
\underline{$2^{\text{nd}}$ case}: $\Hy^1(A,\field) \smile \Hy^p(A,\field)$, with $p \geq 2$. Suppose that
$\varphi \in  \Hy^{1}(A, \field)$ and $\psi \in   \Hy^{p}(A, \field)$,
\begin{align*}
    \varphi &\smile \psi\left(1\ox x^{p + 1} \ox 1\right) = \varphi g_1 \smile \psi g_p \left(f_{p + 1}\left(1\ox x^{p + 1}\ox 1\right)\right)
	    = \varphi g_1 \smile \psi g_p\left(1 \ox x^{\ox p  + 1} \ox 1\right)\\
	&= \varphi g_1(1\ox x \ox 1) \psi g_p\left(1 \ox x^{\ox p} \ox 1\right) = \varphi (1\ox x \ox 1) \psi \left(1 \ox x^{p} \ox 1\right),\\
    \varphi & \smile \psi \left(1 \ox y^2x^{p} \ox 1\right) = \varphi g_1 \smile \psi g_p
        \Bigg(1 \ox y \ox yx \ox x^{ \ox p - 1}\ox 1 - 1 \ox x \ox y^{2}\ox x^{\ox p - 1} \ox 1 \\
	&\qquad  - 1 \ox x \ox yx \ox x^{\ox p - 1} \ox 1 + \sum_{i = 0}^{p - 2}(-1)^{i} \ox x^{\ox 2 + i} \ox y^2 \ox x^{\ox p - 2 - i} \ox 1\\
    &\qquad + \sum_{i = 0}^{p - 2}(-1)^{i} \ox x^{\ox 2 + i} \ox yx \ox x^{\ox p - 2 - i} \ox 1\Bigg) \\
	&= \varphi (1 \ox y \ox 1) \psi g_p \left(1\ox yx \ox x^{ \ox p - 1}\ox 1\right)
		- \varphi(1 \ox x \ox 1) \psi g_p \left(1\ox y^{2}\ox x^{\ox p - 1} \ox 1\right)\\
	&\qquad - \varphi( 1 \ox x \ox 1) \psi g_p (1\ox yx \ox x^{\ox p - 1} \ox 1) \\
	&\qquad + \sum_{i = 0}^{p - 2}(-1)^{i} \varphi (1\ox x \ox 1) \psi g_p \left(1\ox  x^{\ox 1 + i} \ox y^2 \ox x^{\ox p - 2 - i} \ox 1\right)\\
	&\qquad + \sum_{i = 0}^{p - 2}(-1)^{i} \varphi (1\ox x \ox 1) \psi g_p \left(1\ox  x^{\ox 1 + i} \ox yx \ox x^{\ox p - 2 - i} \ox 1\right)\\
    &= \varphi (1 \ox y \ox 1) \psi \left(y \ox x^{p} \ox 1\right) - \varphi(1 \ox x \ox 1) \psi \left(1\ox y^{2} x^{p - 1} \ox 1\right)
        - \varphi( 1 \ox x \ox 1) \psi (y\ox x^{p} \ox 1)\\        
    &\qquad  + \sum_{i = 0}^{p - 2}(-1)^{i} \varphi (1\ox x \ox 1) \psi (0) + \sum_{i = 0}^{p - 2}(-1)^{i} \varphi (1\ox x \ox 1) \psi (0)\\
	&= - \varphi(1 \ox x \ox 1) \psi (1\ox y^{2} x^{p - 1} \ox 1).
\end{align*}
We conclude that
$\eta^1 \smile \eta^{p} = \eta^{p + 1}$, $\eta^1 \smile \omega^p = - \omega^{p + 1}$ and $\omega^1 \smile \eta^{p} = \omega^1 \smile \omega^p= 0$.

\bigskip
\underline{$3^{\text{rd}}$ case}: Given $\psi \in  \Hy^{p}(A, \field)$ with $p \geq 2$ and $\varphi \in \Hy^{1}(A, \field)$,
\begin{align*}
    \psi &\smile \varphi\left(1\ox x^{p + 1} \ox 1\right) = \psi\left(1 \ox x^p \ox 1\right) \varphi(1 \ox x \ox 1),\\
    \psi &\smile \varphi\left(1 \ox y^2x^p \ox 1\right) = \psi\left(1\ox y^2 x^{p - 1}\ox 1\right)\phi(1\ox x \ox 1) \\
    &\qquad + (-1)^{p -2}\psi\left(1 \ox x^{p}\ox 1\right)\Big(\psi g_1\left(1 \ox y^2 \ox 1\right) + \psi g_1\left(1 \ox yx \ox 1\right)\Big)\\
    &= \psi\left(1\ox y^2 x^{p - 1}\ox 1\right)\phi(1\ox x \ox 1).
\end{align*}
As a consequence, $ \eta^{p} \smile \eta^1 = \eta^{p + 1}$, $\omega^{p} \smile \eta^1 = \omega^{p + 1}$
and $\eta^p \smile \omega^1 = \omega^p \smile \omega^1 = 0$.

\bigskip
\underline{$4^{\text{th}}$ case}: Given $\varphi \in  \Hy^{q}(A, \field)$ and $\psi \in  \Hy^{p}(A, \field)$ with $p, q \geq 2$,
\begin{align*}
    \varphi &\smile \psi \left(1 \ox x^{q + p} \ox 1\right) = \varphi\left(1 \ox x^{q} \ox 1\right) \psi\left(1 \ox x^{p} \ox 1\right),\\
    \varphi &\smile \psi \left(1 \ox y^2 x^{q + p - 1} \ox 1\right) = \varphi\left(1 \ox y^2x^{q - 1} \ox 1\right) \psi\left(1 \ox x^{p} \ox 1\right)\\
        &\qquad + (-1)^{q -2}\varphi\left(1 \ox x^{p} \ox 1\right)\psi\left(1 \ox y^2x^{q - 1} \ox 1\right).
\end{align*}
Thus, we obtain the equalities $\eta^{q} \smile \eta^{p} = \eta^{p + q}$,
$\omega^p \smile \eta^q = \omega^{p + q} = (-1)^{q}\eta^q \smile \omega^p$, $\omega^q \smile \omega^p = 0$.
Therefore, there is a well defined injective morphism of $\field$-algebras from the algebra in the statement of the theorem to $\yoneda(A)$.
Since they have the same Hilbert series, the last assertion is now clear.
\end{proof}
Notice that $\left(\Hy^{\bullet}(A, \field), \smile \right)$ is not a graded commutative algebra and in particular,
$\left(\Hy^{\bullet}(A,\field), \smile\right)$ is not a subalgebra of  $\left(\Hy^{\bullet}(A,A), \smile\right)$.
It is proved in \cite[Theorem 1.8]{FS} that if $\left(B, \Delta, \varepsilon, \mu, \eta\right)$ is a Hopf $\field$-algebra,
then $\Hy^{\bullet}(B, \field)$ is isomorphic to a subalgebra of $\Hy^{\bullet}(B,B)$ and as a consequence $\Hy^{\bullet}(B, \field)$
is graded commutative. For a Nichols algebra the proof fails since the procedure to obtain a cocycle
in $C^{\bullet}(B,B)$ from a cocycle in $C^{\bullet}(B,\field)$ does not provide an actual cocycle if
\[
    \Delta\left(b b^{'}\right) \neq \sum b_{(1)} b^{'}_{(1)} \ox b_{(2)}b^{'}_{(2)}.
\]

The algebra $A$ is not $N$-Koszul. This can be deduced from the minimal projective resolution of $\field$ as $A$-module.
There is a generalization of the notion of $N$-Koszul algebra: the notion of $\mathcal{K}_2$-algebra, see \cite{CaSh}.
\begin{corollary}
	The algebra $A$ is $\mathcal{K}_2$.
\end{corollary} 
\begin{proof}
	The algebra $\Hy^{\bullet}(A, \field)$ is generated over $\Hy^{0}(A,\field)$ in degrees $1$ and $2$.
\end{proof}
\section[Smash Product]{The Yoneda algebra of $A\#\field\ZZ$}
\label{smash_product}
\addcontentsline{toc}{chapter}{\nameref{smash_product}}
In this section we will describe the Yoneda algebra of the bosonization $A\#\field\ZZ$
of the super Jordan plane. We recall that $A$ is a $\field\ZZ$-module algebra,
where the action of $\field\ZZ$ on $A$ corresponds to the braiding $c$ of $V(-1, 2)$.
Identifying $\field\ZZ$ with the $\field$-algebra of Laurent polynomials $\field\left[t, t^{-1}\right]$.
The action is defined by
\begin{align*}
	t \cdot x = -x;\quad t\cdot y = -y + x.
\end{align*}
Our main tool is Grothendieck's spectral sequence of the derived functors of the composition
of two functors \cite{G} in Stefan's version \cite{St}, which is multiplicative, see for example
\cite{GK}. The spectral sequence is
\begin{align*}
    E_2^{p, q} = \Hy^{p}\left(\ZZ, \Hy^{q}(A, \field)\right) \Rightarrow \Hy^{p + q}(A\#\field\ZZ, \field).
\end{align*}
The action of $\field \ZZ$ on $\Hy^{\bullet}(A, \field)$ is defined at the complex level as follows, given
$\varphi \in \Hom_{A^{e}}(A^{\ox n}, \field)$,
\begin{align*}
    \left(t\cdot \varphi\right)(a_1 \ox \ldots \ox a_n) = t\varphi(t^{-1}a_1 \ox \ldots \ox t^{-1}a_n),
\end{align*}
where $\field \ZZ$ acts on $\field$ via the counit map $\varepsilon:\field\ZZ \to \field$.

It is well known that the complex
\begin{align}
\xymatrix{
	    0 \ar[r] & \field \ZZ \ar[r]^{1 - t \cdot} & \field \ZZ \ar[r]^(.6){\varepsilon} & \field \ar[r] & 0
}
\end{align}
provides a projective $\field\ZZ$-resolution of $\field$, and that for any $\field\ZZ$ module $M$
\begin{align*}
    \Hy^{0}\left(\ZZ, M\right)  \cong M^{\ZZ},\quad
    \Hy^{1}\left(\ZZ, M\right) \cong M/(m - t \cdot m : m \in M) = M_{\ZZ},\quad
    \Hy^{n}\left(\ZZ, M\right) = 0,\quad \text{for } n \geq 2.
\end{align*}
The filtration we are using to compute the spectral sequence is the second filtration,
see \cite[Chapter XV, p. 332]{CE}, so the shape of the second page is
\begin{align*}
\xymatrix{
	0 & 0 & 0 &\\
	\Hy^{1}(\ZZ, \Hy^{0}(A, \field)) \ar@{--}[r] \ar@{--}[u] & \Hy^{1}(\ZZ, \Hy^{1}(A, \field)) \ar@{--}[r] \ar@{--}[u] &
		\Hy^{1}(\ZZ, \Hy^{2}(A, \field)) \ar@{--}[u] \ar@{--}[r] &\\
	\Hy^{0}(\ZZ, \Hy^{0}(A, \field)) \ar@{--}[r] \ar@{--}[u] & \Hy^{0}(\ZZ, \Hy^{1}(A, \field)) \ar@{--}[r] \ar@{--}[u] &
		\Hy^{0}(\ZZ, \Hy^{2}(A, \field)) \ar@{--}[u] \ar@{--}[r] &
}
\end{align*}
We have thus, in this first quadrant spectral sequence, only two non trivial rows,
and the differential $d_2:E_2^{p, q} \to E_2^{p - 1, q + 2}$ can only be non trivial when
$p = 1$. Moreover, there is a five term exact sequence
\begin{align}\label{yoneda_five_term}
\xymatrix{
	0 \ar[r] & E_2^{0, 1} \ar[r] & \Hy^{1}(A\#\field\ZZ, \field) \ar[r] & E_{2}^{1, 0} \ar[r]^{d_2} &
		E_2^{0 , 2} \ar[r] & \Hy^{2}(A\#\field\ZZ, \field).
}
\end{align}
Due to the shape of the spectral sequence, it will collapse at $E_3^{\bullet, \bullet}$.

We need to describe $\Hy^{i}(A, \field)^{\ZZ}$ and  $\Hy^{i}(A, \field)_{\ZZ}$ for all
$i \geq 0$. Our computation of $\Hy^{\bullet}(A, \field)$ has been obtained using the minimal
resolution $P_\bullet A$ of $A$ as $A$-bimodule, so we need to use again the comparison
maps $f_{\bullet}$ and $g_{\bullet}$ of Section \ref{product}, more precisely, given
$\varphi \in \Hom_{A^{e}}(P_{q}A, \field)$, the action of $t$ on $\varphi$ is
$t \cdot \varphi = (t \cdot \varphi g_q)f_q$. It is straightforward to verify that the action is
\begin{align*}
    &t \cdot e = e,\\
    &t \cdot \eta^1 = -\eta^1,\quad t \cdot \omega^1 = -\eta^1 - \omega^1,\\
    &t \cdot \eta^q = (-1)^{q}\eta^q,\quad t \cdot \omega^{q} = (-1)^{q + 1}\omega^{q}.
\end{align*}
As a consequence, we deduce that:
\begin{align*}
    &E_2^{0 ,0} = \langle e \rangle,\quad E_{2}^{1, 0} = \langle e \rangle,\\
    &E_2^{0 ,1} = 0,\quad E_{2}^{1, 1} = 0,\\
    &E_2^{0 ,2k} = \left\langle \eta^{2k} \right\rangle,\quad E_{2}^{1, 2k}
    		= \left\langle \ov{\eta^{2k}} \right\rangle, \text{ for all } k \geq 1,\\
    &E_2^{0 ,2k + 1} = \left\langle \omega^{2k + 1} \right\rangle,\quad E_{2}^{1, 2k + 1}
    		= \left\langle \ov{\omega^{2k + 1}} \right\rangle, \text{ for all } k \geq 1.
\end{align*}
Since there are only two non trivial rows, the multiplicative structure of $\Hy^{\bullet}(\ZZ, \Hy^{\bullet}(A, \field))$
is induced by the cup product in $\yoneda(A)$. More precisely, the product, that we will denote $\cdot$, is just
the restriction to the $\ZZ$-invariants on the $0$-th row, and it is the action of $\Hy^{\bullet}(A, \field)$
on the coinvariants when one of the elements is in the $0$-th row and the other one belongs to the first row, and it is
null - by degree arguments - when both of them belong to the first row. We need the products amongst the generators
of each space. They are as follows:
\begin{itemize}
    \item $e \cdot \varphi = \varphi$ for all $\varphi \in \Hy^{\bullet}(\ZZ, \Hy^{\bullet}(A, \field))$,
    \item $\ov{e} \cdot \varphi = \ov{\varphi}$, for all $\varphi \in \yoneda(A)^{\ZZ}$ and $\ov{e} \cdot \ov{\varphi} = 0$,
        for all $\ov{\varphi} \in \yoneda(A)_{\ZZ}$,
    \item $\eta^{2k} \cdot \eta^{2k'} = \eta^{2(k + k')}$,
    \item $\eta^{2k} \cdot \omega^{2k' + 1} = \omega^{2(k + k') + 1} = \omega^{2k' + 1} \cdot \eta^{2k}$,
    \item $\omega^{2k + 1} \cdot \omega^{2k' + 1} = 0$,
    \item $\eta^{2k} \cdot \ov{\eta^{2k'}} = \ov{\eta^{2(k + k')}} = \ov{\eta^{2k}} \cdot \eta^{2k'}$,
    \item $\eta^{2k} \cdot \ov{\omega^{2k' + 1}}
        = \ov{\eta^{2k}} \cdot \omega^{2k' + 1}
        = \ov{\omega^{2(k + k') + 1}}
        = \ov{\omega^{2k' + 1}} \cdot \eta^{2k}
        = \omega^{2k' + 1} \cdot \ov{\eta^{2k}}$,
    \item $\omega^{2k + 1} \cdot \ov{\omega^{2k' + 1}} = 0 = \ov{\omega^{2k + 1}} \cdot \omega^{2k' + 1}$.
    \item $\ov{\varphi} \cdot \ov{\phi} = 0$ for all $\varphi, \phi \in \yoneda(A)_{\ZZ}$.
\end{itemize}
The exact sequence (\ref{yoneda_five_term}) is
\begin{align*}
\xymatrix{
	0 \ar[r] & 0 \ar[r] & \Hy^{1}(A\#\field\ZZ, \field) \ar[r] & E_{2}^{1, 0} \ar[r]^{d_2} &
		E_2^{0 , 2} \ar[r] & \Hy^{2}(A\#\field\ZZ, \field).
}
\end{align*}
Since both $E_2^{1, 0}$ and $E_2^{0, 2}$ are one dimensional, $d_2$ is either zero or an isomorphism,
and this will depend on whether $\Hy^{1}(A\#\field\ZZ, \field)$ is zero or not. We will compute this last
space directly, using that it is isomorphic to $\Ext_{A\#\field\ZZ}^1(\field, \field)$, that is, the space
of classes of isomorphisms of $1$-extensions of $\field$ by $\field$. As vector spaces, every extension will
be as follows:
\begin{align*}
\xymatrix@R=0.3em{
	0 \ar[r] & \field \ar[r]^{i} & \field \oplus \field \ar[r]^{j} & \field \ar[r] & 0\\
	& 1 \ar@{|->}[r] & (1, 0) & &
}
\end{align*}
The algebra $A\#\field\ZZ$ is graded with generators $x, y$ and $t$, with $x$ and $y$ in degree $1$ and $t$ in
degree $0$, so $x$ and $y$ act trivially on $(1, 0)$ and $(0, 1)$, while the action of $t$ sends $(1, 0) \mapsto (1, \tau)$
for some $\tau$ and $(0,1) \mapsto (0, 1)$. Suppose that $s:\field \to \field \oplus \field$ is a splitting
of $j$, that is $s(1) = (a, 1)$, and $t\cdot(a, 1) = (a, a\tau + 1)$. All the relations are respected and we have thus
nontrivial extensions, indexed by the non-zero elements of $\field$. As a consequence $\Hy^{1}(A\#\field\ZZ, \field) \cong \field$
and the map $\Hy^{1}(A\#\field\ZZ, \field) \to E_{2}^{1, 0}$ is a monomorphism. We already know that $\dim_{\field}E_2^{1, 0} = 1$;
therefore $d_2: E_2^{1, 0} \to E_2^{0, 2}$ is zero. The spectral sequence being multiplicative and the description
of $E_2^{1, j}$, for $j \geq 2$, allow to conclude that $d_2: E_2^{1, j} \to E_2^{0, j + 2}$ is zero for all $j$,
hence $E_2^{\bullet, \bullet} = E_{\infty}^{\bullet, \bullet}$.
\begin{theorem}
	\begin{enumerate}
		\item The Hochschild cohomology of $A\#\field\ZZ$ with coefficients in $\field$ is as follows:
		\begin{align*}
			\Hy^{i}(A\#\field\ZZ, \field) = \left\{\begin{array}{cc}
				\left\langle e \right\rangle, \quad& \text{if } i = 0,\\
				&\\
				\left\langle \ov{e} \right\rangle,  \quad& \text{if } i = 1,\\
				&\\
				\left\langle \eta^{2} \right\rangle,  \quad& \text{if } i = 2,\\
				&\\
				\left\langle \ov{\eta^{2k}}, \omega^{2k + 1} \right\rangle, \quad& \text{if } i = 2k + 1, k > 0\\
				&\\
				\left\langle \eta^{2k} , \ov{\omega^{2k - 1}} \right\rangle, \quad& \text{if } i = 2k, k > 1.
			\end{array}\right.
		\end{align*}
		\item The Yoneda algebra $\yoneda(A\#\field\ZZ) = \oplus_{i \geq 0}\Hy^{i}(A\#\field\ZZ, \field)$ is
		the $\field$-algebra generated by $\ov{e}, \eta^2, \omega^3$, where $\degree(\ov{e}) = 1$, $\degree(\eta^2) = 2$,
		$\degree(\omega^3) = 3$. It is graded commutative, summarizing
		\begin{align*}
			\yoneda(A\#\field\ZZ) \cong \field\left[\eta^2\right] \ox \Lambda(\omega^3, \ov{e}).
		\end{align*}
		In particular, it is finitely generated.
	\end{enumerate}
\end{theorem}
\begin{proof}
	It follows from the description of the spectral sequence and the fact that it is multiplicative.
\end{proof}
\begin{corollary}
	The algebra $A\#\field\ZZ$ is not $\mathcal{K}_2$.
\end{corollary}
\begin{remark}
	$A\#\field\ZZ$ is isomorphic as a graded algebra to $\gr H$, where $H$ is the Hopf algebra described in \cite{AAH2}
	and the filtration considered is the coradical filtration. Adapting the methods of \cite[Theorem 6.3]{MPSW},
	it is possible to prove that the Yoneda algebra of $H$ is also finitely generated.
\end{remark}

\footnotesize
\noindent S.R.:
\\Departamento de Matem\'atica, Facultad de Ciencias Exactas y Naturales, Universidad
de Buenos Aires,\\
Ciudad Universitaria, Pabell\'on I, 1428 Buenos Aires, Argentina; and
\\ IMAS, UBA-CONICET,
Consejo Nacional de Investigaciones Cient\'\i ficas y T\'ecnicas, \\
Ciudad Universitaria, Pabell\'on I, 1428 Buenos Aires, Argentina
\\{\tt sreca@dm.uba.ar}

\medskip

\noindent A.S.:
\\Departamento de Matem\'atica, Facultad de Ciencias Exactas y Naturales, Universidad
de Buenos Aires,\\
Ciudad Universitaria, Pabell\'on I, 1428 Buenos Aires, Argentina; and \\
IMAS, UBA-CONICET,
Consejo Nacional de Investigaciones Cient\'\i ficas y T\'ecnicas, Argentina
\\{\tt asolotar@dm.uba.ar}

\end{document}